\documentclass{amsart}
\usepackage{amssymb, amsmath,bm}
\usepackage{color}
\usepackage{comment}
\usepackage{graphicx} 
\usepackage{multirow, array} 
\usepackage{float} 
\usepackage[english]{babel} 

\def\red{\color{red}} 

\def\black{\color{black}}

\newtheorem{theorem}{Theorem}[section]

\newtheorem{remark}[theorem]{Remark}

\DeclareMathOperator{\Id}{Id}
\DeclareMathOperator{\diag}{diag}

\usepackage[dvipsnames]{xcolor}

\usepackage{pifont}

\newcommand{\Ricci}{\operatorname{Ric}}

\def\span{\operatorname{span}}
\def\ve{\varepsilon}
\def\ip{\langle \cdot,\cdot \rangle}
\DeclareMathOperator{\ad}{ad}

\newcommand{\mfP}{\mathfrak{P}}
\newcommand{\vsep}{0.05in}
\newcommand{\sdR}{\mathbb{R}^3\rtimes\mathbb{R}}
\newcommand{\sdH}{\mathcal{H}^3\rtimes\mathbb{R}}
\newcommand{\algsdR}{\mathbb{R}^3\rtimes \mathbb{R}}
\newcommand{\algsdH}{\mathfrak{h}_3\rtimes\mathbb{R}}
\newcommand{\cARS}{\bm{\mu}}
\newcommand{\HsignRL}{\delta} 
\newcommand{\alg}{\mathfrak{k}}

\usepackage{tikz,caption}
\usetikzlibrary{decorations.pathreplacing}

\begin{document}
\title[Lorentzian algebraic Ricci solitons]{Four-dimensional Lorentzian algebraic Ricci solitons}
\author{E. Garc\'ia-R\'io, R. Rodríguez-Gigirey, R. V\'azquez-Lorenzo}
\address{EGR: Department of Mathematics, CITMAga, University of Santiago de Compostela, 15782 Santiago de Compostela, Spain}
\email{eduardo.garcia.rio@usc.es}
\address{RRG: IES Bergidum Flavium, 24540 Cacabelos, Spain}
\email{rosalia.rodvil@educa.jcyl.es}
\address{RVL: IES de Ribadeo Dionisio Gamallo, 27700 Ribadeo,  Spain}
\email{ravazlor@edu.xunta.gal}
\thanks{Supported by projects  PID2022-138988NB-I00 (AEI/FEDER, Spain) and ED431C 2023/31  (Xunta de Galicia, Spain).}
	
\begin{abstract}
We describe  four-dimensional Lorentzian algebraic Ricci solitons.
In sharp contrast with the Riemannian situation, any connected and simply connected four-dimensional Lie group admits a left-invariant Lorentz metric which is a Ricci soliton.
\end{abstract}
\maketitle

\section{Introduction}

The search for optimal metrics on a given manifold led to different approaches, some of them based on the consideration of different functionals and their critical metrics. More recently, geometric evolution equations have been used to flow a given metric with the purpose of approaching a more well-behaved one.
The Ricci flow $\partial_t g_t=-2\rho(g_t)$ is among the most well-known and widely studied flows. Although the initial metric is expected to flow to a new metric behaving more nicely with respect to the Ricci tensor, this is not always the case. Einstein manifolds, having the best possible Ricci tensor, $\rho=\lambda g$, only evolve by homotheties $g_t=(1-2\lambda t)g$. More generally, self-similar solutions of the flow evolve by homotheties and diffeomorphisms, i.e., $g_t=\sigma(t)\Psi_t^*g$, where $\{\Psi_t\}$ is a one-parameter group of diffeomorphisms and $\sigma(t)$ is a positive real valued function. 
A \emph{Ricci soliton} is a triple  $(M,g,X)$, where $(M,g)$ is a pseudo-Riemannian manifold and $X$ is a vector field on $M$ such that
$$
\mathcal{L}_Xg+\rho=\cARS g
$$ 
for some constant $\cARS$. The soliton is said to be expanding, steady or shrinking if $\cARS<0$, $\cARS=0$, or $\cARS>0$, respectively. Moreover if the vector field is a gradient, then the Ricci soliton equation becomes $\operatorname{Hes}(f)+\rho=\cARS g$ for some potential function $f\in\mathcal{C}^\infty(M)$.
Any self-similar solution of the Ricci flow gives rise to a Ricci soliton and conversely, since the Ricci tensor is homothetically homogeneous, any Ricci soliton corresponds to a self-similar solution of the Ricci flow.
Since Einstein metrics are trivially Ricci solitons, we focus on the \emph{non-trivial} (i.e., non-Einstein) case. We refer to \cite{CaoTran, Wears} and references therein for more information on Ricci solitons (see also  \cite{CaZa,J} for Einstein metrics on four-dimensional Lie groups).

Ricci solitons on Lie groups are closely related to algebraic Ricci solitons introduced in \cite{Lauret}. A left-invariant metric on a Lie group, $(G,\ip)$, is an \emph{algebraic Ricci soliton} if $\mathfrak{D}=\Ricci-\cARS\Id$ is a derivation of the Lie algebra $\mathfrak{g}=\operatorname{Lie}(G)$, where $\operatorname{Ric}$ denotes the Ricci operator. Thus the existence of algebraic Ricci solitons is a property involving both the Lie group structure and the pseudo-Riemannian structure. Any algebraic Ricci soliton gives rise to a Ricci soliton $\mathcal{L}_X\ip+\rho=\cARS \ip$ and thus  to   a self-similar solution of the Ricci flow evolving by automorphisms of the group (see \cite{Lauret2, Wears} for some more general algebraic solitons). 

The three- and four-dimensional cases are particularly interesting. We note that in the Riemannian category homogeneous steady Ricci solitons are flat and, moreover, any non-Einstein four-dimensional homogeneous Ricci soliton is either a gradient one (and hence a product $N^{4-k}(c)\times \mathbb{R}^{k}$, where $N(c)$ is a two- or three-dimensional real space form \cite{PW}), or  homothetic to an expanding algebraic Ricci soliton in the simply connected case \cite{AL, Lauret2} (see also \cite{CaoTran} and references therein).
While three- and four-dimensional Riemannian homogeneous gradient Ricci solitons are locally symmetric, there exist non-symmetric homogeneous Ricci solitons in dimensions three and four. The list of the corresponding algebraic Ricci solitons is relatively small (see \cite{Lauret, Lauret2}).
In sharp contrast with the Riemannian situation, there are plenty of non-symmetric homogeneous Lorentzian Ricci solitons and the purpose of this work is to describe all three- and four-dimensional Lorentzian algebraic Ricci solitons. Results in  \cite{Israel, Maria} coupled with those in this paper show that:
\begin{quote}
\emph{Any connected and simply connected four-dimensional Lie group admits Lorentzian left-invariant Ricci soliton metrics}.
\end{quote}

We refer to \cite{Conti-Rossi, Wears2, Y-D} for previous work on Lorentzian algebraic Ricci solitons with the focus on the nilpotent case, and remark that the main contribution of the present work focuses on the more general solvable situation.
It is worth to emphasize that homogeneous steady Ricci solitons are not necessarily flat in the Lorentzian case, as it occurs for example in the plane wave situation. Moreover, some of these solitons are invariant by cocompact subgroups, thus passing to compact quotients and providing examples of compact steady Lorentzian Ricci solitons which are not Einstein. 

We emphasize that 
	Ricci solitons are unique up to homothetic vector fields. Indeed, two Ricci soliton vector fields $X_1$ and $X_2$ on a given pseudo-Riemannian manifold $(M,g)$ differ on a homothetic vector field $\xi=X_1-X_2$, in which case $(M,g)$ admits different Ricci soliton structures. If a pseudo-Riemannian Lie group admits a homothetic vector field, then the Ricci operator either vanishes or it is nilpotent unless the homothetic vector field is Killing. 
A straightforward calculation shows that all the Ricci solitons constructed in this paper are unique (i.e., they differ from another Ricci soliton in a Killing vector field) unless the underlying Lorentzian structure is a plane wave. 

The paper is structured as follows. 
Plane waves, playing a distinguished role in the analysis, are reviewed in Section~\ref{sse:plane-wave-homog}. We show that although they are Ricci solitons (in the connected and simply connected case), there are Lie groups with a plane wave left-invariant metric which are not algebraic Ricci solitons (cf. Section~\ref{re:producto} and Remark~\ref{re:plane-wave no ARS}).
It is shown in Section~\ref{se:extensions} that three-dimensional algebraic Ricci solitons extend to four-dimensional products which are algebraic Ricci solitons. Therefore we focus on algebraic Ricci solitons which are neither plane waves nor direct products. 
The description of left-invariant Lorentz metrics on three- and four-dimensional Lie groups is revised in Section~\ref{se:3D} and Section~\ref{se:Lorentz Lie groups}, due to some inaccuracies in \cite{CC, CP} where some metrics are missed.
This resulted in the omission of certain three-dimensional algebraic Ricci solitons in \cite{Batat-Onda}, and following \cite{Sandro}, we show in Section~\ref{se:3D} that any connected and simply connected three-dimensional solvable Lie group admits a left-invariant Lorentz metric resulting in an algebraic Ricci soliton.
The description of four-dimensional Lorentz left-invariant metrics will be subsequently used to prove the main results in this paper (Theorem~\ref{th:main R3} and Theorem~\ref{th:main H3}) which describe all non-symmetric, irreducible and non-Einstein algebraic Ricci solitons which are not plane waves. It turns out that they are semi-direct extensions of the Abelian Lie group or the Heisenberg group.
 The proof of Theorem~\ref{th:main R3} and Theorem~\ref{th:main H3} follows from the analysis in Sections~\ref{se:R3}--\ref{se:SLSU}. The purpose of the last section (final remarks) is twofold. First of all to show that any connected and simply connected four-dimensional Lie group admits left-invariant Lorentzian metrics resulting in a Ricci soliton, which is in sharp contrast with the Riemannian situation, and secondly to point out the existence of compact steady Ricci solitons on nilmanifolds and solvmanifolds.

Finally, we mention that the description of algebraic Ricci solitons amounts to solve some systems of polynomial equations on the structure constants of the Lorentzian Lie algebra. In most cases we solve these systems after a straightforward manipulation, but in some cases one can reduce the problem by using Gröbner bases (we refer to \cite{Cox} for more information about the subject).

\section{Summary of results}
\subsection{Four-dimensional homogeneous plane waves}\label{sse:plane-wave-homog}

Let $(M^4,g,\mathcal{U})$ be a four-dimen\-sio\-nal Brink\-mann wave, i.e., a Lorentzian manifold admitting a parallel degenerate line field $\mathcal{U}$. $(M,g,\mathcal{U})$ is said to be a \emph{pp-wave} if the parallel line field is locally generated by a parallel null vector field $U$  and $(M,g)$ is transversally flat, i.e., $R(X,Y)=0$ for all $X,Y\in\mathcal{U}^\perp$. In such  a case there exist local coordinates $(x^+,x^-,x^1,x^2)$ so that 
$$
g=2 dx^+\circ dx^-+ H(x^+,x^1,x^2) dx^+\circ dx^+ +dx^1\circ dx^1+dx^2\circ dx^2\,,
$$
where the degenerate parallel line field is generated by $U=\partial_{x^-}$.
Leistner showed in \cite{Leistner} that a Brinkmann wave $(M,g,\mathcal{U})$ is a $pp$-wave if and only if it is transversally flat and Ricci isotropic, i.e., $g(\operatorname{Ric}X,\operatorname{Ric}X)=0$ for any vector field $X$ on $M$. Furthermore, if  the covariant derivative of the curvature tensor satisfies $\nabla_XR~=~0$ for all $X\in\mathcal{U}^\perp$, then the local coordinates above can be specialized so that $H(x^+,x^1,x^2)=a_{ij}(x^+) x^ix^j$, and we refer to $(M,g,\mathcal{U})$ as a \emph{plane wave}.

The existence of Ricci solitons on plane waves was investigated in \cite{Sandra} where it is shown that any connected and simply connected plane wave is a steady gradient Ricci soliton. Moreover, due to the existence of homothetic vector fields, one also has the existence of expanding and shrinking Ricci solitons on plane waves.

Homogeneous plane waves in dimension four are described in terms of a $2\times 2$ skew-symmetric matrix $F$ and a $2\times 2$ symmetric matrix $A_0$ so that the defining function $H(x^+,x^1,x^2)$ takes the form
$H=\vec{x}^T\, A(x^+) \,\vec{x}$, where $\vec{x}=(x^1,x^2)$ and the matrix $A(x^+)$ is given by (see \cite{BO03})
$$ 
(i)\,\,A(x^+)=e^{x^+\,F}A_0 e^{-x^+\, F},
\quad\text{or}\quad
(ii)\,\,A(x^+)=\frac{1}{(x^+)^2}e^{\log(x^+)F}A_0 e^{-\log(x^+) F}\,.
$$
Moreover, the plane wave metric is Ricci-flat if and only if $A_0$ is trace-free. A straightforward calculation shows that four-dimensional homogeneous plane waves in class (i) have parallel Ricci tensor, while metrics in class (ii) have parallel Ricci tensor if and only if they are Ricci-flat. Furthermore, non-flat locally symmetric plane waves correspond to metrics in class (i) determined by a matrix $A(x^+)$ with constant coefficients.

Four-dimensional plane wave Lie groups reduce to the following families in the non-Einstein case, which correspond to cases (i) and (ii) discussed above depending on whether the Ricci tensor is parallel or not (see \cite{cita_ondas_planas}):
\begin{itemize}
	\item[(a)] A left-invariant metric on $\mathcal{H}^3\rtimes\mathbb{R}$, whose restriction to $\mathcal{H}^3$ is degenerate, given by  
	$[u_1,u_2]=  u_3$,
	$[u_1,u_4]= \kappa_1 u_1 - \kappa_2 u_2 +\kappa_3 u_3$,
	$[u_3,u_4]=(\kappa_1+\kappa_4) u_3$ and
	$[u_2,u_4]=  \kappa_2 u_1 + \kappa_4 u_2+\kappa_5 u_3$,
	where $\{u_i\}$ is a pseudo-orthonormal basis with $\langle u_1,u_1\rangle$ $=$ $\langle u_2,u_2\rangle$ $=$ $\langle u_3,u_4\rangle$ $=1$ and  $4\kappa_1\kappa_4+1\neq 0$. Moreover, the Ricci tensor is parallel if and only if $\kappa_1+\kappa_4=0$, in which case it is of type~(i).
	\smallskip
	\item[(b)] A left-invariant metric on $\mathbb{R}^3\rtimes\mathbb{R}$, whose restriction to $\mathbb{R}^3$ is Lorentzian, determined by 		
	$[u_2,u_4]=(1-\kappa)u_3$ and
	$[u_3,u_4]=(\kappa+1)u_1$, 
	where $\{u_i\}$ is a pseudo-orthonormal basis with  $\langle u_1,u_2\rangle\!=\!\langle u_3,u_3\rangle\!=\!\langle u_4,u_4\rangle\!=\!1$ and $\kappa\neq0$. Moreover, the Ricci tensor is parallel, thus corresponding to a plane wave of type~(i).
	
	\smallskip
	\item[(c)] A left-invariant metric on $\mathbb{R}^3\rtimes\mathbb{R}$, whose restriction to $\mathbb{R}^3$ is degenerate, determined by  	
	$[u_1,u_4]\!=\!\kappa_1 u_1-\kappa_2 u_2 + \kappa_3 u_3$, 
	$[u_2,u_4]\!=\!\kappa_2 u_1+\kappa_4 u_2 + \kappa_5 u_3$ and
	$[u_3,u_4]\!=\!\kappa_6 u_3$, 
	where $\{u_i\}$ denotes a pseudo-orthonormal   basis   with $\langle u_1,u_1\rangle=\langle u_2,u_2\rangle=\langle u_3,u_4\rangle=1$ and  $\kappa_1^2+\kappa_4^2-(\kappa_1+\kappa_4)\kappa_6\neq 0$. Moreover, the Ricci tensor is parallel (and hence of type~(i)) if and only if $\kappa_6=0$.
	
	\smallskip
	\item[(d)] A left-invariant metric on $\widetilde{E}(2)\rtimes\mathbb{R}$, whose restriction to $\widetilde{E}(2)$ is degenerate, determined by
	$[u_1,u_3]= u_2$, 
	$[u_2,u_3]=- u_1$,
	$[u_1,u_4]=\kappa_1 u_1 + \kappa_2 u_2$  and
	$[u_2,u_4]=-\kappa_2 u_1+\kappa_1 u_2$, where $\{ u_i\}$ is a pseudo-orthonormal basis with 
	$\langle u_1,u_1\rangle=\langle u_2,u_2\rangle=\langle u_3,u_4\rangle=1$ and $\kappa_1\neq 0$.
	Moreover, the curvature tensor is parallel, thus corresponding to a plane wave of type~(i).
\end{itemize}

It follows from the proofs of Theorem~\ref{H3-th: g_D0}, Theorem~\ref{R3-th: g_L.III} and Theorem~\ref{R3-th: g_D} that all the plane wave Lie groups in cases (a), (b) and (c) are steady algebraic Ricci solitons. In contrast those in case~(d) are not algebraic  Ricci solitons (Remark~\ref{re:plane-wave no ARS}).
We therefore exclude plane wave Lie groups from the description of Lorentzian algebraic Ricci solitons. 

Furthermore, left-invariant metrics above corresponding to plane waves of type~(i) have parallel Ricci tensor, and thus it follows from \cite{AYan} that they correspond to double extensions of the two-dimensional Abelian Lie algebra.

\subsection{Direct extensions of algebraic Ricci solitons}\label{se:extensions}
Let $(G,\ip_G)$ be an algebraic Ricci soliton and consider the product Lie group $G\times\mathbb{R}^k$ equipped with the left-invariant product metric $\ip=\ip_G\oplus\ip_{\mathbb{R}^k}$. Since $\mathfrak{D}_G=\operatorname{Ric}_G-\cARS\operatorname{Id}_\mathfrak{g}$ is a derivation of the Lie algebra $\mathfrak{g}=\operatorname{Lie}(G)$ and the Ricci operator of $G\times\mathbb{R}^k$ is block-diagonal
$\operatorname{Ric}=\operatorname{Ric}_G\oplus\, 0$, one has that 
$$
\mathfrak{D}=\operatorname{Ric}-\cARS\operatorname{Id}_{\mathfrak{g}\times\mathbb{R}^k}=\left(\operatorname{Ric}_G-\cARS\operatorname{Id}_\mathfrak{g}\right)\oplus (-\cARS\operatorname{Id}_{\mathbb{R}^k})
$$
is a derivation of $\mathfrak{g}\times\mathbb{R}^k=\operatorname{Lie}(G\times\mathbb{R}^k)$ since $\mathbb{R}^k$ is Abelian. As a consequence, the direct extension of any of the non-Einstein three-dimensional algebraic Ricci solitons (see Section~\ref{se:3D}) is a non-Einstein four-dimensional algebraic Ricci soliton. Therefore we exclude those direct products from the subsequent analysis.

\begin{remark}\rm\label{se:51}
Indecomposable Lorentzian symmetric spaces are of constant curvature in the irreducible case, and thus Einstein, or Cahen-Wallach symmetric spaces, which are a special class of plane waves \cite{Cahen}. In the four-dimensional reducible case they are products $N(c)\times\mathbb{R} $, $\Sigma(c)\times\mathbb{R}^2$ or $\Sigma_1(c_1)\times\Sigma_2(c_2)$, where $N(c)$ is a three-dimensional space of constant sectional curvature and $\Sigma(\cdot)$ denotes a surface of constant Gauss curvature. In addition to Einstein products $\Sigma_1(\kappa)\times\Sigma_2(\kappa)$, the products $N(c)\times\mathbb{R} $ and $\Sigma(c)\times\mathbb{R}^2$ are gradient Ricci solitons.

Since two Ricci solitons differ in a homothetic vector field and the products $N(c)\times \mathbb{R}$ and $\Sigma(c)\times\mathbb{R}^2$ do not support any non-Killing homothetic vector field, Ricci solitons are unique (up to Killing vector fields) in these cases. Therefore, we exclude symmetric spaces from the subsequent analysis.
\end{remark}
 
\subsubsection*{Notation}
	Motivated by the results above, as well as those in Section~\ref{sse:plane-wave-homog}, we say that an algebraic Ricci soliton is \emph{strict} if it is non-symmetric, irreducible and neither Einstein nor a plane wave.

\subsection{The non-solvable cases $SU(2)\times\mathbb{R}$ and $\widetilde{SL}(2,\mathbb{R})\times\mathbb{R}$}
We show in Section~\ref{se:SLSU} that the product Lie groups  $SU(2)\times\mathbb{R}$ and $\widetilde{SL}(2,\mathbb{R})\times\mathbb{R}$ are non-Einstein algebraic Ricci solitons only when the left-invariant Lorentz metric gives rise to one of the locally symmetric and locally conformally flat products $\mathbb{S}^3\times\mathbb{R}$ or $\mathbb{S}^3_1\times\mathbb{R}$ (cf. Theorem~\ref{Nonsolvable-ARS-thm}), thus corresponding to the situation in Remark~\ref{se:51}.

\subsection{Extensions of the Euclidean and Poincaré Lie groups}
We show in Section~\ref{se:EE} that any semi-direct extension $G\rtimes\mathbb{R}$ of the Euclidean or Poincaré Lie groups which is a non-Einstein algebraic Ricci soliton is necessarily unimodular and isomorphic to a product $E(1,1)\times\mathbb{R}$ or $\widetilde{E}(2)\times\mathbb{R}$. Hence it is isomorphic to an almost Abelian Lie group corresponding to the discussion in Section~\ref{se:R3} (see \cite{ABDO}).

\subsection{Left-invariant Lorentz metrics on the product $\mathcal{H}^3\times\mathbb{R}$}
\label{re:producto}
\rm
Rahmani showed in \cite{Rahmani} that there exist three non-homothetic classes of left-invariant Lorentzian metrics in the Heisenberg group $\mathcal{H}^3$, which correspond to the situations when the structure operator has Lorentzian, Riemannian or degenerate kernel.   
Kondo and Tamaru have recently shown in \cite{KT} that there exist exactly six non-homothetic classes of left-invariant Lorentzian metrics on $\mathcal{H}^3\times\mathbb{R}$ up to automorphisms, which are described by 	the Lie algebra structures
$$
[e_1,e_2]=-(\alpha e_1-e_4),\quad
[e_2,e_3]=\beta (\alpha e_1-e_4),\quad
[e_2,e_4]=\alpha (\alpha e_1-e_4),
$$ 
where $\{ e_1,e_2,e_3, e_4\}$ is an orthonormal basis of $\mathfrak{h}_3\times\mathbb{R}$ with $e_4$ timelike, and the parameters $(\alpha,\beta)\in\{ (0,0), (1,0), (1,1), (2,0), (2,\sqrt{3}), (2,2)\}$.
A direct calculation shows that they are algebraic Ricci solitons in all cases, corresponding to the soliton constant $\cARS=\frac{3}{2}(\alpha^2-1)(\alpha^2-\beta^2-1)$, as shown in \cite{KT}.

The special case when $(\alpha,\beta)=(1,0)$ determines a flat metric which corresponds to the product metric on $\mathcal{H}^3\times\mathbb{R}$ where the structure operator of the Lorentzian metric on $\mathcal{H}^3$ has degenerate kernel. 	
The cases   $(\alpha,\beta)=(0,0)$ and $(\alpha,\beta)=(2,0)$ reduce to the product metric on $\mathcal{H}^3\times\mathbb{R}$, where the left-invariant metric on $\mathcal{H}^3$ is Lorentzian  and its associated structure operator $L$ has Riemannian and Lorentzian kernel, respectively.
Moreover, the case $(\alpha,\beta)=(2,2)$ corresponds to the product metric on $\mathcal{H}^3\times\mathbb{R}$, where the left-invariant metric in $\mathcal{H}^3$ is Riemannian. Section~\ref{se:extensions} shows that all these cases correspond to four-dimensional algebraic Ricci solitons obtained from those in Section~\ref{se:3D}.

The left-invariant metric in case $(\alpha,\beta)=(1,1)$ is locally symmetric and locally conformally flat. Moreover it is locally isometric to a conformally flat Cahen-Wallach space, where $U=e_1-e_4$ determines a left-invariant parallel null vector field. 	
The Ricci tensor is parallel in the case $(\alpha,\beta)=(2,\sqrt{3})$, although the metric is not locally symmetric. It corresponds to a plane wave associated to the parallel null vector field $U=e_1+\sqrt{3}\,e_3-2e_4$.
In contrast with the previously considered situations, the left-invariant metrics on $\mathcal{H}^3\times\mathbb{R}$ in these last two cases are not the product one.
 
We emphasize that the existence of algebraic Ricci solitons depend both on the pseudo-Riemannian and the Lie group structures. This is clearly observed when comparing the plane wave Lie groups in Section~\ref{sse:plane-wave-homog} corresponding to cases (d) and (b) with $\kappa=1$.
The left-invariant metrics corresponding to both situations are isometric (they are locally conformally flat Cahen-Wallach symmetric spaces). Moreover, while metrics (b) are steady algebraic Ricci solitons, those in class (d) are not.
Indeed, although the locally conformally flat Cahen-Wallach symmetric spaces are Ricci solitons, they admit Lie group presentations of the form $\mathcal{H}^3\times\mathbb{R}$, as in (b), which are algebraic Ricci solitons (corresponding to the metrics determined by $(\alpha,\beta)=(1,1)$), and Lie group presentations on $\operatorname{Aff}(\mathbb{C})$, as in case~(d) of Section~\ref{sse:plane-wave-homog}, which are not algebraic Ricci solitons.

\subsection{Almost Abelian algebraic Ricci solitons}
Connected and simply connected semi-direct extensions of the Abelian Lie group $\mathbb{R}^3$  are isomorphic to one of the Lie groups determined by the following Lie algebras (see \cite{ABDO}):
\begin{enumerate}
	\item The product Lie algebras
	$\mathbb{R}^4$, $\mathfrak{h}_3\times\mathbb{R}$, $\mathfrak{r}_3\times\mathbb{R}$,
	$\mathfrak{r}_{3,\lambda}\times\mathbb{R}$  and
	$\mathfrak{r}'_{3,\lambda}\times\mathbb{R}$.
	\item The irreducible Lie algebras
	$\mathfrak{r}_4$,   $\mathfrak{r}_{4,\lambda}$, $\mathfrak{n}_4$,
	$\mathfrak{r}_{4,\mu,\lambda}$ and
	 $\mathfrak{r}'_{4,\mu,\lambda}$.
\end{enumerate}

In addition to the algebraic Ricci solitons on $\mathfrak{h}_3\times\mathbb{R}$ discussed above and the product ones on $\mathfrak{r}_3\times\mathbb{R}$,
$\mathfrak{r}_{3,\lambda}\times\mathbb{R}$,  and
$\mathfrak{r}'_{3,\lambda}\times\mathbb{R}$ obtained from the three-dimensional algebraic Ricci solitons in Section~\ref{se:3D}, one has the following description.

\begin{theorem}\label{th:main R3}
Let $(\sdR,\ip)$ be a semi-direct extension of the Abelian Lie group equipped with a left-invariant Lorentzian metric. Then it is a strict algebraic Ricci soliton if and only if it is isomorphically homothetic to one of the following:	
	\medskip
	
	\noindent
	{\rm (R)} The restriction of the metric to $\mathbb{R}^3$ is Riemannian and  
	\begin{itemize}
		\item[(i)]  
		$[ e_1, e_4] =  e_2$,
		$[ e_2, e_4] =  e_3$.
		
		\smallskip
		
		\item[(ii)]
		$[e_1,e_4] = e_1 -\gamma_1 e_2$,
		$[e_2,e_4] = \gamma_1 e_1 + e_2$,
		$[e_3,e_4] = \eta_3 e_3$,
		$\eta_3\notin\{0,1\}$, $\gamma_1\geq0$.
		
		\smallskip
		
		\item[(iii)]
		$[e_1,e_4] = e_1$,
		$[e_2,e_4] = \eta_2 e_2$,
		$[e_3,e_4] = \eta_3 e_3$,
		$\eta_2,\eta_3\notin\{0,1\}$, $\eta_2\neq\eta_3$.

	\end{itemize}
	
	\medskip
	
	Here $\{e_i\}$ is an orthonormal basis of the Lie algebra with $e_4$ timelike.

	\medskip
	
	\noindent
	{\rm (L)} The restriction of the metric to $\mathbb{R}^3$ is Lorentzian and  
	\begin{itemize}
		
		\item[(Ia.i)]  
		$[e_1,e_4] = e_1 -\gamma_1 e_2$,
		$[e_2,e_4] = \gamma_1 e_1 + e_2$,
		$[e_3,e_4] = \eta_3 e_3$,
		$\eta_3\notin\{0,1\}$, $\gamma_1\geq 0$.
		
		\smallskip
		
		\item[(Ia.ii)]  
		$[e_1,e_4] = e_1 +\gamma_2 e_3$,
		$[e_2,e_4] = \eta_2 e_2$,
		$[e_3,e_4] = \gamma_2 e_1+e_3$,
		$\eta_2\notin\{0,1\}$, $\gamma_2\geq 0$.
		
		\smallskip
		
		\item[(Ia.iii)]
		$[e_1,e_4] = e_1$,
		$[e_2,e_4] = \eta_2 e_2$,
		$[e_3,e_4] = \eta_3 e_3$,
		$\eta_2,\eta_3\notin\{0,1\}$, $\eta_2\neq \eta_3$.

		\smallskip
		
		\item[(Ib.i)]  
		$[e_1,e_4] = \eta e_1$,
		$[e_2,e_4] = \delta e_2-e_3$,
		$[e_3,e_4] = e_2+\delta e_3$,
		with $\eta\neq 0$ and  $(\eta,\delta)\notin\{
		(\frac{2}{\sqrt{3}},-\frac{1}{\sqrt{3}}),
		(-\frac{2}{\sqrt{3}},\frac{1}{\sqrt{3}})\}$.

		\smallskip
		
		\item[(Ib.ii)]  
		$[e_1,e_4] =  e_3$,
		$[e_2,e_4] = -e_3$,
		$[e_3,e_4] =  e_1+e_2$.

		\smallskip

		\item[(II.i)]  
		$[u_1,u_4] = \eta_1 u_1+ u_2$,
		$[u_2,u_4] = \eta_1 u_2$,
		$[u_3,u_4] = \eta_2 u_3$, $\eta_2\neq 0$.

		\smallskip
		
		\item[(II.ii)]  
		$[u_1,u_4] = -\tfrac{\eta_2}{2}u_1+  u_2$,
		$[u_2,u_4] = \tfrac{4\eta_1+\eta_2}{2} u_2$,
		$[u_3,u_4] = \eta_2 u_3$,
		
		\noindent
		with $\eta_2(\eta_2-\eta_1)(\eta_2+2\eta_1)\neq 0$.
		
		
		\smallskip
		
		\item[(II.iii)]
		$[u_1,u_4] = \eta_1 u_1+  u_2-\gamma_3 u_3$,
		$[u_2,u_4] = \eta_1 u_2$,
		$[u_3,u_4] = \gamma_3 u_2 +\eta_1 u_3$, 
		$\eta_1\gamma_3\neq 0$.

		\smallskip

		\item[(III)]
		$[u_1,u_4]=\eta u_1$,
		$[u_2,u_4]= \eta u_2+u_3$,
		$[u_3,u_4]=u_1+\eta u_3$,
		$ \eta\neq 0$.

	\end{itemize}

	\medskip
	
	Here $\{e_i\}$ is an orthonormal basis of the Lie algebra with $e_3$ timelike, while $\{u_i\}$ is a pseudo-orthonormal basis with $\langle u_1,u_2\rangle=\langle u_3,u_3\rangle=\langle u_4,u_4\rangle=1$. 
\end{theorem}

\begin{remark}
	\rm
	Lorentzian metrics on $\mathbb{R}^3\rtimes\mathbb{R}$ which restrict to degenerate metrics on $\mathbb{R}^3$ and which are algebraic Ricci solitons correspond to plane waves, as shown in the proof of Theorem~\ref{R3-th: g_D}.
\end{remark}

\begin{remark}\rm
	It follows from Remark~\ref{re:riem}, Remark~\ref{re:lor-Ia}, Remark~\ref{re:Ib}, Remark~\ref{re:II}  and Remark~\ref{re:III} that any non-Abelian semi-direct extension 	$\mathfrak{r}_4$,   $\mathfrak{r}_{4,\lambda}$, $\mathfrak{n}_4$,
	$\mathfrak{r}_{4,\mu,\lambda}$ and $\mathfrak{r}'_{4,\mu,\lambda}$ of the Abelian Lie group $\mathbb{R}^3$ admits a left-invariant Lorentzian metric resulting in a non-Einstein algebraic Ricci soliton, with the exception of $\mathfrak{r}_{4,1,1}$ where any left-invariant Riemannian or Lorentzian metric is of constant sectional curvature and hence Einstein (see \cite{Milnor, Nomizu}).  
	The existence of non-Einstein algebraic Ricci solitons on the product Lie algebras $\mathfrak{h}_3\times\mathbb{R}$, $\mathfrak{r}_3\times\mathbb{R}$,
	$\mathfrak{r}_{3,\lambda}\times\mathbb{R}$ and
	$\mathfrak{r}'_{3,\lambda}\times\mathbb{R}$ follows from the analysis in Section~\ref{se:3D} and Section~\ref{se:extensions}.
	In addition to the product metrics, we emphasize that the product Lie groups corresponding to $\mathfrak{r}_{3,\lambda}\times\mathbb{R}$ with $\lambda\neq 0,\frac{1}{2}$ (Remark~\ref{re:lor-Ia}--(ii)), and $\mathfrak{r}_{3,-2}\times\mathbb{R}$ (Remark~\ref{re:II}--(ii)) admit non-product metrics resulting on algebraic Ricci solitons.
\end{remark}

\subsection{Algebraic Ricci solitons on semi-direct extensions of the Heisenberg group}
Connected and simply connected semi-direct extensions of the Heisenberg Lie group $\mathcal{H}^3$  are isomorphic to one of the Lie groups determined by the following Lie algebras:
\begin{enumerate}
	\item The product Lie algebra $\mathfrak{h}_3\times\mathbb{R}$.
	\item The irreducible Lie algebras $\mathfrak{d}_4$, $\mathfrak{d}_{4,\lambda}$, $\mathfrak{d}'_{4,\lambda}$, $\mathfrak{n}_4$  and $\mathfrak{h}_4$.
\end{enumerate}

Since the Lie algebras $\mathfrak{n}_4$ and $\mathfrak{h}_3\times\mathbb{R}$ are already covered by the analysis of the almost Abelian case and Section~\ref{re:producto}, in  Theorem~\ref{th:main H3}  we focus on semi-direct extensions of the Heisenberg group which are not almost Abelian.

\begin{theorem}\label{th:main H3}
Let $(\sdH,\ip)$ be a semi-direct extension of the Heisenberg Lie group equipped with a left-invariant Lorentzian metric. Then it is a strict algebraic Ricci soliton not covered by Theorem~\ref{th:main R3}
if and only if  
the  restriction of the metric to $\mathfrak{h}_3$ is Lorentzian
and $(\sdH,\ip)$ is isomorphically homothetic to one of the following:
	\medskip
	\begin{itemize}	
		\item[(Ia-)]  
		$[e_1,e_3] =-e_2$,
		$[e_1,e_4] = \gamma_1 e_1 + \gamma_3 e_3$,
		$[e_2,e_4] = \gamma_4 e_2$,
		
		\smallskip
		
		\noindent
		$[e_3,e_4] = -\gamma_3 e_1-(\gamma_1-\gamma_4)e_3$,
		
		\smallskip
		
		\noindent
		where $\gamma_3$ is the only positive solution of 
		$4\gamma_3^2 = 4 (\gamma_1^2  + \gamma_4^2 - \gamma_1 \gamma_4) + 3$.
		
		\medskip
		
		\item[(II)] 
		$[u_1,u_3] = - u_2$,
		$[u_1,u_4] = \gamma_1 u_1 + \gamma_2 u_2 + 3 \gamma_6 u_3$,
		$[u_2,u_4] = -\gamma_1 u_2$,
		
		\smallskip
		
		\noindent
		$[u_3,u_4] = \gamma_6 u_2 - 2\gamma_1 u_3$, 
		$\gamma_1\neq 0$.
	\end{itemize}	
	
	\medskip
	
	\noindent
	Here $\{ e_i\}$ is an orthonormal basis of the Lie algebra  with $e_3$ timelike, while 	  $\{u_i\}$ is a pseudo-orthonormal basis with $\langle u_1,u_2\rangle=\langle u_3,u_3\rangle=\langle u_4,u_4\rangle=1$.
\end{theorem}

\begin{remark}\rm
	The Lie algebra underlying left-invariant metrics in case~(Ia-) is  $\mathfrak{d}'_{4,\lambda}$ with $|\lambda|<\frac{1}{\sqrt{3}}$ (cf. Remark~\ref{re:Heisenberg-(i)}),
	while for metrics in case~(II) of Theorem~\ref{th:main H3} the underlying Lie algebra is $\mathfrak{d}_{4,2}$ (cf. Remark~\ref{re:Heisenberg-(ii)}).
\end{remark}

%

\subsection{Critical metrics for quadratic curvature functionals}
	\label{grafico 1}\rm
	Quadratic curvature functionals are homothetically invariant in dimension four. Moreover, any 
	quadratic curvature functional in dimensions three and four is equivalent to 
	$$
	\mathcal{S}:g\mapsto\mathcal{S}(g)=\int_M\tau^2\operatorname{dvol}_g,\quad\text{or}\quad
	\mathcal{F}[t]:g\mapsto\mathcal{F}[t](g)=\int_M \{ \|\rho\|^2+t\tau^2\}\operatorname{dvol}_g,
	$$
	where $t\in\mathbb{R}$ (see \cite{CM}).
Considering variations with constant volume in dimension three, or just  arbitrary variations in dimension four, the corresponding critical metrics are determined in the constant scalar curvature case by 
$$
\tau(\rho-\tfrac{1}{n}\tau g)=0,\quad\text{or}\quad
-\Delta\rho+\tfrac{2}{n}(\|\rho\|^2+t\tau^2)g-2R[\rho]-2t\tau\rho=0 ,
$$
respectively,
where $R[\rho]_{ij}=R_{ikj\ell}\rho^{k\ell}$.
Einstein metrics are critical for all quadratic curvature functionals in dimensions three and four. Moreover, non-Einstein metrics with constant scalar curvature are $\mathcal{S}$-critical if and only if $\tau=0$. Furthermore, a metric with constant scalar curvature is critical for two-distinct quadratic curvature functionals if and only if it is critical for all quadratic curvature functionals. Note that Lorentzian products $\Sigma_1(c)\times\Sigma_2(-c)$ of two surfaces with opposite constant Gauss curvature are critical for all quadratic curvature functionals, as well as plane wave metrics, and observe that the energy of the functionals $\mathcal{F}[t]$, given by $\mathcal{E}_t=\|\rho\|^2+t\tau^2$, vanishes identically for plane waves, but it is always non-zero in the products $\Sigma_1(c)\times\Sigma_2(-c)$.

A special feature of connected and simply connected three- and four-dimensional homogeneous Riemannian Ricci solitons $(M,g,X)$ is that they are critical for some quadratic curvature functional with zero energy (i.e., $\|\rho\|^2+t\tau^2=0$) \cite{Sandro2}. If $\cARS\in\mathbb{R}$ denotes the soliton constant ($\mathcal{L}_Xg+\rho=\cARS g$), then $\cARS\tau=\|\rho\|^2$, and the quadratic curvature functional $\mathcal{F}[t]$ with zero energy is determined by $t= -\cARS\tau^{-1}$, provided that $\tau\neq 0$ (see \cite{CaoTran, catino}).	

		A case by case analysis shows that four-dimensional Lorentzian algebraic Ricci soliton metrics are $\mathcal{S}$-critical or $\mathcal{F}[t]$-critical with zero energy. Since quadratic curvature functionals are homothetically invariant in dimension four, we represent in the following diagram the possible values of $t$ for which the left-invariant metrics in Theorem~\ref{th:main R3} and Theorem~\ref{th:main H3} are $\mathcal{F}[t]$-critical with zero energy. We omit the cases of algebraic Ricci solitons with vanishing scalar curvature which are $\mathcal{S}$-critical, which correspond to metrics (L.Ib.i) in Theorem~\ref{th:main R3} with $\eta^2+3\delta^2+2\eta\delta-1=0$, and metrics (L.Ia-) in Theorem~\ref{th:main H3} for $1-2\gamma_4^2=0$. Moreover, in both situations one has that $\|\rho\|=\tau=0$.
		
		We focus on the irreducible case, since the reducible situation follows from the corresponding three-dimensional one in Figure~\ref{fi: figura 2}. This is due to the fact that any three-dimensional homogeneous $\mathcal{F}[t]$-critical manifold with zero energy $(N,g_N)$ gives rise to a four-dimensional $\mathcal{F}[t]$-critical product metric on $M=N\times\mathbb{R}$ with zero energy (see \cite{Sandro2}).
		
		Each row in the figure below indicates the range of $t=-\cARS\tau^{-1}$ for the corresponding homothetic class of irreducible non-Einstein algebraic Ricci solitons with non-zero scalar curvature which are
		not locally symmetric. The arrow on the left (resp. on the right) indicates that the interval extends to $-\infty$ (resp. to $+\infty$). An empty dot means that the point is not included in the interval, whereas a filled dot indicates that the point belongs to the range of $t$. Furthermore algebraic Ricci solitons corresponding to $\mathcal{F}[t]$-critical metrics colored in red are shrinking, those in green color are steady and those in blue are expanding.
	
\begin{center} 
	\def\ColorExp{blue}
	\def\ColorShr{red}
	\def\ColorSte{teal}
	
	\def\RTres{\mathbb{R}^3}
	\def\HTres{\mathcal{H}^3}
	
	\def\Xmenos{-5.6}
	\def\Xmenost{\Xmenos+.5+1.2+.8}
	\def\Xmenosinf{\Xmenos+.7+.8+.3}
	\def\Xmas{5.6}
	\def\Xmast{\Xmas-1.2}
	\def\Xmasinf{\Xmas-.3}
	\def\Y{0}
	\def\YLineaPunteada{6.45}
	
	\def\YSepDosMetricas{0.1}
	
	\def\XCorreccion{.08}
	\def\XCero{\Xmast-.2} 
	\def\XUnCuarto{\XCero-.94+2*\XCorreccion}
	\def\XUnTercio{\XUnCuarto-0.31-2*\XCorreccion}
	\def\XUnMedio{\XUnTercio-0.63-\XCorreccion}
	\def\XTresCuartos{\XUnMedio-.94-\XCorreccion}
	\def\XUno{\XTresCuartos-.94-\XCorreccion}
	\def\XTresMedios{\XUno-1.88-\XCorreccion}

	
	\medskip

	\begin{tikzpicture}
	\draw[thick]   (\Xmenos-1,\Y) -- (\Xmenos+.35+.8+.3,\Y); 
	\draw (\Xmenos-1.1,0.2) node[anchor=west] {\footnotesize  $\cdots\!\rtimes\!\mathbb{R}$ };

	\draw[thick,-latex]   (\Xmenos+.48+.8+.3,\Y) -- (\Xmas,\Y) node[right] {$t$}; 
	
	\foreach \i/\n in 
	{\XTresMedios/$-\frac{3}{2}$, 
		\XUno/$-1$, 
		\XTresCuartos/$\frac{-3}{4}$, 
		\XUnMedio/$\frac{-1}{2}$,
		\XUnTercio/$\frac{-1}{3}$,  
		\XUnCuarto/$\frac{-1}{4}$,
		\XCero/$0$ 
	}{\draw (\i,\Y-.2)--(\i,\Y+.2);}

	\foreach \i/\n in 
	{\XUno/{\footnotesize $-1$}, 
		\XCero/{\footnotesize $0$}
	}{\draw  (\i,\Y+.2) node[above] {\footnotesize \n};}

	\foreach \i/\n in 
	{\XTresMedios/{\footnotesize $\frac{-3}{2}$}, 
		\XTresCuartos/{\footnotesize $\frac{-3}{4}$}, 
		\XUnMedio/{\footnotesize $\frac{-1}{2}$},  
		\XUnTercio/{\footnotesize$\frac{-1}{3}$}, 
		\XUnCuarto/{\footnotesize $\frac{-1}{4}$} 
	}{\draw  (\i,\Y+.12) node[above] {\footnotesize \n};}

	\foreach \i/\n in 
	{\XTresMedios/$\frac{-3}{2}$, 
		\XUno/$-1$, 
		\XTresCuartos/$\frac{-3}{4}$, 
		\XUnMedio/$\frac{-1}{2}$, 
		\XUnTercio/$\frac{-1}{3}$, 
		\XUnCuarto/$\frac{-1}{4}$,
		\XCero/$0$ 
	}{
		\draw[dashed] (\i,\Y-.25)--(\i,\Y-\YLineaPunteada);
	}
	
	\def\YPrimeraFamilia{0.3}
	\def\YSepEntreFamilias{0.42}
	\def\CorreccionYTexto{-.02}
	\def\XGrupo{\Xmenos-1.1}

	\def\YGrupo{\Y-\YPrimeraFamilia}
	\draw (\XGrupo,\YGrupo+\CorreccionYTexto) node[anchor=west] {\footnotesize $\color{black} \mathbb{R}^3\!\rtimes\!\mathbb{R}$ \,\,(Thm\,\ref{th:main R3})};
	
	\def\YGrupo{\Y-\YPrimeraFamilia-\YSepEntreFamilias}
	\draw (\XGrupo,\YGrupo+\CorreccionYTexto) node[anchor=west] {\footnotesize 
		\color{black} \,\,\,\,(R.i)};
	\filldraw [color=\ColorShr,fill=\ColorShr] (\XTresMedios,\YGrupo) circle (2pt); 
	
	\def\YGrupo{\Y-\YPrimeraFamilia-2*\YSepEntreFamilias}
	\draw (\XGrupo,\YGrupo+\CorreccionYTexto) node[anchor=west] {\footnotesize 
		\color{black} \,\,\,\,(R.ii)}; 
	\draw[color=\ColorShr] (\XUno,\YGrupo)--(\XUnCuarto,\YGrupo);	
	\filldraw [color=\ColorShr,fill=\ColorShr] (\XUno,\YGrupo) circle (2pt);
	\filldraw [color=black,fill=white] (\XUnCuarto,\YGrupo) circle (2pt);
	
	\def\YGrupo{\Y-\YPrimeraFamilia-3*\YSepEntreFamilias}
	\draw (\XGrupo,\YGrupo+\CorreccionYTexto) node[anchor=west] {\footnotesize 
		\color{black} \,\,\,\,(R.iii)}; 
	\draw[color=\ColorShr] (\XUno,\YGrupo)--(\XUnCuarto,\YGrupo);	
	\filldraw [color=\ColorShr,fill=\ColorShr] (\XUno,\YGrupo) circle (2pt);
	\filldraw [color=black,fill=white] (\XUnCuarto,\YGrupo) circle (2pt);
	
	\def\YGrupo{\Y-\YPrimeraFamilia-4*\YSepEntreFamilias}
	\draw (\XGrupo,\YGrupo+\CorreccionYTexto) node[anchor=west] {\footnotesize 
		\color{black} \,\,\,\,(L.Ia.i)}; 
	\draw[color=\ColorExp] (\XUno,\YGrupo)--(\XUnCuarto,\YGrupo);	
	\filldraw [color=\ColorExp,fill=\ColorExp] (\XUno,\YGrupo) circle (2pt);
	\filldraw [color=black,fill=white] (\XUnCuarto,\YGrupo) circle (2pt);
	
	\def\YGrupo{\Y-\YPrimeraFamilia-5*\YSepEntreFamilias}
	\draw (\XGrupo,\YGrupo+\CorreccionYTexto) node[anchor=west] {\footnotesize 
		\color{black} \,\,\,\,(L.Ia.ii)}; 
	\draw[color=\ColorExp] (\XUno,\YGrupo)--(\XUnCuarto,\YGrupo);	
	\filldraw [color=\ColorExp,fill=\ColorExp] (\XUno,\YGrupo) circle (2pt);
	\filldraw [color=black,fill=white] (\XUnCuarto,\YGrupo) circle (2pt);
	
	\def\YGrupo{\Y-\YPrimeraFamilia-6*\YSepEntreFamilias}
	\draw (\XGrupo,\YGrupo+\CorreccionYTexto) node[anchor=west] {\footnotesize 
		\color{black} \,\,\,\,(L.Ia.iii)}; 
	\draw[color=\ColorExp] (\XUno,\YGrupo)--(\XUnCuarto,\YGrupo);	
	\filldraw [color=\ColorExp,fill=\ColorExp] (\XUno,\YGrupo) circle (2pt);
	\filldraw [color=black,fill=white] (\XUnCuarto,\YGrupo) circle (2pt);
	
	\def\YGrupo{\Y-\YPrimeraFamilia-7*\YSepEntreFamilias}
	\draw (\XGrupo,\YGrupo+\CorreccionYTexto) node[anchor=west] {\footnotesize 
		\color{black} \,\,\,\,(L.Ib.i)}; 	
	\draw[color=\ColorShr,latex-](\Xmenosinf,\YGrupo-\YSepDosMetricas)--(\XUno,\YGrupo-\YSepDosMetricas); 
	\draw[color=\ColorExp,-](\XUno,\YGrupo+\YSepDosMetricas)--(\XCero,\YGrupo+\YSepDosMetricas);
	\draw[color=\ColorShr,-latex](\XCero,\YGrupo-\YSepDosMetricas)--(\Xmasinf,\YGrupo-\YSepDosMetricas); 
	
	\filldraw [color=\ColorShr,fill=\ColorShr] (\XUno,\YGrupo-\YSepDosMetricas )circle (2pt); 
	\filldraw [color=\ColorSte,fill=\ColorSte] (\XCero,\YGrupo-\YSepDosMetricas)circle (2pt); 
	\filldraw [color=\ColorExp,fill=\ColorExp] (\XUno,\YGrupo+\YSepDosMetricas )circle (2pt); 
	\filldraw [color=black,fill=white] (\XCero,\YGrupo+\YSepDosMetricas )circle (2pt); 	
	
	\def\YGrupo{\Y-\YPrimeraFamilia-8*\YSepEntreFamilias}
	\draw (\XGrupo,\YGrupo+\CorreccionYTexto) node[anchor=west] {\footnotesize 
		\color{black} \,\,\,\,(L.Ib.ii)};
	\filldraw [color=\ColorShr,fill=\ColorShr] (\XTresMedios,\YGrupo) circle (2pt); 
	
	\def\YGrupo{\Y-\YPrimeraFamilia-9*\YSepEntreFamilias}
	\draw (\XGrupo,\YGrupo+\CorreccionYTexto) node[anchor=west] {\footnotesize 
		\color{black} \,\,\,\,(L.II.i)}; 
	\draw[color=\ColorExp] (\XUno,\YGrupo)--(\XUnCuarto,\YGrupo);	
	\filldraw [color=\ColorExp,fill=\ColorExp] (\XUno,\YGrupo) circle (2pt);
	\filldraw [color=\ColorExp,fill=\ColorExp] (\XUnCuarto,\YGrupo) circle (2pt);
	
	\def\YGrupo{\Y-\YPrimeraFamilia-10*\YSepEntreFamilias}
	\draw (\XGrupo,\YGrupo+\CorreccionYTexto) node[anchor=west] {\footnotesize 
		\color{black} \,\,\,\,(L.II.ii)}; 
	\draw[color=\ColorExp] (\XUno,\YGrupo)--(\XUnCuarto,\YGrupo);	
	\filldraw [color=black,fill=white] (\XUno,\YGrupo) circle (2pt);
	\filldraw [color=black,fill=white] (\XUnCuarto,\YGrupo) circle (2pt);
	
	\def\YGrupo{\Y-\YPrimeraFamilia-11*\YSepEntreFamilias}
	\draw (\XGrupo,\YGrupo+\CorreccionYTexto) node[anchor=west] {\footnotesize 
		\color{black} \,\,\,\,(L.III)};
	\filldraw [color=\ColorExp,fill=\ColorExp] (\XUnCuarto,\YGrupo) circle (2pt);

	\def\YGrupo{\Y-\YPrimeraFamilia-12*\YSepEntreFamilias}
	\draw (\XGrupo,\YGrupo+\CorreccionYTexto) node[anchor=west] {\footnotesize $\HTres\!\rtimes\!\mathbb{R}$
		\,\,(Thm\,\ref{th:main H3})};
	
	\def\YGrupo{\Y-\YPrimeraFamilia-13*\YSepEntreFamilias}
	\draw (\XGrupo,\YGrupo+\CorreccionYTexto) node[anchor=west] {\footnotesize 
		\color{black} \,\,\,\,(L.Ia-)}; 	
	\draw[color=\ColorShr,latex-](\Xmenosinf,\YGrupo)--(\XTresCuartos,\YGrupo);
	\filldraw [color=\ColorShr,fill=\ColorShr] (\XTresCuartos,\YGrupo) circle (2pt);
	\draw[color=\ColorShr,-latex](\XCero,\YGrupo)--(\Xmasinf,\YGrupo);
	\filldraw [color=black,fill=white] (\XCero,\YGrupo) circle (2pt);
	
	\def\YGrupo{\Y-\YPrimeraFamilia-14*\YSepEntreFamilias}
	\draw (\XGrupo,\YGrupo+\CorreccionYTexto) node[anchor=west] {\footnotesize 
		\color{black} \,\,\,\,(L.II)};
	\filldraw [color=\ColorExp,fill=\ColorExp] (\XUnMedio,\YGrupo) circle (2pt); 	
	
	\end{tikzpicture}

	
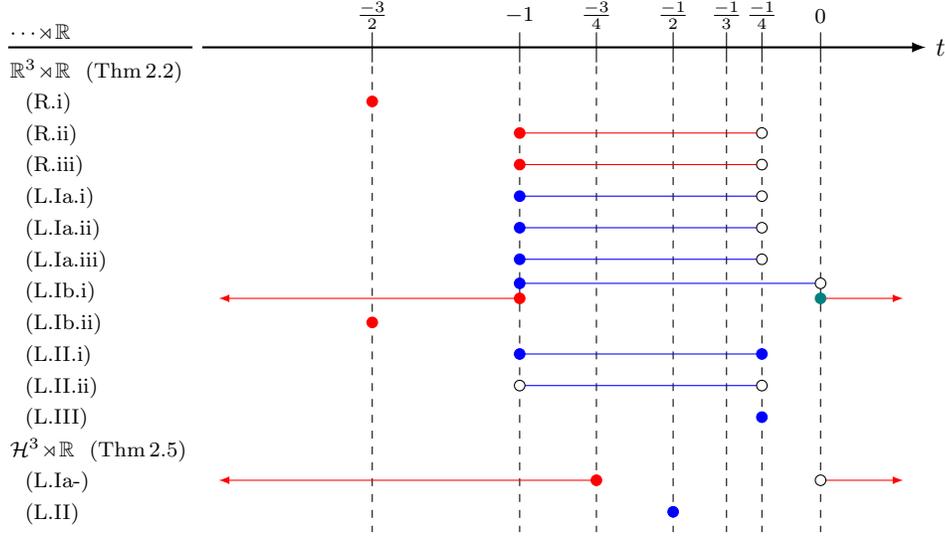
\captionof{figure}{Range of the parameter $t$ for homothetic classes of four-dimensional strict algebraic Lorentzian Ricci solitons with $\tau\neq 0$.}
	\label{fi: figura main}	
\end{center}

Note that left-invariant metrics corresponding to case~(L.II.iii) in Theorem~\ref{th:main R3} do not appear in Figure~\ref{fi: figura main} since they are homothetic (although not isomorphically homothetic) to those in case~(L.II.i), as shown in Remark~\ref{re:II}.
Moreover, it follows after a detailed analysis of the spectral structure of the Ricci operator and  of the curvature operator $R:\Lambda^2\rightarrow\Lambda^2$ that all classes in Figure~\ref{fi: figura main} are homothetically inequivalent, except possibly those corresponding to (L.II) in Theorem~\ref{th:main H3} and (L.II.ii) in Theorem~\ref{th:main R3}.

\section{Three-dimensional Lorentzian algebraic Ricci solitons}\label{se:3D} 
 Algebraic Ricci solitons on three-dimensional Lorentzian Lie groups were considered by Batat and Onda in \cite{Batat-Onda} (see also \cite{Wears}).  Following \cite{Sandro} we review  their classification, due to some inaccuracies in the description  of left-invariant Lorentz metrics in non-unimodular Lie groups in \cite{CP}.
 
 \subsection{The unimodular case}
 In the unimodular situation, the Heisenberg group $\mathcal{H}^3$, the Euclidean group $\widetilde{E}(2)$ and the Poincaré group $E(1,1)$ admit non-Einstein Lorentzian algebraic Ricci solitons, which are determined as follows. 
 Let $L$ be the structure operator of the three-dimensional Lorentzian Lie algebra $(\mathfrak{g},\ip)$, determined by $L(e_i\times e_j)=[e_i,e_j]$, where `$\times$' denotes the Lorentzian vector-cross product.
 It was shown by Milnor \cite{Milnor} and Rahmani \cite{Rahmani} that a Riemannian or Lorentzian three-dimensional Lie group $(G,\ip)$ is unimodular if and only if the corresponding structure operator $L$ is self-adjoint. In the Lorentzian setting there are four-distinct possibilities for the Jordan normal form  and thus we have:
 
 \medskip
 \noindent\underline{Type Ia}. Diagonalizable structure operator with real eigenvalues.\\
 Let $\{ e_1,e_2,e_3\}$ 
 be a basis of $\mathfrak{g}$ so that 
 \begin{equation}\label{eq:3D-unimodularIa}
 [e_1,e_2]=-\lambda_3 e_3,\quad [e_1,e_3]=-\lambda_2 e_2,\quad [e_2,e_3]=\lambda_1 e_1,
\end{equation}
   and the Lorentzian inner product is  determined by $\langle e_1,e_1\rangle=\langle e_2,e_2\rangle=-\langle e_3,e_3\rangle=~1$ \cite{Milnor, Rahmani}. Then it follows from the work in \cite{Batat-Onda} that $(G,\ip)$ is a non-Einstein algebraic Ricci soliton if and only if
 \begin{itemize}
 	\item $G$ is the Heisenberg group with a metric homothetic to the metrics determined by  \eqref{eq:3D-unimodularIa} with $(\lambda_1,\lambda_2,\lambda_3)=(1,0,0)$ or $(\lambda_1,\lambda_2,\lambda_3)=(0,0,1)$, or
 	\item $G$ is the Euclidean group with a metric homothetic to the metric determined by  \eqref{eq:3D-unimodularIa} with $(\lambda_1,\lambda_2,\lambda_3)=(1,0,-1)$, or
 	\item $G$ is the Poincaré group with a metric homothetic to the metric determined by  \eqref{eq:3D-unimodularIa} with $(\lambda_1,\lambda_2,\lambda_3)=(1,-1,0)$.
 \end{itemize}

\medskip
\noindent\underline{Type Ib}. Diagonalizable structure operator with complex eigenvalues.\\
Let $\{ e_1,e_2,e_3\}$ 
be a basis of $\mathfrak{g}$ so that 
\begin{equation}\label{eq:3D-unimodularIb}
[e_1,e_2]=-\beta e_2-\alpha e_3,\quad [e_1,e_3]=-\alpha e_2+\beta e_3,\quad [e_2,e_3]=\lambda e_1,\qquad \beta\neq 0,
\end{equation}
 and the Lorentzian inner product is  determined by $\langle e_1,e_1\rangle=\langle e_2,e_2\rangle=-\langle e_3,e_3\rangle=~1$ \cite{Rahmani}. Then it follows from the work in \cite{Batat-Onda} that $(G,\ip)$ is a non-Einstein  algebraic Ricci soliton if and only if
\begin{itemize}
	\item $G$ is the Poincaré group with a metric homothetic to the metric determined by  \eqref{eq:3D-unimodularIb} with $(\alpha,\beta,\lambda)=(0,1,0)$. 
\end{itemize} 

\medskip
\noindent\underline{Type II}. The structure operator has a double root of its minimal polynomial.\\
Let $\{ u_1,u_2,u_3\}$ 
be a basis of $\mathfrak{g}$ so that 
\begin{equation}\label{eq:3D-unimodularII}
[u_1,u_2]=\lambda_2 u_3,\quad [u_1,u_3]=-\lambda_1 u_1-\varepsilon u_2,\quad [u_2,u_3]=\lambda_1 u_2,\qquad \varepsilon=\pm 1,
\end{equation}
where the Lorentzian inner product is determined by $\langle u_1,u_2\rangle=\langle u_3,u_3\rangle=1$. Then it follows from the work in \cite{Batat-Onda} that there is no non-Einstein algebraic Ricci solitons in this setting.

\medskip
\noindent\underline{Type III}. The structure operator has a triple root of its minimal polynomial.\\
Let $\{ u_1,u_2,u_3\}$ 
be a basis of $\mathfrak{g}$ so that 
\begin{equation}\label{eq:3D-unimodularIII}
[u_1,u_2]=u_1+\lambda u_3,\quad [u_1,u_3]=-\lambda u_1,\quad [u_2,u_3]=\lambda u_2+u_3,
\end{equation}
and the Lorentzian inner product is determined by $\langle u_1,u_2\rangle=\langle u_3,u_3\rangle=1$. Then it follows from the work in \cite{Batat-Onda} that $(G,\ip)$ is a non-Einstein  algebraic Ricci soliton if and only if
\begin{itemize}
	\item $G$ is the Poincaré group with a metric homothetic to the metric determined by  \eqref{eq:3D-unimodularIII} with $\lambda=0$. 
\end{itemize}


\subsection{The non-unimodular case}
The description of the three-dimensional non-unimodular Lorentzian Lie groups is based on the consideration of the unimodular kernel $\mathfrak{u}=\{ x\in\mathfrak{g};\, \operatorname{tr}\operatorname{ad}_x=0\}$, which is a two-dimensional Abelian ideal. 
Hence, these groups are semi-direct extensions $\mathbb{R}^2\rtimes_\varphi\mathbb{R}$  and thus correspond to one of the following families determined by the Jordan normal form of the endomorphism $\varphi:\mathbb{R}^2\rightarrow\mathbb{R}^2$ (see, for example, \cite{ABDO}):
$$
\begin{array}{cccc}
\varphi=\left(\begin{array}{cc}1&0\\0&0\end{array}\right), &
\varphi=\left(\begin{array}{cc}1&1\\0&1\end{array}\right), &
\varphi=\left(\begin{array}{cc}1&0\\0&\lambda\end{array}\right), &
\varphi=\left(\begin{array}{cc}\lambda&1\\-1&\lambda\end{array}\right).
\\
\noalign{\medskip}
\mathfrak{r}_{3,0} &
\mathfrak{r}_3 &
\mathfrak{r}_{3,\lambda}\,\, \scalebox{.75}{$(\lambda\neq 0,-1)$} &
\mathfrak{r}'_{3,\lambda}\,\, \scalebox{.75}{$(\lambda\neq 0)$}
\end{array}
$$
Note that $\mathfrak{r}_{3,0}$ is the product Lie algebra $\mathfrak{r}_{3,0}=\mathfrak{aff}(\mathbb{R})\times\mathbb{R}$, while the unimodular semi-direct extensions $\mathfrak{r}_{3,-1}=\mathfrak{e}(1,1)$ and $\mathfrak{r}'_{3,0}=\mathfrak{e}(2)$ correspond to the Poincaré and the Euclidean Lie algebras, respectively.

Since the restriction of the Lorentzian inner product to the unimodular kernel may be Riemannian, Lorentzian or degenerate, one must consider the following three cases separately (see \cite{Sandro}). In all the situations one may assume $\operatorname{tr}\varphi\neq 0$.
However some further restrictions assumed in \cite{CP} are not always valid (see the discussion in \cite{Sandro}), thus leading to some omissions in \cite{Batat-Onda} that we correct below.
 
 \medskip
 \noindent\underline{Type IV.1}. The restriction of the metric to the unimodular kernel is Lorentzian.\\
 Let $\{ e_1,e_2,e_3\}$ 
 be a basis of $\mathfrak{g}$ so that 
 \begin{equation}\label{eq:3D-nunimodularIV1}
 [e_1,e_2]=0,\quad [e_1,e_3]=\alpha e_1+\beta e_2,\quad [e_2,e_3]=\gamma e_1+\delta e_2,
 \qquad \alpha+\delta\neq 0,
 \end{equation}
 and the Lorentzian inner product is determined by $1=-\langle e_1,e_1\rangle=\langle e_2,e_2\rangle=\langle e_3,e_3\rangle$. The underlying Lie group is the semi-direct extension $\mathbb{R}^2\rtimes_\varphi\mathbb{R}$ determined by an endomorphism $\varphi$ of the unimodular kernel which is self-adjoint if and only if $\beta=-\gamma$. In this situation all the possibilities $\mathfrak{r}_{3,0}$, $\mathfrak{r}_3$, $\mathfrak{r}_{3,\lambda}$ and $\mathfrak{r}'_{3,\lambda}$ are realized for different values of the parameters $\alpha$ and $\gamma$, being Einstein only those metrics which are realized on $\mathfrak{r}_{3,1}$.
 
 
 \medskip
 \noindent\underline{Type IV.2}. The restriction of the metric to the unimodular kernel is Riemannian.\\
 Let $\{ e_1,e_2,e_3\}$ 
 be a basis of $\mathfrak{g}$ so that 
 \begin{equation}\label{eq:3D-nunimodularIV2}
 [e_1,e_2]=0,\quad [e_1,e_3]=\alpha e_1+\beta e_2,\quad [e_2,e_3]=\gamma e_1+\delta e_2,
 \qquad \alpha+\delta\neq 0,
 \end{equation}
 and the Lorentzian inner product is determined by $1=\langle e_1,e_1\rangle=\langle e_2,e_2\rangle=-\langle e_3,e_3\rangle$. The underlying Lie group is the semi-direct extension $\mathbb{R}^2\rtimes_\varphi\mathbb{R}$ determined by an endomorphism $\varphi$  which is self-adjoint if and only if $\beta=\gamma$. The left-invariant metrics in the self-adjoint case are realized on $\mathfrak{r}_{3,0}$ or $\mathfrak{r}_{3,\lambda}$ and they are Einstein if and only if they are realized on $\mathfrak{r}_{3,1}$.
 
 \medskip
 \noindent\underline{Type IV.3}. The restriction of the metric to the unimodular kernel is degenerate.\\
 Let $\{ u_1,u_2,u_3\}$ 
 be a basis of $\mathfrak{g}$ so that 
 \begin{equation}\label{eq:3D-nunimodularIV3}
 [u_1,u_2]=0,\quad [u_1,u_3]=\alpha u_1+\beta u_2,\quad [u_2,u_3]=\gamma u_1+\delta u_2,
 \qquad \alpha+\delta\neq 0,
 \end{equation}
 and the Lorentzian inner product is determined by $1=\langle u_1,u_1\rangle=\langle u_2,u_3\rangle$. The underlying Lie group is the semi-direct extension $\mathbb{R}^2\rtimes_\varphi\mathbb{R}$ determined by an endomorphism $\varphi$ of the unimodular kernel $\mathfrak{u}$  which is self-adjoint if and only if $\gamma=0$. The left-invariant metrics with $\gamma=0$ are realized on $\mathfrak{r}_{3,0}$, $\mathfrak{r}_3$, or $\mathfrak{r}_{3,\lambda}$. Moreover, all metrics realized on  $\mathfrak{r}_3$ and $\mathfrak{r}_{3,1}$ are Einstein, as well as some metrics on $\mathfrak{r}_{3,0}$.

 \medskip
 Now it is easy to check that there exist Einstein metrics of types IV.1 and IV.2 with non self-adjoint derivation $\varphi$. Moreover, a straightforward calculation shows that:
 \begin{itemize}
 	\item A three-dimensional non-unimodular Lie group $\mathbb{R}^2\rtimes_\varphi\mathbb{R}$ is an algebraic Ricci soliton if and only if it is Einstein or $\varphi$ is a self-adjoint endomorphism.
 \end{itemize}
As a consequence, non-Einstein connected and simply connected three-dimensional Lorentzian algebraic Ricci solitons are realized either on the Heisenberg, the Euclidean and the Poincaré group, or on any of the non-unimodular Lie groups except those corresponding to the Lie algebra $\mathfrak{r}_{3,1}$ where all left-invariant Lorentz metrics are Einstein \cite{Nomizu}. It now follows from the above that 
\begin{quote}
\emph{Any connected and  simply connected three-dimensional solvable Lie group admits Lorentzian left-invariant Ricci soliton metrics}.
\end{quote}

In the non-solvable case, there are no non-Einstein Lorentzian algebraic Ricci solitons. Moreover, Lorentzian left-invariant Einstein metrics on non-solvable three-dimensional Lie groups may only occur in the special linear group $\widetilde{SL}(2,\mathbb{R})$.

\begin{remark}
\rm 
It was shown in \cite{Sandro} that any three-dimensional algebraic Lorentzian Ricci soliton is critical for some quadratic curvature functional with zero energy. The converse holds true for all three-dimensional Lorentzian Lie groups but the non-unimodular ones of type~IV.3 corresponding to $\operatorname{det}(\varphi)=0$ with $\gamma\neq 0$, which are $\mathcal{F}[-3]$-critical with zero energy but not algebraic Ricci solitons. On the other hand, note that there exist left-invariant Lorentzian metrics on three-dimensional Lie groups which are Ricci solitons with left-invariant soliton vector field, but not critical for any quadratic curvature functional (cf. \cite{Israel}).

The possible values of $t$ for which the three-dimensional algebraic Ricci solitons are $\mathcal{F}[t]$-critical with zero energy, up to homothety, are represented in Figure~\ref{fi: figura 2}.
We use the same legend as in Figure~\ref{fi: figura main} to represent only the irreducible non-Einstein cases, where algebraic Ricci solitons corresponding to $\mathcal{F}[t]$-critical metrics colored in red are shrinking, those in  green color are steady and those in blue are expanding. Further note that the Heisenberg group admits two non-homothetic left-invariant Lorentz metrics of type~Ia, which are shrinking algebraic Ricci solitons.

\black
\begin{center} 
	\def\ColorExp{blue}
	\def\ColorShr{red}
	\def\ColorSte{teal}
	
	\def\RTres{\mathbb{R}^3}
	\def\HTres{\mathcal{H}^3}
	
	\def\Xmenos{-5.6}
	\def\Xmenost{\Xmenos+.5+1.2+.8}
	\def\Xmenosinf{\Xmenos+.7+.8+.3}
	\def\Xmas{5.6}
	\def\Xmast{\Xmas-1.2}
	\def\Xmasinf{\Xmas-.3}
	\def\Y{0}
	\def\YLineaPunteada{3.09}

	\def\XCorreccion {.08}
	\def\XCero{\Xmast-.2} 
	\def\XUnCuarto{\XCero - .94 + 2*\XCorreccion}
	\def\XUnTercio{\XUnCuarto - 0.31 - 2*\XCorreccion}
	\def\XUnMedio{\XUnTercio - 0.63 - \XCorreccion}
	\def\XTresCuartos{\XUnMedio - .94 - \XCorreccion}
	\def\XUno{\XTresCuartos - .94 - \XCorreccion}
	\def\XTresMedios{\XUno - 1.88 - \XCorreccion}
	\def\XTres{\Xmenost}	
	
	\def\YSepDosMetricas{0.1}

	
	\medskip
	
	\begin{tikzpicture}
	\draw[thick]   (\Xmenos-1,\Y) -- (\Xmenos+.35+.8+.3,\Y); 
	\draw (\Xmenos-1.1,0.2) node[anchor=west] {\footnotesize  $Case$ };

	\draw[thick,-latex]   (\Xmenos+.48+.8+.3,\Y) -- (\Xmas,\Y) node[right] {$t$}; 
	
	\foreach \i/\n in 
	{
		\XCero/$0$,
		\XUnTercio/$\frac{-1}{3}$, 
		\XUnMedio/$\frac{-1}{2}$, 
		\XUno/$-1$,
		\XTres/$-3$ 
	}{\draw (\i,\Y-.2)--(\i,\Y+.2);}

	\foreach \i/\n in 
	{
		\XCero/{\footnotesize $0$},
		\XUno/{\footnotesize $-1$}, 
		\XTres/{\footnotesize $-3$}
	}{\draw  (\i,\Y+.2) node[above] {\footnotesize \n};}

	\foreach \i/\n in 
	{   
		\XUnMedio/{\footnotesize$\frac{-1}{2}$},
		\XUnTercio/{\footnotesize$\frac{-1}{3}$} 
	}{\draw  (\i,\Y+.12) node[above] {\footnotesize \n};}

	\foreach \i/\n in 
	{ 
		\XCero/$0$,
		\XUnTercio/$\frac{-1}{3}$, 
		\XUnMedio/$\frac{-1}{2}$, 
		\XUno/$-1$,
		\XTres/$-3$ 
	}{
		\draw[dashed] (\i,\Y-.25)--(\i,\Y-\YLineaPunteada);
	}
	
	\def\YPrimeraFamilia{0.3}
	\def\YSepEntreFamilias{0.42}
	\def\CorreccionYTexto{-.02}
	\def\XGrupo{\Xmenos-1.1}
	
	\def\YGrupo{\Y-\YPrimeraFamilia}
	\draw (\XGrupo,\YGrupo+\CorreccionYTexto) node[anchor=west] {\footnotesize 
		$\color{black} \HTres \hspace*{.8cm}(Ia)$};
	\filldraw [color=\ColorShr,fill=\ColorShr] (\XTres,\YGrupo+\YSepDosMetricas) circle (2pt);
	\filldraw [color=\ColorShr,fill=\ColorShr] (\XTres,\YGrupo-\YSepDosMetricas) circle (2pt);

	\def\YGrupo{\Y-\YPrimeraFamilia-\YSepEntreFamilias}
	\draw (\XGrupo,\YGrupo+\CorreccionYTexto) node[anchor=west] {\footnotesize 
		$\color{black} \widetilde{E}(2) \hspace*{.585cm}(Ia)$};
	\filldraw [color=\ColorShr,fill=\ColorShr] (\XUno,\YGrupo) circle (2pt); 
	
	\def\YGrupo{\Y-\YPrimeraFamilia-2*\YSepEntreFamilias}
	\draw (\XGrupo,\YGrupo+\CorreccionYTexto) node[anchor=west] {\footnotesize 
		$\color{black} E(1,1) \hspace*{.303cm}(Ia)$};
	\filldraw [color=\ColorShr,fill=\ColorShr] (\XUno,\YGrupo) circle (2pt); 
	
	\def\YGrupo{\Y-\YPrimeraFamilia-3*\YSepEntreFamilias}
	\draw (\XGrupo,\YGrupo+\CorreccionYTexto) node[anchor=west] {\footnotesize 
		$\color{black} E(1,1) \hspace*{.303cm}(Ib)$};
	\filldraw [color=\ColorExp,fill=\ColorExp] (\XUno,\YGrupo) circle (2pt);

	\def\YGrupo{\Y-\YPrimeraFamilia-4*\YSepEntreFamilias}
	\draw (\XGrupo,\YGrupo+\CorreccionYTexto) node[anchor=west] {\footnotesize 
		$\color{black} IV.1$};
	\draw[color=\ColorShr,latex-](\Xmenosinf,\YGrupo)--(\XUno,\YGrupo); 
	\draw[color=\ColorExp,-](\XUno,\YGrupo)--(\XCero,\YGrupo); 
	\draw[color=\ColorShr,-latex](\XCero,\YGrupo)--(\Xmasinf,\YGrupo); 
	\filldraw [color=black,fill=white] (\XUno,\YGrupo) circle (2pt);
	\filldraw [color=black,fill=white] (\XUnMedio,\YGrupo) circle (2pt);
	\filldraw [color=\ColorSte,fill=\ColorSte] (\XCero,\YGrupo) circle (2pt);

	\def\YGrupo{\Y-\YPrimeraFamilia-5*\YSepEntreFamilias}
	\draw (\XGrupo,\YGrupo+\CorreccionYTexto) node[anchor=west] {\footnotesize 
		$\color{black} IV.2$};
	\draw[color=\ColorShr](\XUno,\YGrupo)--(\XUnTercio,\YGrupo);
	\filldraw [color=black,fill=white] (\XUno,\YGrupo) circle (2pt);
	\filldraw [color=black,fill=white] (\XUnMedio,\YGrupo) circle (2pt);
	\filldraw [color=black,fill=white] (\XUnTercio,\YGrupo) circle (2pt);

	\def\YGrupo{\Y-\YPrimeraFamilia-6*\YSepEntreFamilias}
	\draw (\XGrupo,\YGrupo+\CorreccionYTexto) node[anchor=west] {\footnotesize 
		$\color{black} IV.3$};
	\draw[color=\ColorSte,latex-latex](\Xmenosinf,\YGrupo)--(\Xmasinf,\YGrupo); 
	
	\end{tikzpicture}
	
	
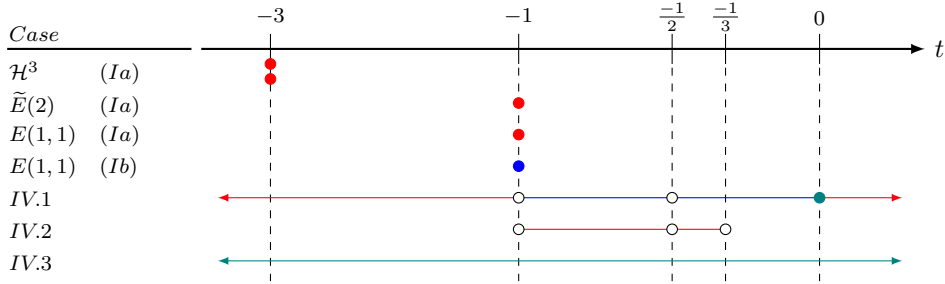
\captionof{figure}{Range of the parameter $t$ for homothetic classes of three-dimensional irreducible non-Einstein algebraic Lorentzian Ricci solitons.}
	\label{fi: figura 2}	
\end{center}

	The left-invariant metric on $E(1,1)$ with structure operator of type~III and $\lambda=0$ is a plane wave, 
	as well as left-invariant metrics of type~IV.3 with $\gamma=0$. Hence they are critical for all quadratic curvature functionals with zero energy \cite{Sandro}. These metrics are either symmetric or homothetic to a corresponding plane wave model  $\mathcal{P}_c$ described in \cite{GGN}. In particular, the plane wave metric on $E(1,1)$ is homothetic (although not isomorphically homothetic) to a left-invariant metric of type~IV.3 with $\gamma=0$ for a suitable value of the parameters, and thus omitted in Figure~\ref{fi: figura 2}. 
Moreover, it follows after a detailed analysis of the spectral structure of the Ricci operator that all classes in Figure~\ref{fi: figura 2} are homothetically inequivalent.	
\end{remark}

\section{Four-dimensional Lorentzian Lie groups}\label{se:Lorentz Lie groups}
 
Let $\mathfrak{g}$ be a four-dimensional Lie algebra. Then it is a product Lie algebra $\mathfrak{g}=\alg\times\mathbb{R}$, where the three-dimensional subalgebra $\alg=\mathfrak{sl}(2,\mathbb{R})$ or $\alg=\mathfrak{su}(2)$, or otherwise it is a solvable Lie algebra which can be obtained as a semi-direct extension of a three-dimensional unimodular Lie algebra,  $\mathfrak{g}=\alg\rtimes\mathbb{R}$, where $\alg$ is one of the Poincaré Lie algebra $\mathfrak{e}(1,1)$, the Euclidean Lie algebra $\mathfrak{e}(2)$, the Heisenberg Lie algebra $\mathfrak{h}_3$, or  the Abelian Lie algebra $\mathbb{R}^3$ (see, for example, \cite{ABDO}). Next, we briefly schedulle the description of left-invariant Lorentz metrics on four-dimensional Lie groups, thus complementing previous work in \cite{CC}.

Let $\ip$ be a Lorentzian inner product on $\mathfrak{g}$. Then the restriction of $\ip$ to the three-dimensional unimodular ideal $\alg$ may be positive definite, of Lorentzian signature or degenerate. These three possibilities give rise to the following cases.

\subsection{Positive definite metrics on $\bm{\alg}$}\label{se:Lorentz-Riemannian}
If the restriction $(\alg,\ip)$ is positive definite, then the description of such inner products follows from the work of Milnor \cite{Milnor}, based on the fact that the structure operator $L$ given by  $L(X \times Y)= [X,Y ]$ is self-adjoint in the unimodular case, where the vector-cross product $\langle X \times Y, Z \rangle = \det (X,Y,Z)$. Hence there exist an orthonormal basis $\{ e_1,e_2,e_3\}$ of $\alg$ so that 
\[
[e_1,e_2]=\lambda_3 e_3,  \quad
[e_2,e_3]=\lambda_1 e_1,  \quad
[e_3,e_1]=\lambda_2 e_2,
\]
and a complementary timelike vector $e_4$ so that $\mathfrak{g}=\alg\rtimes\operatorname{span}\{ e_4\}$ is to be determined by using the Jacobi identity. Moreover, if $L$ is non-singular then the Lie algebra $\alg$ is $\mathfrak{su}(2)$  if all the eigenvalues have the same sign, and it is  $\mathfrak{sl}(2,\mathbb{R})$ otherwise. If $L$ is of rank two, then the Lie algebra is $\mathfrak{e}(2)$ if the non-zero eigenvalues have the same sign, and it is $\mathfrak{e}(1,1)$ otherwise. The Lie algebra is $\mathfrak{h}_3$ if the structure operator is of rank one, and it is the Abelian Lie algebra $\mathbb{R}^3$ if $L$ vanishes.

\subsection{Lorentzian metrics on $\bm{\alg}$}\label{se:Lorentz-Lorentz}
If the restriction $(\alg,\ip)$ is of Lorentzian signature, then the description of such inner products follows from the work of Rahmani \cite{Rahmani}, based on the  fact that although the structure operator  $L(X \times Y)= [X,Y ]$ is self-adjoint in the unimodular case it is not necessarily diagonalizable. Considering the possible Jordan normal forms of the structure operator one has the following (see \cite{Maria}).

\medskip
\noindent(Ia) \underline{Diagonalizable structure operator}.
	 There exist an orthonormal basis $\{ e_1,e_2,e_3\}$ of $\alg$ so that 
	\[
	[e_1,e_2]=\lambda_3\varepsilon_3 e_3,  \quad
	[e_2,e_3]=\lambda_1\varepsilon_1 e_1,  \quad
	[e_3,e_1]=\lambda_2\varepsilon_2 e_2,
	\]
	where $\varepsilon_i=\langle e_i,e_i\rangle=\pm 1$,
	and a complementary spacelike vector $e_4$ so that $\mathfrak{g}=\alg\rtimes\operatorname{span}\{ e_4\}$ is to be determined by using the Jacobi identity.

Moreover, if $L$ is non-singular then the Lie algebra $\alg$ is $\mathfrak{su}(2)$  if $\varepsilon_i\lambda_i$ have the same sign, and it is  $\mathfrak{sl}(2,\mathbb{R})$ otherwise. If $L$ is of rank two, then the Lie algebra is $\mathfrak{e}(2)$ if $\varepsilon_i\lambda_i$ have the same sign, and it is $\mathfrak{e}(1,1)$ otherwise. 
The analysis splits into two non-equivalent cases depending on the causality of $\operatorname{ker}L$, which are considered in Section~\ref{E-Ia-spacelike} and Section~\ref{E-Ia-timelike}. The Lie algebra is $\mathfrak{h}_3$ if the structure operator is of rank one, and we consider separately the cases when the restriction of the metric to $\operatorname{ker}L$ is positive definite (Section~\ref{H3-se:Lorentz-1}) or Lorentzian (Section~\ref{H3-se:Lorentz-2}). Finally, the Lie algebra is $\mathbb{R}^3$ if $L$ vanishes. The different left-invariant metrics on $\mathbb{R}^3\rtimes\mathbb{R}$ are considered in Section~\ref{R3-se:Lorentz}.
	
\medskip
\noindent (Ib) \underline{Structure operator with complex eigenvalues}.
	There exist an orthonormal basis $\{ e_1,e_2,e_3\}$ of $\alg$ with $e_3$ timelike so that 
	\[
	[e_1,e_2]=-\beta e_2 -\alpha e_3,  \quad
	[e_2,e_3]=\lambda e_1,  \quad
	[e_1,e_3]=-\alpha e_2 +\beta e_3,\quad \beta\neq 0,
	\]
	where $L(e_1)=\lambda e_1$,
	and a complementary spacelike vector $e_4$ so that $\mathfrak{g}=\alg\rtimes\operatorname{span}\{ e_4\}$ is to be determined by using the Jacobi identity. 
	
	Moreover, the Lie algebra $\alg$ is $\mathfrak{sl}(2,\mathbb{R})$  if $L$ is non-singular (see Section~\ref{sl-complex}),  and it is $\mathfrak{e}(1,1)$ if the real eigenvalue $\lambda=0$ (cf. Section~\ref{E-complex}). 
	
\medskip
\noindent (II) \underline{The structure operator has a double root of its minimal polynomial}.
	There exist a pseudo-orthonormal basis $\{ u_1,u_2,u_3\}$ of $\alg$ with  
	$\langle u_1, u_2 \rangle = \langle u_3, u_3 \rangle =1$ 
	so that 
	\[
	[u_1,u_2]=\lambda_2 u_3,\quad 
	[u_1,u_3]=-\lambda_1 u_1 -\varepsilon u_2,\quad 
	[u_2,u_3]=\lambda_1 u_2,\quad \varepsilon=\pm 1,
	\]
where the structure operator has eigenvalues $\lambda_1,\lambda_2$ ($\lambda_1$ being a double root of the minimal polynomial), and a complementary spacelike vector $u_4$ so that $\mathfrak{g}=\alg\rtimes\operatorname{span}\{ u_4\}$ is to be determined by using the Jacobi identity. 

Moreover the Lie algebra is $\mathfrak{h}_3$ if $\lambda_1=\lambda_2=0$, i.e., the structure operator has rank one (see Section~\ref{H3-se:Lorentz-3}). If $\lambda_1=0$ and $\lambda_2\neq 0$, then the Lie algebra $\alg$ is $\mathfrak{e}(1,1)$ or $\mathfrak{e}(2)$ depending on whether the sign of $\varepsilon\lambda_2$ is negative or positive, respectively (cf. Section~\ref{E-II-degenerate}). If $\lambda_1\neq 0$ and $\lambda_2=0$ then the underlying Lie algebra is $\alg=\mathfrak{e}(1,1)$ (cf. Section~\ref{E-II-spacelike}), while it is $\alg=\mathfrak{sl}(2,\mathbb{R})$ if the structure operator is non-singular (see Section~\ref{sl-II}).

\medskip
\noindent (III) \underline{The structure operator has a triple root of its minimal polynomial}.
There exist a pseudo-orthonormal basis $\{ u_1,u_2,u_3\}$ of $\alg$ with 
$\langle u_1, u_2 \rangle = \langle u_3, u_3 \rangle =1$ 
so that 
\[
[u_1,u_2]=u_1 +\lambda u_3,\quad 
[u_1,u_3]=-\lambda u_1,\quad 
[u_2,u_3]=\lambda u_2 +u_3,
\]
where the structure operator has a single eigenvalue $\lambda$ (which is a triple root of the minimal polynomial), and a complementary spacelike vector $u_4$ so that $\mathfrak{g}=\alg\rtimes\operatorname{span}\{ u_4\}$ is to be determined by using the Jacobi identity. 

Moreover the Lie algebra is $\alg=\mathfrak{e}(1,1)$ if $\lambda=0$ and $\mathfrak{sl}(2,\mathbb{R})$ otherwise. These cases are considered in Section~\ref{E-III} and Section~\ref{sl-III}, respectively.

\subsection{Degenerate metrics on $\bm{\alg}$}\label{se:Lorentz-degenerate}
Assume that the restriction of the metric $\ip$ of $\mathfrak{g}=\alg\rtimes\mathbb{R}$ to the unimodular subalgebra $\alg$ is degenerate of signature $(++0)$. Then one of the following situations occurs, depending on the dimension of the derived subalgebra $\alg'=[\alg,\alg]$.

\medskip\noindent
(i) If \underline{$\operatorname{dim}\alg'=0$}, then $\alg=\mathbb{R}^3$ and there exists a pseudo-orthonormal basis $\{ u_i\}$ with $\alg=\operatorname{span}\{ u_1,u_2,u_3\}$ and 
	$\langle u_1,u_1\rangle=\langle u_2,u_2\rangle=\langle u_3,u_4\rangle=1$ where $\operatorname{ad
}_{u_4}$ is determined by any endomorphism of $\mathbb{R}^3$. These left-invariant metrics 
are considered in Section~\ref{R3-se:degenerate}.
	
\medskip\noindent
(ii)
If \underline{$\operatorname{dim}\alg'=1$}, then $\alg=\mathfrak{h}_3$ and there are two-distinct situations corresponding to $\alg'$ to be spacelike or null since the restriction of the metric to $\alg$ has signature $(++0)$. Metrics corresponding to a null $\mathfrak{h}_3'$ are considered in Section~\ref{H3-se:deg-null subspace}, while those corresponding to spacelike $\mathfrak{h}_3'$
are discussed in Section~\ref{H3-se:deg-spacelike subspace}.

\medskip\noindent (iii) 
If \underline{$\operatorname{dim}\alg'=2$}, then $\alg=\mathfrak{e}(1,1)$ or $\alg=\mathfrak{e}(2)$, and two-distinct situations may occur depending on whether the restriction of the metric to $\alg'$ is positive definite (Section~\ref{EE-se:riemannian}) or degenerate (Section~\ref{EE-se:degenerate}).

\medskip\noindent (iv)
If \underline{$\operatorname{dim}\alg'=3$}, then $\alg=\mathfrak{sl}(2,\mathbb{R})$ or $\alg=\mathfrak{su}(2)$. 
For any vector $u$ in the radical, one has that $\operatorname{ad}_{u}:\alg\rightarrow\alg$ is of rank two. Hence it has two purely imaginary complex eigenvalues,  two non-zero real opposite eigenvalues, or it is three-step nilpotent. These three distinct situations are analyzed in Sections~\ref{se:complex}, \ref{se:real} and \ref{se:nilpotent},  respectively.

\section{Semi-direct extensions of the Abelian Lie group}\label{se:R3}

We consider separately the cases when the restriction of the metric to the three-dimensional Abelian ideal $\mathbb{R}^3$ is Riemannian (Section~\ref{R3-se: R3 Riemann}), Lorentzian (Section~\ref{R3-se:Lorentz}) or degenerate (Section~\ref{R3-se:degenerate}). The proof of Theorem~\ref{th:main R3} now follows from the analysis below.

\subsection{Semi-direct extensions with Riemannian Lie group $\pmb{\mathbb{R}}^{\bm{3}}$} \label{R3-se: R3 Riemann}

Let $\mathfrak{g}=\algsdR$  be a semi-direct extension of the Abelian Lie algebra $\mathbb{R}^3$ determined by a derivation $D\in\operatorname{End}(\mathbb{R}^3)$. If  $\ip$ is a Lorentzian inner product on $\mathfrak{g}$ whose restriction to $\mathbb{R}^3$ is of Riemannian signature then  the   self-adjoint part of $D$ is diagonalizable. As a consequence,   there exists  an orthonormal basis $\{e_1,e_2,e_3,e_4\}$ of $\mathfrak{g}$, with $e_4$ timelike,  where  $\mathbb{R}^3=\span\{e_1,e_2,e_3\}$ and $\mathbb{R}=\span\{e_4\}$, so that the structure of the metric Lie algebra is given by
\begin{equation}\label{R3-eq: g_R ip}
\mathfrak{g}_R
\left\{
\begin{array}{l}
[e_1,e_4]=\eta_1 e_1-\gamma_1 e_2 - \gamma_2 e_3, \qquad
[e_2,e_4]=\gamma_1 e_1+\eta_2 e_2 - \gamma_3 e_3,
\\
\noalign{\medskip}
[e_3,e_4]=\gamma_2 e_1+\gamma_3 e_2 + \eta_3 e_3, 
\end{array}
\right.
\end{equation}
for certain $\eta_i,\gamma_i\in\mathbb{R}$.

\begin{remark}\rm\label{R3-re: R iso}
	Left-invariant metrics given by Equation~\eqref{R3-eq: g_R ip} are determined by a vector $(\eta_1,\eta_2,\eta_3,\gamma_1,\gamma_2,\gamma_3)\in\mathbb{R}^6$.
	The isometry  
	$( e_1, e_2, e_3, e_4)\mapsto(e_2,e_1,e_3,e_4)$
	shows that $(\eta_1,\eta_2,\eta_3,\gamma_1,\gamma_2,\gamma_3)\sim (\eta_2,\eta_1,\eta_3,-\gamma_1,\gamma_3,\gamma_2)$. 
	Analogously,  the isometry 
	$( e_1, e_2, e_3, e_4)\mapsto(e_3,e_2,e_1,e_4)$ gives the correspondence  $(\eta_1,\eta_2,\eta_3,\gamma_1,\gamma_2,\gamma_3)\sim (\eta_3,\eta_2,\eta_1,-\gamma_3,-\gamma_2,-\gamma_1)$, while  the isometry 
	$(e_1, e_2, e_3, e_4)\mapsto(e_1,e_3,e_2,e_4)$ shows that $(\eta_1,\eta_2,\eta_3,\gamma_1,\gamma_2,\gamma_3)\sim (\eta_1,\eta_3,\eta_2,\gamma_2,\gamma_1,-\gamma_3)$.
\end{remark}

\begin{remark}\label{R3-re: R - loc symm}\rm 	 
	A metric~\eqref{R3-eq: g_R ip}  is Einstein if and only if $\eta_1=\eta_2=\eta_3$ (in which case it is a space of constant sectional curvature). 
	Moreover, a metric~\eqref{R3-eq: g_R ip}  is locally symmetric if and only if it is  Einstein, isomorphically isometric to a metric $\mathfrak{g}_R$ with $\eta_1=\eta_2\neq 0$ and $\eta_3=\gamma_2=\gamma_3=0$ 
	(in which case it is locally conformally flat and locally isometric to a product $ N_{-1}^3(\kappa)\times\mathbb{R}$, where $N_{-1}^3(\kappa)$ is a three-dimensional Lorentzian manifold of constant sectional curvature), 
	or isomorphically isometric to a metric $\mathfrak{g}_R$ with    $\eta_1=\eta_2=\gamma_2=\gamma_3=0$ and $\eta_3\neq 0$  (in which case it is locally isometric to $ N_{-1}^2(\kappa)\times\mathbb{R}^2$).
\end{remark}

\begin{theorem}\label{R3-th: g_R}
	A  left-invariant metric $\mathfrak{g}_R$ on $\sdR$  given by Equation~\eqref{R3-eq: g_R ip} 
	is a strict algebraic Ricci soliton
	if and only if it is isomorphically homothetic to one of the following:
	\begin{itemize}
		\item[(i)]  
		$[ e_1, e_4] =  e_2$,
		$[ e_2, e_4] =  e_3$.
		
		\smallskip
		
		\item[(ii)]
		$[e_1,e_4] = e_1 -\gamma_1 e_2$,
		$[e_2,e_4] = \gamma_1 e_1 + e_2$,
		$[e_3,e_4] = \eta_3 e_3$,
		$\eta_3\notin\{0,1\}$, $\gamma_1\geq0$.
		
		\smallskip
		
		\item[(iii)]
		$[e_1,e_4] = e_1$,
		$[e_2,e_4] = \eta_2 e_2$,
		$[e_3,e_4] = \eta_3 e_3$,
		$\eta_2,\eta_3\notin\{0,1\}$, $\eta_2\neq\eta_3$.
	\end{itemize}
	
	\smallskip
	
	\noindent
	Here $\{e_i\}$ is an orthonormal basis of the Lie algebra with $e_4$ timelike.
\end{theorem}

\begin{remark}\rm\label{re:riem}
(i)  The underlying Lie algebra in Theorem~\ref{R3-th: g_R}--(i) is $\mathfrak{n}_4$. The scalar curvature $\tau=1$ and the algebraic Ricci soliton is shrinking with  $\cARS =\|\rho\|^2\tau^{-1}= \frac{3}{2}$. Moreover this metric is $\mathcal{F}[-3/2]$-critical with zero energy.
		
		\medskip
		
	(ii) The underlying Lie algebra in case~(ii) above is $\mathfrak{r}_{4,1,\eta_3}$ if $\gamma_1=0$, while it corresponds to $\mathfrak{r}'_{4,\mu,\lambda}$ with $\mu=\frac{\eta_3}{\gamma_1}$ and $\lambda=\frac{1}{\gamma_1}$ if $\gamma_1> 0$. The scalar curvature $\tau=2 \left(\eta_3^2+2 \eta_3+3\right)>0$ and the algebraic Ricci solitons are shrinking with $\cARS = \eta_3^2+2$. Moreover, these metrics are $\mathcal{F}[t]$-critical with zero energy for 
		$t=-\cARS\tau^{-1} \in  [-1,-\frac{1}{4})$. 
	

		\medskip
		
(iii) The underlying Lie algebra corresponding to case~(iii) in Theorem~\ref{R3-th: g_R} is $\mathfrak{r}_{4,\eta_2,\eta_3}$. The scalar curvature  $\tau=2 \left(\eta_2^2+\eta_3^2+\eta_2 \eta_3+\eta_2+\eta_3+1\right)>0$ and the algebraic Ricci solitons are shrinking with $\cARS =\eta_2^2+\eta_3^2+1$. These metrics are $\mathcal{F}[t]$-critical for 
		$t=-\cARS\tau^{-1}
		\in  [-1,-\frac{1}{4})$ with zero energy.

%
	\end{remark}

\begin{proof}	
		The endomorphism $\mathfrak{D}= \operatorname{Ric} -\cARS \Id$ is a derivation of the Lie algebra if it satisfies the condition
		$
		\mathfrak{D}[e_i,e_j]-[\mathfrak{D} e_i,e_j]-[e_i,\mathfrak{D} e_j] = 0$, for
		$i,j=1,\dots,4$, 
		which, when expressed with respect to the basis $\{ e_1, e_2, e_3, e_4\}$, is equivalent to
		\[
		\mfP_{ijk} = \mathfrak{D}_\ell{}^k c_{ij}{}^\ell - \mathfrak{D}_i{}^\ell c_{\ell j}{}^k - \mathfrak{D}_j{}^\ell c_{i\ell}{}^k = 0, 
		\]
		where $\mathfrak{D}_s{}^r = \operatorname{Ric}_s{}^r - \cARS \delta_s{}^r$, and the structure constants $c_{ij}{}^\ell$ are determined by the Lie brackets as $[e_i,e_j]=c_{ij}{}^\ell e_\ell$.
	A straightforward calculation  shows that the conditions for $\mathfrak{D}=\Ricci-\cARS\Id$ to be a derivation are determined by a system of polynomial equations on the soliton constant $\cARS$ and the structure constants in~\eqref{R3-eq: g_R ip}, given by $\{\mfP_{ijk}=0\}$, where

	\medskip

	$
	\begin{array}{l}
	\mfP_{141} =     
	-2 (\eta_1 - \eta_2) \gamma_1^2 - 
	2 (\eta_1 - \eta_3) \gamma_2^2 - \eta_1 (\eta_1^2 + \eta_2^2 + 
	\eta_3^2 - \cARS),
	\end{array}
	$

	\vspace{\vsep}

	$
	\begin{array}{l}
	\mfP_{142} =  
	- (\eta_1 + \eta_2 - 
	2 \eta_3) \gamma_2 \gamma_3 + (3 \eta_1^2 + \eta_2^2 + 
	\eta_3^2 - 
	2 \eta_1 \eta_2 + \eta_1 \eta_3 - \eta_2 \eta_3 - 
	\cARS) \gamma_1
	,
	\end{array}
	$

	\vspace{\vsep}

	$
	\begin{array}{l}
	\mfP_{143} =   
	(\eta_1 - 
	2 \eta_2 + \eta_3) \gamma_1 \gamma_3 + (3 \eta_1^2 + 
	\eta_2^2 + \eta_3^2 + \eta_1 \eta_2 - 
	2 \eta_1 \eta_3 - \eta_2 \eta_3 - \cARS) \gamma_2,
	\end{array}
	$

	\vspace{\vsep}

	$
	\begin{array}{l}
	\mfP_{241} =     
	-(\eta_1 + \eta_2 - 2 \eta_3) \gamma_2 \gamma_3 - (\eta_1^2 + 
	3 \eta_2^2 + \eta_3^2 - 
	2 \eta_1 \eta_2 - \eta_1 \eta_3 + \eta_2 \eta_3 - 
	\cARS) \gamma_1 ,
	\end{array}
	$

	\vspace{\vsep}

	$
	\begin{array}{l}
	\mfP_{242} =  
	2  (\eta_1 - \eta_2) \gamma_1^2 + 
	2 \gamma_3^2 (-\eta_2 + \eta_3) - \eta_2 (\eta_1^2 + \eta_2^2 
	+ \eta_3^2 - \cARS) ,
	\end{array}
	$

	\vspace{\vsep}

	$
	\begin{array}{l}
	\mfP_{243} =  
	(2 \eta_1 - \eta_2 - \eta_3) \gamma_1 \gamma_2 + (\eta_1^2 + 
	3 \eta_2^2 + \eta_3^2 + \eta_1 \eta_2 - \eta_1 \eta_3 - 
	2 \eta_2 \eta_3 - \cARS) \gamma_3 ,
	\end{array}
	$

	\vspace{\vsep}

	$
	\begin{array}{l}
	\mfP_{341} =   
	(\eta_1 - 
	2 \eta_2 + \eta_3) \gamma_1 \gamma_3 - (\eta_1^2 + \eta_2^2 
	+ 3 \eta_3^2 - \eta_1 \eta_2 - 
	2  \eta_1 \eta_3 + \eta_2 \eta_3 - \cARS) \gamma_2,
	\end{array}
	$

	\vspace{\vsep}

	$
	\begin{array}{l}
	\mfP_{342} =  
	(2 \eta_1 - \eta_2 - \eta_3) \gamma_1 \gamma_2 - (\eta_1^2 + 
	\eta_2^2 + 3 \eta_3^2 - \eta_1 \eta_2 + \eta_1 \eta_3 - 
	2 \eta_2 \eta_3 - \cARS) \gamma_3,
	\end{array}
	$

	\vspace{\vsep}
	
	$
	\begin{array}{l}
	\mfP_{343} =   
	2  (\eta_1 - \eta_3) \gamma_2^2 + 
	2  (\eta_2 - \eta_3) \gamma_3^2 - \eta_3 (\eta_1^2 + \eta_2^2 
	+ \eta_3^2 - \cARS) .
	\end{array}
	$


	\bigskip
	
	In order to solve the system $\{\mfP_{ijk}=0\}$ we consider the self-adjoint part of the derivation given by $\operatorname{diag}[\eta_1,\eta_2,\eta_3]$ and split the analysis into three cases: some of the parameters is zero or, otherwise,  
	two of the parameters are equal or the three of them are different (note that the metric is Einstein if the three parameters coincide). Moreover, in the former case we may assume
	$\eta_1=0$, in the second case we may take  $\eta_1=\eta_2=1 \neq \eta_3$ and in the latter case we may fix $\eta_1=1$ (see Remark~\ref{R3-re: R iso}).

	\smallskip
	
	\subsection*{Case 1: $\bm{\eta_1=0}$}\label{R3-th g_R Case eta_1=0}
	In this case on easily checks that 
	\[
	\mfP_{142}-\mfP_{241} = 2\gamma_1 (2\eta_2^2+\eta_3^2-\cARS).
	\] 
	Moreover, if $\gamma_1=0$, then
	\[
	\mfP_{141}=2\eta_3\gamma_2^2,
	\quad
	\mfP_{143}=(\eta_2^2+\eta_3^2-\eta_2 \eta_3-\cARS) \gamma_2.
	\]
	Next we analyze the case $\gamma_1=\gamma_2=0$, the case $\gamma_1=0$, $\gamma_2\neq 0$, $\eta_3=0$, $\cARS=\eta_2^2$, and the case $\gamma_1\neq 0$, $\cARS= 2\eta_2^2+\eta_3^2$ separately.

	\subsubsection*{\underline{Case 1.1: $\eta_1=0$, $\gamma_1=\gamma_2=0$}} 
	In this case \eqref{R3-eq: g_R ip} reduces to
	\[
	[e_2,e_4]=\eta_2 e_2-\gamma_3 e_3,
	\quad
	[e_3,e_4]=\gamma_3 e_2 + \eta_3 e_3,
	\] 
	which corresponds to a product Lie algebra $\alg\times\mathbb{R}$, with $\alg=\operatorname{span}\{ e_2,e_3,e_4\}$,
	and
  a direct calculation shows that   $\nabla_{e_i} e_1 = 0$, $i=1,\dots,4$. Hence, $e_1$ generates a left-invariant spacelike parallel vector field and  the Lorentzian Lie group splits as a Lorentzian product Lie group. Therefore this case follows from the discussion in Section~\ref{se:3D}.

	\subsubsection*{\underline{Case 1.2: $\eta_1=0$, $\gamma_1=0$, $\gamma_2\neq 0$, $\eta_3=0$, $\cARS=\eta_2^2$}}
	
	In this case  it is easy to check that  the conditions for $\mathfrak{D}=\Ricci-\cARS\Id$ 
	to be a derivation reduce to $\gamma_3\eta_2=0$. Hence, we have $\eta_2\neq 0$,  $\gamma_3=0$  (since otherwise the space would be Einstein) and we get a   locally
		symmetric space, locally isometric to  $ N_{-1}^2(\kappa)\times\mathbb{R}^2$
		(see Remarks~\ref{R3-re: R - loc symm}~and~\ref{R3-re: R iso}).

	\subsubsection*{\underline{Case 1.3: $\eta_1=0$, $\gamma_1\neq 0$, $\cARS= 2\eta_2^2+\eta_3^2$}}
	The vanishing of $\mfP_{141}=2(\eta_2 \gamma_1^2+\eta_3 \gamma_2^2)$ leads to
	$\eta_2=-\frac{\eta_3\gamma_2^2}{\gamma_1^2}$ and, as a consequence,
	\[
	\mfP_{242} + \mfP_{343} = \tfrac{\eta_3^3 \gamma_2^4(\gamma_1^2-\gamma_2^2)}{\gamma_1^6}.
	\]
	Note that $\eta_3$ must be non-zero to avoid the Einstein case.
	Moreover, if  $\gamma_2=0$ 	then the system $\{\mfP_{ijk}=0\}$ reduces to
	$\gamma_3\eta_3=0$, so that $\gamma_3=0$ and the space is locally symmetric, isometric to $ N_{-1}^2(\kappa)\times\mathbb{R}^2$ (see Remark~\ref{R3-re: R - loc symm}).
	
	\smallskip
	
	Next we analyze the case   $\gamma_2=\ve_1 \gamma_1$, with $\ve_1^2=1$. In this case,  $\mfP_{342}= -3\eta_3^2 \gamma_3$, which implies $\gamma_3=0$ and the system  $\{\mfP_{ijk}=0\}$  reduces to
	\[
	\mfP_{242} = - \mfP_{343}= \eta_3(2\gamma_1^2-\eta_3^2) =0.
	\]
	Thus, $\gamma_1=\frac{\ve_2\eta_3}{\sqrt{2}}$, with $\ve_2^2=1$, and the associated left-invariant metric is given by
	\[
	[e_1,e_4] = -\tfrac{\ve_2\eta_3}{\sqrt{2}}e_2 - \tfrac{\ve_1 \ve_2\eta_3}{\sqrt{2}}e_3
	,\,
	[e_2,e_4] = \tfrac{\ve_2\eta_3}{\sqrt{2}}e_1 -\eta_3e_2,\,
	[e_3,e_4] = \tfrac{\ve_1\ve_2\eta_3}{\sqrt{2}}e_1 + \eta_3e_3 .
	\]
	Now, making the change of basis
	\[
	\begin{array}{ll}
	\bar e_1 = \tfrac{-\ve_1}{2\sqrt{2}\,\eta_3}(\sqrt{2}\,\ve_2 e_1+e_2-\ve_1 e_3),
	&
	\bar e_2 = \tfrac{1}{2\eta_3}(\ve_1 e_2+e_3),
	
	\\[0.1in]
	
	\bar e_3 = \tfrac{\ve_1}{2\sqrt{2}\,\eta_3}(\sqrt{2}\,\ve_2 e_1-e_2+\ve_1 e_3),
	&
	\bar e_4 = \tfrac{1}{\sqrt{2}\,\eta_3}e_4,
	\end{array}
	\]
	the Lie bracket transforms into
	\[
	[\bar e_1,\bar e_4] = \bar e_2,\quad
	[\bar e_2,\bar e_4] = \bar e_3,
	\]
	while the inner product is rescaled by $\tfrac{1}{2\eta_3^2}\ip$. 
	Since we are working at the homothetic level we can maintain the initial inner product
	remaining in the same homothetic class. Thus, we get  an algebraic Ricci soliton with soliton constant $\cARS=\frac{3}{2}$  whose   underlying Lie algebra is $\mathfrak{n}_4$  and case~(i) is obtained.

	\smallskip
	
	\subsection*{Case 2: $\bm{\eta_1=\eta_2=1\neq \eta_3}$,  $\bm{\eta_3\neq 0}$}
	In this case, a direct calculation shows that $\mfP_{141}-\mfP_{242} = 2 (\eta_3-1)(\gamma_2^2-\gamma_3^2)$, which implies  $\gamma_3=\ve \gamma_2$, with $\ve^2=1$. Now, $\mfP_{141}=2(\eta_3-1)\gamma_2^2-\eta_3^2-2+\cARS$, which leads to 
	$\cARS = -2(\eta_3-1)\gamma_2^2+\eta_3^2+2$. At this point, we calculate
	\[
	\mfP_{142} = 2(\eta_3-1) (\gamma_1+\ve)\gamma_2^2,
	\quad 
	\mfP_{241} = -2(\eta_3-1) (\gamma_1-\ve)\gamma_2^2,
	\]
	from where $\gamma_2=0$, and a direct checking shows that the associated left-invariant metric, given by
	\[
	[e_1,e_4] = e_1 -\gamma_1 e_2,\quad
	[e_2,e_4] = \gamma_1 e_1 + e_2,\quad
	[e_3,e_4] = \eta_3 e_3,
	\] 
	determines and algebraic Ricci soliton with soliton constant $\cARS=\eta_3^2+2$.  The isomorphic isometry $(e_1,e_2,e_3,e_4)\mapsto (e_2,e_1,e_3,e_4)$ interchanges the sign of $\gamma_1$, so we may assume $\gamma_1\geq 0$. 
	Moreover, note that if $\gamma_1=0$, then the underlying Lie algebra is $\mathfrak{r}_{4,1,\eta_3}$, while it corresponds to $\mathfrak{r}'_{4,\mu,\lambda}$ with $\mu=\frac{\eta_3}{\gamma_1}$ and $\lambda=\frac{1}{\gamma_1}$ if $\gamma_1> 0$, thus obtaining case~(ii).
	
%

	\smallskip
	
	\subsection*{Case 3: $\bm{\eta_1=1\neq\eta_2\neq \eta_3}$, $\bm{\eta_3\neq 1}$, $\bm{\eta_2\eta_3\neq 0}$}\label{R3-th g_R Case eta_i distintos}
	In this last case we introduce auxiliary variables $\eta_2'$ and $\eta_3'$ to indicate that $\eta_2\eta_3\neq 0$ by means of the polynomials $\eta_2 \eta_2'-1$ and $\eta_3 \eta_3'-1$, and   we consider the ideal $\langle \mfP_{ijk} \cup 
	\{\eta_1-1, \eta_2 \eta_2'-1,\eta_3 \eta_3'-1\}\rangle$ in the polynomial ring
	$\mathbb{R}[\cARS,\eta_1,\gamma_1,\gamma_2,\eta_2',\eta_2,\gamma_3,\eta_3',\eta_3]$. Computing a Gröbner basis of this ideal with respect to the lexicographical order we get 
	$27$ polynomials among which we find
	\[
	\mathbf{g}_1 = (\eta_2-1)^2\gamma_1,\,\,\,\,
	\mathbf{g}_2 = (\eta_3-1)^2\gamma_2,\,\,\,\,
	\mathbf{g}_3 = (\eta_2-\eta_3)^2 \gamma_3,\,\,\,\,
	\mathbf{g}_4 =-\eta_2^2-\eta_3^2-1+\cARS.
	\]
	Hence, it follows that $\gamma_1=\gamma_2=\gamma_3=0$ and $\cARS =\eta_2^2+\eta_3^2+1$. Moreover, the associated left-invariant metric, described by 
	\[
	[e_1,e_4] = e_1,\quad
	[e_2,e_4] = \eta_2 e_2,\quad
	[e_3,e_4] = \eta_3 e_3,
	\] 	  
	determines an algebraic Ricci soliton whose underlying Lie algebra is $\mathfrak{r}_{4,\eta_2,\eta_3}$, and case~(iii) is obtained. 
\end{proof}

\subsection{Semi-direct extensions with Lorentzian Lie group  $\pmb{\mathbb{R}}^{\bm{3}}$} \label{R3-se:Lorentz}
Let $\mathfrak{g}=\algsdR$ be a semi-direct extension of the Abelian Lie algebra $\mathbb{R}^3$ determined by an endomorphism $D\in\operatorname{End}({\mathbb{R}^3})$. 
As in the previous section we consider the self-adjoint part of the derivation, $D_{sad}$, but since the induced inner product on $\mathbb{R}^3$ is Lorentzian, one must consider the possible Jordan normal forms of $D_{sad}$.
We proceed as in \cite{Maria} in order to simplify the structure constants.
Let $\Phi(x,y)=\langle Dx,y \rangle$ be the associated bilinear form, and let $\Phi_{s}=\frac{1}{2}(\Phi+{}^t\Phi)$ and $\Phi_{a}=\frac{1}{2}(\Phi-{}^t\Phi)$ be the symmetric and skew-symmetric parts of $\Phi$, respectively. Moreover, let $D_{sad}$ and $D_{asad}$ defined by $\Phi_{s}(x,y)=\langle D_{sad} x,y\rangle$ and $\Phi_{a}(x,y)=\langle D_{asad} x,y\rangle$ be the corresponding self-adjoint and anti-self-adjoint endomorphisms. We analyze separately the different Jordan normal forms of $D_{sad}$.

\subsubsection{\bf The self-adjoint part of the derivation $\bm{D_{sad}}$ is diagonalizable}
In this case, there exists an orthonormal basis $\{e_1,e_2,e_3\}$ of $\mathbb{R}^3$, with $e_3$ timelike, so that
\[
D_{sad}=\left(
\begin{array}{ccc}
\eta_1 & 0 & 0
\\
0 & \eta_2 & 0
\\
0 & 0 & \eta_3
\end{array}
\right),
\quad
D_{asad}=\left(
\begin{array}{ccc}
0 & \gamma_1  & \gamma_2
\\
-\gamma_1 & 0 & \gamma_3
\\
\gamma_2 & \gamma_3 & 0
\end{array}
\right)
\]
and therefore left-invariant metrics are described by
\begin{equation}\label{R3-eq: g_L.Ia ip}
\mathfrak{g}_{L.Ia}
\left\{
\begin{array}{l}
[e_1,e_4]=\eta_1 e_1-\gamma_1 e_2 + \gamma_2 e_3, \qquad
[e_2,e_4]=\gamma_1 e_1+\eta_2 e_2 + \gamma_3 e_3,
\\
\noalign{\medskip}
[e_3,e_4]=\gamma_2 e_1+\gamma_3 e_2 + \eta_3 e_3, 
\end{array}
\right.
\end{equation}
where $\{ e_1,e_2,e_3,e_4\}$ is an orthonormal basis of $\algsdR$ with $e_3$ timelike.

\begin{remark}\rm\label{R3-re: L.Ia iso}
	Left-invariant metrics given by Equation~\eqref{R3-eq: g_L.Ia ip} are determined by a vector $(\eta_1,\eta_2,\eta_3,\gamma_1,\gamma_2,\gamma_3)\in\mathbb{R}^6$.
	The isometry  
	$( e_1, e_2, e_3, e_4)\mapsto(e_2,e_1,e_3,e_4)$
	shows that $(\eta_1,\eta_2,\eta_3,\gamma_1,\gamma_2,\gamma_3)\sim (\eta_2,\eta_1,\eta_3,-\gamma_1,\gamma_3,\gamma_2)$. 
\end{remark}

\begin{remark}\label{R3-re: L.Ia - loc symm}\rm
	A metric~\eqref{R3-eq: g_L.Ia ip}  is Einstein if and only if $\eta_1=\eta_2=\eta_3$ (in which case it is of constant sectional curvature). 	 
	Moreover, a metric~\eqref{R3-eq: g_L.Ia ip}  is locally symmetric if and only if it is  Einstein or  isometric to a metric $\mathfrak{g}_{L.Ia}$ with:
	\begin{itemize}
		\item[(a)]  $\eta_1=\eta_2\neq 0$ and  $\eta_3=\gamma_2=\gamma_3=0$ (locally isometric to $ N^3(\kappa)\times\mathbb{R}_{-1}$).
		
		\item[(b)]  $\eta_1=\eta_3\neq 0$ and  $\eta_2=\gamma_1=\gamma_3=0$ (locally isometric to $ N_{-1}^3(\kappa)\times\mathbb{R}$).
		
		\item[(c)] $\eta_1=\eta_2=\gamma_2=\gamma_3=0$ and  $\eta_3\neq 0$ (locally isometric to $ N_{-1}^2(\kappa)\times\mathbb{R}^2$).
		
		\item[(d)] $\eta_1=\eta_3=\gamma_1=\gamma_3=0$ and  $\eta_2\neq 0$ (locally isometric to $ N^2(\kappa)\times\mathbb{R}_{-1}^2$). 
	\end{itemize}
	In cases (a) and (b) above the metric is  also locally conformally flat.
\end{remark}

To determine the cases in which $\mathfrak{D}=\Ricci-\cARS\Id$ is a derivation of the Lie algebra    we will follow the same structure as in the proof of Theorem~\ref{R3-th: g_R}. However, the fact that there exists  only a basic isometry instead of three (compare Remarks~\ref{R3-re: L.Ia iso}~and~\ref{R3-re: R iso}) will imply slight but interesting  differences between \emph{diagonalizable} Lorentzian extensions and Riemannian extensions.

\begin{theorem}\label{R3-th: g_L.Ia}
	A left-invariant metric $\mathfrak{g}_{L.Ia}$ on $\sdR$  given by Equation~\eqref{R3-eq: g_L.Ia ip} 
	is a strict algebraic Ricci soliton
	if and only if it is isomorphically homothetic to one of the following:
	\begin{itemize}
		
		\item[(i)]  
		$[e_1,e_4] = e_1 -\gamma_1 e_2$,
		$[e_2,e_4] = \gamma_1 e_1 + e_2$,
		$[e_3,e_4] = \eta_3 e_3$,
		$\eta_3\notin\{0,1\}$, $\gamma_1\geq 0$.
		
		\smallskip
		
		\item[(ii)]  
		$[e_1,e_4] = e_1 +\gamma_2 e_3$,
		$[e_2,e_4] = \eta_2 e_2$,
		$[e_3,e_4] = \gamma_2 e_1+e_3$,
		$\eta_2\notin\{0,1\}$, $\gamma_2\geq 0$.
		
		\smallskip
		
		\item[(iii)]
		$[e_1,e_4] = e_1$,
		$[e_2,e_4] = \eta_2 e_2$,
		$[e_3,e_4] = \eta_3 e_3$,
		$\eta_2,\eta_3\notin\{0,1\}$, $\eta_2\neq \eta_3$.

	\end{itemize}
	
	\smallskip
	
	\noindent
	Here $\{e_i\}$ is an orthonormal basis of the Lie algebra with $e_3$ timelike.
\end{theorem}

	\begin{remark}\rm\label{re:lor-Ia}
		\,
		(i) If $\gamma_1=0$, then the underlying Lie algebra in Theorem~\ref{R3-th: g_L.Ia}--(i) is $\mathfrak{r}_{4,1,\eta_3}$, while it corresponds to $\mathfrak{r}'_{4,\mu,\lambda}$ with $\mu=\frac{\eta_3}{\gamma_1}$ and $\lambda=\frac{1}{\gamma_1}$ if $\gamma_1> 0$. They are expanding algebraic Ricci solitons with $\tau=-2 \left(\eta_3^2+2 \eta_3+3\right)<0$ and soliton constant $\cARS = -(\eta_3^2+2)$. Moreover they are $\mathcal{F}[t]$-critical with zero energy for 
		$t=-\cARS\tau^{-1} \in [-1,-\frac{1}{4})$. 
		
		\medskip
		
		(ii) 
			 For case~(ii) above a straightforward calculation shows that $\operatorname{ad}_{e_4}$ is diagonalizable with eigenvalues
		$\{ 0,-(\gamma_2+1),\gamma_2-1,-\eta_2\}$. Hence if $\gamma_2\neq 1$, then the underlying Lie algebra is $\mathfrak{r}_{4,\mu,\lambda}$ with $\mu=\frac{1}{\eta_2}(\gamma_2+1)$ and $\lambda=\frac{1}{\eta_2}(1-\gamma_2)$. On the other hand, if $\gamma_2=1$,  the Lie algebra corresponds to the product $\mathfrak{r}_{3,\lambda}\times\mathbb{R}$ with $\lambda=\frac{\eta_2}{2}$, although the left-invariant metric is not the product one.
		These metrics are expanding algebraic Ricci solitons with  $\cARS = -(\eta_2^2+2)$ and scalar curvature $\tau=-2 \left(\eta_2^2+2 \eta_2+3\right)<0$. Moreover they are $\mathcal{F}[t]$-critical with zero energy for 
		$t=-\cARS\tau^{-1}\in [-1,-\frac{1}{4})$. 
		
		\medskip
		
		(iii) The underlying Lie algebra corresponding to Theorem~\ref{R3-th: g_L.Ia}--(iii) is $\mathfrak{r}_{4,\eta_2,\eta_3}$. The scalar curvature is $\tau=-2 \left(\eta_2^2+\eta_3^2+\eta_2 \eta_3+\eta_2+\eta_3+1\right)<0$ and the algebraic Ricci solitons are expanding with $\cARS =-(\eta_2^2+\eta_3^2+1)$. Moreover, these metrics are $\mathcal{F}[t]$-critical with zero energy for 
		$t=-\cARS\tau^{-1}\in [-1,-\frac{1}{4})$.  
		
		\medskip

	\end{remark}

\begin{proof} 
	
	We start determining  the conditions for $\mathfrak{D}=\Ricci-\cARS\Id$ to be a derivation, which are given by   a system of polynomial equations $\{\mfP_{ijk}=0\}$ on the soliton constant $\cARS$ and  the structure constants in~\eqref{R3-eq: g_L.Ia ip}, where
	
	\medskip

	$
	\begin{array}{l}
	\mfP_{141} =     
	2 (\eta_1 - \eta_2) \gamma_1^2 - 
	2 (\eta_1 - \eta_3) \gamma_2^2 
	+ \eta_1 (\eta_1^2 + \eta_2^2 + 
	\eta_3^2 +\cARS),
	\end{array}
	$

	\vspace{\vsep}

	$
	\begin{array}{l}
	\mfP_{142} =  
	- (\eta_1 + \eta_2 - 
	2 \eta_3) \gamma_2 \gamma_3 
	- (3 \eta_1^2 + \eta_2^2 + 
	\eta_3^2 - 
	2 \eta_1 \eta_2 + \eta_1 \eta_3 - \eta_2 \eta_3 
	+ \cARS) \gamma_1
	,
	\end{array}
	$

	\vspace{\vsep}

	$
	\begin{array}{l}
	\mfP_{143} =   
	(\eta_1 - 
	2 \eta_2 + \eta_3) \gamma_1 \gamma_3 + (3 \eta_1^2 + 
	\eta_2^2 + \eta_3^2 + \eta_1 \eta_2 - 
	2 \eta_1 \eta_3 - \eta_2 \eta_3 + \cARS) \gamma_2,
	\end{array}
	$

	\vspace{\vsep}

	$
	\begin{array}{l}
	\mfP_{241} =     
	-(\eta_1 + \eta_2 - 2 \eta_3) \gamma_2 \gamma_3 
	+ (\eta_1^2 + 
	3 \eta_2^2 + \eta_3^2 - 
	2 \eta_1 \eta_2 - \eta_1 \eta_3 + \eta_2 \eta_3 
	+ \cARS) \gamma_1 ,
	\end{array}
	$

	\vspace{\vsep}

	$
	\begin{array}{l}
	\mfP_{242} =  
	- 2  (\eta_1 - \eta_2) \gamma_1^2 + 
	2 \gamma_3^2 (-\eta_2 + \eta_3)
	+ \eta_2 (\eta_1^2 + \eta_2^2 
	+ \eta_3^2 + \cARS) ,
	\end{array}
	$

	\vspace{\vsep}

	$
	\begin{array}{l}
	\mfP_{243} =  
	(2 \eta_1 - \eta_2 - \eta_3) \gamma_1 \gamma_2 + (\eta_1^2 + 
	3 \eta_2^2 + \eta_3^2 + \eta_1 \eta_2 - \eta_1 \eta_3 - 
	2 \eta_2 \eta_3 + \cARS) \gamma_3 ,
	\end{array}
	$

	\vspace{\vsep}

	$
	\begin{array}{l}
	\mfP_{341} =   
	-(\eta_1 - 
	2 \eta_2 + \eta_3) \gamma_1 \gamma_3 
	+ (\eta_1^2 + \eta_2^2 
	+ 3 \eta_3^2 - \eta_1 \eta_2 - 
	2  \eta_1 \eta_3 + \eta_2 \eta_3 + \cARS) \gamma_2,
	\end{array}
	$

	\vspace{\vsep}

	$
	\begin{array}{l}
	\mfP_{342} =  
	- (2 \eta_1 - \eta_2 - \eta_3) \gamma_1 \gamma_2 
	+ (\eta_1^2 + 
	\eta_2^2 + 3 \eta_3^2 - \eta_1 \eta_2 + \eta_1 \eta_3 - 
	2 \eta_2 \eta_3 + \cARS) \gamma_3,
	\end{array}
	$

	\vspace{\vsep}
	
	$
	\begin{array}{l}
	\mfP_{343} =   
	2  (\eta_1 - \eta_3) \gamma_2^2 + 
	2  (\eta_2 - \eta_3) \gamma_3^2 
	+ \eta_3 (\eta_1^2 + \eta_2^2 
	+ \eta_3^2 + \cARS) .
	\end{array}
	$


	\bigskip

	As in Theorem~\ref{R3-th: g_R} we consider the self-adjoint part of the derivation given by $\operatorname{diag}[\eta_1,\eta_2,\eta_3]$ and split the analysis into three cases: some of the parameters is zero or, otherwise,
	two of the parameters are equal or the three of them are different (note that the metric is Einstein if the three parameters coincide). Contrary to Riemannian extensions, in the former case one has to analyze the cases
	$\eta_1=0$ and $\eta_3=0$, while if two of the parameters are equal we have to consider the cases  $\eta_1=\eta_2=1 \neq \eta_3$ and  $\eta_1=\eta_3=1 \neq \eta_2$ (see Remark~\ref{R3-re: L.Ia iso}). Although some of the cases below are obtained in a similar way to the cases in Theorem~\ref{R3-th: g_R}, there are some interesting differences, so  we include them for sake of completeness.

	\smallskip
	
	\subsection*{Case 1a: $\bm{\eta_1=0}$}\label{R3-th g_L.Ia Case eta_1=0}
	We have
	\[
	\mfP_{142}-\mfP_{241} = -2\gamma_1 (2\eta_2^2+\eta_3^2+\cARS),
	\] 
	and,  if $\gamma_1=0$, 
	\[
	\mfP_{141}=2\eta_3\gamma_2^2,
	\quad
	\mfP_{143}=(\eta_2^2+\eta_3^2-\eta_2 \eta_3+\cARS) \gamma_2.
	\]
	Next we analyze the case $\gamma_1=\gamma_2=0$, the case $\gamma_1=0$, $\gamma_2\neq 0$, $\eta_3=0$, $\cARS=-\eta_2^2$, and the case $\gamma_1\neq 0$, $\cARS= -2\eta_2^2-\eta_3^2$ separately.

	\subsubsection*{\underline{Case 1a.1: $\eta_1=0$, $\gamma_1=\gamma_2=0$}} In this case \eqref{R3-eq: g_L.Ia ip} reduces to
	\[
	[e_2,e_4]=\eta_2 e_2+\gamma_3 e_3,
	\quad
	[e_3,e_4]=\gamma_3 e_2 + \eta_3 e_3,
	\]
	which corresponds to a product Lie algebra $\alg\times\mathbb{R}$, with $\alg=\operatorname{span}\{ e_2,e_3,e_4\}$,
	and
	a direct calculation shows that   $\nabla_{e_i} e_1 = 0$, $i=1,\dots,4$. Thus, $e_1$ generates a left-invariant spacelike parallel vector field and  the Lorentzian Lie group splits as a Lorentzian product Lie group. Hence this case follows from the discussion in Section~\ref{se:3D}.

	\subsubsection*{\underline{Case 1a.2: $\eta_1=0$, $\gamma_1=0$, $\gamma_2\neq 0$, $\eta_3=0$, $\cARS=-\eta_2^2$}}
	
	In this case  it is easy to check that the system $\{\mfP_{ijk}=0\}$ reduces to $\gamma_3\eta_2=0$. Since the space is Einstein for $\eta_2=0$, it follows that
	$\eta_2\neq 0$,  $\gamma_3=0$  and we get a   locally
		symmetric space, which is also locally conformally flat and locally isometric to  $ N^2(\kappa)\times\mathbb{R}_{-1}^2$
		(see Remark~\ref{R3-re: L.Ia - loc symm}).

	\subsubsection*{\underline{Case 1a.3: $\eta_1=0$, $\gamma_1\neq 0$, $\cARS= -2\eta_2^2-\eta_3^2$}}
	Since $\mfP_{141}=-2(\eta_2 \gamma_1^2-\eta_3 \gamma_2^2)$, we get
	$\eta_2=\frac{\eta_3\gamma_2^2}{\gamma_1^2}$ and, as a consequence,
	\[
	\mfP_{242} + \mfP_{343} = - \tfrac{\eta_3^3 \gamma_2^4  (\gamma_1^2+\gamma_2^2)}{\gamma_1^6}.
	\]
	Once again $\eta_3$ must be non-zero to avoid the Einstein case. 
	Hence, necessarily    $\gamma_2=0$
	and the system $\{\mfP_{ijk}=0\}$ reduces to
	$\gamma_3\eta_3=0$, so that $\gamma_3=0$ and the space is  locally symmetric, and locally isometric to $ N_{-1}^2(\kappa)\times\mathbb{R}^2$ (see Remark~\ref{R3-re: L.Ia - loc symm}).

	\smallskip
	
	\subsection*{Case 1b:  $\bm{\eta_1\neq 0}$, $\bm{\eta_3=0}$}\label{R3-th g_L.Ia Case eta_3=0}
	In this case on easily checks that 
	\[
	\mfP_{143}+\mfP_{341} = 2\gamma_2 (2\eta_1^2+\eta_2^2+\cARS).
	\] 
	Moreover, if $\gamma_2=0$, we have
	\[
	\mfP_{343}=2\eta_2\gamma_3^2,
	\quad
	\mfP_{342}=(\eta_1^2+\eta_2^2-\eta_1 \eta_2+\cARS) \gamma_3.
	\]
	Next we analyze the case $\gamma_2=\gamma_3=0$, the case $\gamma_2=0$, $\gamma_3\neq 0$, $\eta_2=0$, $\cARS=-\eta_1^2$, and the case $\gamma_2\neq 0$, $\cARS= -2\eta_1^2-\eta_2^2$ separately.

	\subsubsection*{\underline{Case 1b.1: $\eta_1\neq 0$, $\eta_3=0$, $\gamma_2=\gamma_3=0$}} In this setting \eqref{R3-eq: g_L.Ia ip} becomes
	\[
	[e_1,e_4]=\eta_1 e_1-\gamma_1 e_2,
	\quad
	[e_2,e_4]=\gamma_1 e_1 + \eta_2 e_2,
	\]
	which corresponds to a product Lie algebra $\alg\times\mathbb{R}$, with $\alg=\operatorname{span}\{ e_1,e_2,e_4\}$,
	and
	a direct calculation shows that   $\nabla_{e_i} e_3 = 0$, $i=1,\dots,4$. Therefore, $e_3$ generates a left-invariant timelike parallel vector field and  the Lorentzian Lie group splits as a Lorentzian product Lie group. Hence this case follows from the discussion in Section~\ref{se:3D}.

	\subsubsection*{\underline{Case 1b.2: $\eta_1\neq 0$, $\eta_3=0$, $\gamma_2=0$, $\gamma_3\neq 0$, $\eta_2=0$, $\cARS=-\eta_1^2$}}
	
	In this case  it is easy to check that  the conditions for $\mathfrak{D}=\Ricci-\cARS\Id$  
	to be a derivation reduce to $\gamma_1\eta_1=0$. Hence, necessarily  $\gamma_1=0$   and we obtain a locally symmetric space which is locally isometric to  $ N^2(\kappa)\times\mathbb{R}^2_{-1}$ (see Remarks~\ref{R3-re: L.Ia - loc symm}~and~\ref{R3-re: L.Ia iso}).

	\subsubsection*{\underline{Case 1b.3: $\eta_1\neq 0$, $\eta_3=0$, $\gamma_2\neq 0$, $\cARS= -2\eta_1^2-\eta_2^2$}}
	Since $\mfP_{343}=2(\eta_1 \gamma_2^2+\eta_2 \gamma_3^2)$ must vanish it follows that 
	$\eta_1=-\frac{\eta_2\gamma_3^2}{\gamma_2^2}$ and now we get
	\[
	\mfP_{141} + \mfP_{242} = - \tfrac{\eta_2^3 \gamma_3^4  (\gamma_2^2-\gamma_3^2)}{\gamma_2^6}.
	\]
	Note that  $\eta_2$ and $\gamma_3$ must be non-zero since $\eta_1\neq 0$.
	Hence, $\gamma_3=\ve\gamma_2$, with $\ve^2=1$, and a direct calculation shows that 
	\[
	\mfP_{142}=-3\eta_2^2\gamma_1,
	\quad
	\mfP_{242}= \eta_2(4\gamma_1^2-2\gamma_2^2-\eta_2^2).
	\]
	Since $\eta_2\neq 0$, we get $\gamma_1=0$ and $2\gamma_2^2+\eta_2^2=0$, which is not possible. Therefore, there is no algebraic Ricci soliton in this case.

	\smallskip

	\subsection*{Case 2a: $\bm{\eta_1=\eta_2=1\neq \eta_3}$,  $\bm{\eta_3\neq 0}$}
	In this case a direct calculation shows that $\mfP_{141}-\mfP_{242} = 2 (\eta_3-1)(\gamma_2^2-\gamma_3^2)$ and therefore  $\gamma_3=\ve \gamma_2$, with $\ve^2=1$. Now, $\mfP_{141}=2(\eta_3-1)\gamma_2^2+\eta_3^2+2+\cARS$, which leads to 
	$\cARS = -2(\eta_3-1)\gamma_2^2-\eta_3^2-2$ and, as a consequence, 
	\[
	\mfP_{142} = 2(\eta_3-1) (\gamma_1+\ve)\gamma_2^2,
	\quad 
	\mfP_{241} = -2(\eta_3-1) (\gamma_1-\ve)\gamma_2^2.
	\]
	Hence,  $\gamma_2=0$ and a direct checking shows that the associated left-invariant metric, 
	\[
	[e_1,e_4] = e_1 -\gamma_1 e_2,\quad
	[e_2,e_4] = \gamma_1 e_1 + e_2,\quad
	[e_3,e_4] = \eta_3 e_3,
	\] 
	determines and algebraic Ricci soliton with soliton constant $\cARS=-\eta_3^2-2$.   Since the isomorphic isometry $(e_1,e_2,e_3,e_4)\mapsto (e_2,e_1,e_3,e_4)$ interchanges the sign of $\gamma_1$, we may assume $\gamma_1\geq 0$. 
		Moreover, note that if $\gamma_1=0$, then the underlying Lie algebra is $\mathfrak{r}_{4,1,\eta_3}$, while it corresponds to $\mathfrak{r}'_{4,\mu,\lambda}$ with $\mu=\frac{\eta_3}{\gamma_1}$ and $\lambda=\frac{1}{\gamma_1}$ if $\gamma_1> 0$. Thus,  case~(i) is obtained.

	\smallskip
	
	\subsection*{Case 2b: $\bm{\eta_1=\eta_3=1\neq \eta_2}$,  $\bm{\eta_2\neq 0}$}
	In this case, a direct calculation shows that $\mfP_{141}-\mfP_{343} = -2 (\eta_2-1)(\gamma_1^2+\gamma_3^2)$, which implies  $\gamma_1=\gamma_3=0$. Now, a direct checking shows that the system $\{\mfP_{ijk}=0\}$ reduces to 
	$\cARS=-\eta_2^2-2$, and we get  an algebraic Ricci soliton with associated left-invariant metric given by
	\[
	[e_1,e_4] = e_1 +\gamma_2 e_3,\quad
	[e_2,e_4] = \eta_2 e_2,\quad
	[e_3,e_4] = \gamma_2 e_1+e_3.
	\] 
	The isomorphic isometry $e_3\mapsto -e_3$ interchanges the sign of $\gamma_2$, so we may assume $\gamma_2\geq 0$.	
	In this case $\Ricci=-\diag\left[ \eta_2+2, (\eta_2+2)\eta_2, \eta_2+2, \eta_2^2+2 \right]$, which implies that   $\ker\Ricci=\{0\}$ unless  $\eta_2=-2$, in which case  $\ker\Ricci=\span\{ e_1,e_2,e_3\}$ and the covariant derivatives of the left-invariant vector fields are given by
		$$
		\begin{array}{ll}
		\nabla e_1=-e^1\otimes e_4-\gamma_2e^4\otimes e_3,&
		\nabla e_2=2e^2\otimes e_4,
		\\
		\noalign{\medskip}
		\nabla e_3=-e^3\otimes e_4-\gamma_2e^4\otimes e_1,&
		\nabla e_4=e^1\otimes e_1-2e^2\otimes e_2-e^3\otimes e_3,
		\end{array}
		$$
		where $e^\ell=\langle e_\ell,\cdot\rangle$,
		from where it follows that $\ker\Ricci=\span\{ e_1,e_2,e_3\}$ does not contain any non-zero parallel vector field and hence the manifold is not a product Lorentzian Lie group. 	This corresponds to case~(ii).
	
	

	\smallskip 
	
	\subsection*{Case 3: $\bm{\eta_1=1\neq\eta_2\neq \eta_3}$, $\bm{\eta_3\neq 1}$, $\bm{\eta_2\eta_3\neq 0}$}\label{R3-th g_L.Ia Case eta_i distintos}
	In this   case we use auxiliary variables $\eta_2'$ and $\eta_3'$ to indicate that $\eta_2\eta_3\neq 0$ by means of the polynomials $\eta_2 \eta_2'-1$ and $\eta_3 \eta_3'-1$ and,  in the polynomial ring
	$\mathbb{R}[\cARS,\eta_1,\gamma_1,\gamma_2,\eta_2',\eta_2,\gamma_3,\eta_3',\eta_3]$,
	we  consider the ideal generated by $ \mfP_{ijk} \cup 
	\{\eta_1-1, \eta_2 \eta_2'-1,\eta_3 \eta_3'-1\}$. Computing a Gröbner basis of this ideal with respect to the lexicographical order we get 
	$27$ polynomials among which we find
	\[
	\mathbf{g}_1 = (\eta_2-1)^2\gamma_1,\,\,\,\,
	\mathbf{g}_2 = (\eta_3-1)^2\gamma_2,\,\,\,\,
	\mathbf{g}_3 = (\eta_2-\eta_3)^2 \gamma_3,\,\,\,\,
	\mathbf{g}_4 =\eta_2^2+\eta_3^2+1+\cARS.
	\]
	Hence,   $\gamma_1=\gamma_2=\gamma_3=0$ and $\cARS =-\eta_2^2-\eta_3^2-1$, and  we obtain an algebraic Ricci soliton with    associated left-invariant metric 
	\[
	[e_1,e_4] = e_1,\quad
	[e_2,e_4] = \eta_2 e_2,\quad
	[e_3,e_4] = \eta_3 e_3.
	\] 	   
	Since $\eta_2\eta_3\neq 0$, the underlying Lie algebra is $\mathfrak{r}_{4,\eta_2,\eta_3}$, which corresponds to case~(iii), thus finishing the proof.
\end{proof}

\subsubsection{\bf The self-adjoint part of the derivation $\bm{D_{sad}}$ has complex eigenvalues}
If the self-adjoint part of the derivation, $D_{sad}$, has complex eigenvalues then  there exists an orthonormal basis $\{e_1,e_2,e_3\}$ of $\mathbb{R}^3$, with $e_3$ timelike, so that
\[
D_{sad}=\left(
\begin{array}{ccc}
\eta & 0 & 0
\\
0 & \delta & \nu
\\
0 & -\nu & \delta
\end{array}
\right),
\quad
D_{asad}=\left(
\begin{array}{ccc}
0 & \gamma_1  & \gamma_2
\\
-\gamma_1 & 0 & \gamma_3
\\
\gamma_2 & \gamma_3 & 0
\end{array}
\right),
\]
where $\nu\neq 0$. The corresponding left-invariant metrics are described by
\begin{equation}\label{R3-eq: g_L.Ib ip}
\mathfrak{g}_{L.Ib}
\left\{\!\!\!
\begin{array}{l}
[e_1,e_4]=\eta e_1-\gamma_1 e_2 + \gamma_2 e_3,  \qquad\,\,\,\,
[e_2,e_4]=\gamma_1 e_1+\delta e_2 + (\gamma_3-\nu) e_3,
\\
\noalign{\medskip}
[e_3,e_4]=\gamma_2 e_1+(\gamma_3+\nu) e_2 + \delta e_3, 
\end{array}
\right.
\end{equation}
where $\{ e_1,e_2,e_3,e_4\}$ is an orthonormal basis of $\algsdR$ with $e_3$ timelike.

\begin{remark}\label{R3-re: L.Ib - loc symm}\rm
	A metric~\eqref{R3-eq: g_L.Ib ip} is Einstein (indeed Ricci-flat) if and only if 
	$\gamma_1=\gamma_2=\gamma_3=0$ and $\eta=-2\delta=\pm \frac{2}{\sqrt{3}}\nu$.
	Moreover, it is never   locally symmetric.
\end{remark}

\begin{theorem}\label{R3-th: g_L.Ib}
	A  left-invariant metric $\mathfrak{g}_{L.Ib}$ on $\sdR$  given by Equation~\eqref{R3-eq: g_L.Ib ip} 
	is a strict algebraic Ricci soliton
	if and only if it is isomorphically homothetic to one of the following:
	\begin{itemize}
		
		\item[(i)]  
		$[e_1,e_4] = \eta e_1$,
		$[e_2,e_4] = \delta e_2-e_3$,
		$[e_3,e_4] = e_2+\delta e_3$,
		with $\eta\neq 0$ and  $(\eta,\delta)\notin\{
		(\frac{2}{\sqrt{3}},-\frac{1}{\sqrt{3}}),
		(-\frac{2}{\sqrt{3}},\frac{1}{\sqrt{3}})\}$.

		\smallskip
		
		\item[(ii)]  
		$[e_1,e_4] =  e_3$,
		$[e_2,e_4] = -e_3$,
		$[e_3,e_4] =  e_1+e_2$.

	\end{itemize}
	
	\smallskip
	
	\noindent
	Here $\{e_i\}$ is an orthonormal basis of the Lie algebra with $e_3$ timelike.
\end{theorem}

	\begin{remark}\rm\label{re:Ib}
		\,
		(i) The underlying Lie algebra in case~(i) is  $\mathfrak{r}'_{4,\eta,\delta}$. The soliton constant is $\cARS = -\eta^2-2\delta^2+2$ so that they can be expanding, steady or shrinking depending on the values of $(\eta,\delta)$.
		They are steady   if  $\eta^2+2\delta^2=2$, shrinking solitons if $\eta^2+2\delta^2<2$, and expanding algebraic Ricci solitons otherwise.
		 Moreover metrics in this family are $\mathcal{F}[t]$-critical with zero energy for 
		$t=-\frac{\eta^2+2\delta^2-2}{2(\eta^2+3\delta^2+2\eta\delta-1)}\in\mathbb{R}$ if the scalar curvature $\tau=-2((\eta+\delta)^2+2\delta^2-1)$ does not vanish, and   $\mathcal{S}$-critical if and only if the scalar curvature vanishes, in which case $\|\rho\|=0$. 
		This family provides $\mathcal{F}[t]$-critical metrics with zero energy for all values of $t\in\mathbb{R}$ (as it occurs with plane waves) without having vanishing scalar curvature. 
	 Algebraic Ricci solitons which are $\mathcal{F}[t]$-critical are steady if $t=0$, expanding if $-1\leq t<0$, and shrinking if either $t\leq -1$ or $t>0$.

		\medskip
		
		(ii) The underlying Lie algebra is $\mathfrak{n}_4$ in  case~(ii) above. The algebraic Ricci soliton is shrinking with soliton constant $\cARS = 3$ and the scalar curvature is $\tau=2$. Moreover, this metric is $\mathcal{F}[t]$-critical with zero energy for 
		$t=-\frac{3}{2}$.  			
	\end{remark}

\begin{proof} 
	A straightforward calculation  shows that the conditions for $\mathfrak{D}=\Ricci-\cARS\Id$ to be a derivation are determined by a system of polynomial equations on the soliton constant $\cARS$ and the structure constants in~\eqref{R3-eq: g_L.Ib ip}, given by $\{\mfP_{ijk}=0\}$, where

	\medskip

	$
	\!\!\!\!\!\!\!\begin{array}{l}
	\mfP_{141} =     
	(\eta - \delta) (2 \gamma_1^2 - 2 \gamma_2^2)  + 
	4 \nu  \gamma_1 \gamma_2 + (\eta^2 + 
	2 \delta^2 - 2 \nu^2 + \cARS) \eta ,
	\end{array}
	$

	\vspace{\vsep}

	$
	\!\!\!\!\!\!\!\begin{array}{l}
	\mfP_{142} =  
	3 \nu \gamma_1 \gamma_3 - (\eta - \delta) 
	\gamma_2 \gamma_3 - (3 \eta^2 + \delta^2 - \eta \delta - 3 \nu^2 + \cARS) \gamma_1 - 
	3  \eta \nu \gamma_2 ,
	\end{array}
	$

	\vspace{\vsep}

	$
	\!\!\!\!\!\!\!\begin{array}{l}
	\mfP_{143} =   
	(\eta - \delta) \gamma_1 \gamma_3 + 
	3 \nu \gamma_2 \gamma_3 - 
	3  \eta \nu  \gamma_1 + (3 \eta^2 + \delta^2 - \eta \delta - 
	3 \nu^2 + \cARS) \gamma_2 ,
	\end{array}
	$

	\vspace{\vsep}

	$
	\!\!\!\!\!\!\!\begin{array}{l}
	\mfP_{241} =     
	3 \nu \gamma_1 \gamma_3 -  (\eta - \delta) 
	\gamma_2 \gamma_3 + (\eta^2 + 5 \delta^2 - 
	3  \eta \delta - 
	3 \nu^2 + \cARS) \gamma_1 + (\eta - 
	4 \delta) \nu  \gamma_2 ,
	\end{array}
	$

	\vspace{\vsep}

	$
	\!\!\!\!\!\!\!\begin{array}{l}
	\mfP_{242} =  
	- 2 \gamma_1^2 (\eta - \delta) - 
	2 \nu  \gamma_1 \gamma_2 - 
	2  (\eta + 
	2 \delta) \nu \gamma_3 + (\eta^2 + 
	2 \delta^2 - 
	2 \nu^2 + \cARS) \delta ,
	\end{array}
	$

	\vspace{\vsep}

	$
	\!\!\!\!\!\!\!\begin{array}{l}
	\mfP_{243} =  
	-\nu (\gamma_1^2  - \gamma_2^2 - 4 \gamma_3^2) + 
	2  (\eta - \delta) \gamma_1 \gamma_2 
	+ (\eta^2 + 2 \delta^2 - 6 \nu^2 + \cARS) (\gamma_3 -\nu) 
	-4\nu^3 ,  
	\end{array}
	$

	\vspace{\vsep}

	$
	\!\!\!\!\!\!\!\begin{array}{l}
	\mfP_{341} =   
	-(\eta - \delta) \gamma_1 \gamma_3 - 
	3 \nu \gamma_2 \gamma_3 - (\eta - 
	4 \delta) \nu  \gamma_1 + (\eta^2 + 
	5 \delta^2 - 3  \eta \delta - 
	3 \nu^2 + \cARS) \gamma_2 ,
	\end{array}
	$

	\vspace{\vsep}

	$
	\!\!\!\!\!\!\!\begin{array}{l}
	\mfP_{342} =  
	\nu (\gamma_1^2 - \gamma_2^2 - 4 \gamma_3^2 ) - 
	2  (\eta - \delta) \gamma_1 \gamma_2 
	+ (\eta^2 + 2 \delta^2 - 6 \nu^2 + \cARS) (\gamma_3+\nu)
	+4\nu^3 ,  
	\end{array}
	$
	
	\vspace{\vsep}
	
	$
	\!\!\!\!\!\!\!\begin{array}{l}
	\mfP_{343} =   
	2  (\eta - \delta) \gamma_2^2 - 
	2 \nu \gamma_1 \gamma_2 + 
	2  (\eta + 
	2 \delta) \nu \gamma_3 + (\eta^2 + 
	2 \delta^2 - 
	2 \nu^2 + \cARS) \delta .
	\end{array}
	$


	\bigskip

	Since $\nu\neq0$ we may assume $\nu=1$ in the rest of the proof working in the homothetic class of the initial metric, just taking the orthogonal basis $\hat e_i=\frac{1}{\nu}e_i$.
	We start considering the ideal $\langle \mfP_{ijk}\cup \{\nu-1\}\rangle$ in the polynomial ring $\mathbb{R}[\cARS,\eta,\delta,\gamma_1,\gamma_2,\gamma_3,\nu]$. Computing a Gröbner basis  of this ideal with respect to the graded reverse lexicographical order we obtain a set of $35$ polynomials containing
	\[
	\begin{array}{lll}
	\mathbf{g}_1 = (\gamma_1^2+\gamma_2^2)\eta,
	&
	\mathbf{g}_2 = (\gamma_1^2+\gamma_2^2)\delta,
	&
	\mathbf{g}_3 = (\gamma_1^2+\gamma_2^2)(\cARS-3),
	\\[\vsep]
	\mathbf{g}_4 = (\gamma_1^2+\gamma_2^2)\gamma_1\gamma_2,
	&
	\mathbf{g}_5 = (\gamma_1^2+\gamma_2^2)(\gamma_1^2-\gamma_2^2+1),
	&
	\mathbf{g}_6 = \gamma_2\gamma_3(\cARS-12).
	\end{array}
	\]
	Hence, either $\gamma_1=\gamma_2=0$ or, otherwise, $\eta=\delta=\gamma_1=\gamma_3=0$,   $\cARS=3$ and $\gamma_2=\ve$, with $\ve^2=1$. Next we analyze these two cases separately.

	\medskip
	
	\subsection*{Case 1: $\bm{\gamma_1=\gamma_2=0}$}
	In this case, we have 
	\[
	\mfP_{141} = \eta(\eta^2+2\delta^2-2+\cARS).
	\]
	
	If $\eta=0$ then the left-invariant metric is given by 
	\[
	[e_2,e_4] = \delta e_2 +(\gamma_3-1)e_3,\quad
	[e_3,e_4] = (\gamma_3+1)e_2+\delta e_3,
	\]
which corresponds to a product Lie algebra $\alg\times\mathbb{R}$, with $\alg=\operatorname{span}\{ e_2,e_3,e_4\}$,
and
a direct calculation shows that   $\nabla_{e_i} e_1 = 0$, $i=1,\dots,4$. Hence, $e_1$ generates a left-invariant timelike parallel vector field and  the Lorentzian Lie group splits as a Lorentzian product Lie group. Hence this case follows from the discussion in Section~\ref{se:3D}.	


	\medskip
	
	Now, if $\eta\neq0$ then $\cARS=-\eta^2-2\delta^2+2$, and the system $\{\mfP_{ijk}=0\}$ reduces to
	\[
	\mfP_{243} = 4(\gamma_3-1)\gamma_3=0,\,
	\mfP_{342} = -4(\gamma_3+1)\gamma_3=0,\,
	\mfP_{343} = - \mfP_{242} =  2(\eta+2\delta)\gamma_3=0 .
	\]
	Thus, $\gamma_3=0$ and we get  an algebraic Ricci soliton with associated left-invariant metric given by
	\[
	[e_1,e_4] = \eta e_1,\quad
	[e_2,e_4] = \delta e_2-e_3,\quad
	[e_3,e_4] = e_2+\delta e_3.
	\] 
This space is Einstein if and only if $\eta=-2\delta=\pm\frac{2}{\sqrt{3}}$ (see Remark~\ref{R3-re: L.Ib - loc symm}).   
The Ricci operator has eigenvalues $\{ -(\eta+2\delta)\eta, 2-(\eta^2+2\delta^2),-(\eta+2\delta)(\delta\pm\sqrt{-1})\}$,
which implies that the eigenvalues vanish if and only if the space is Einstein. Therefore, in the non-Einstein case, although the scalar curvature $\tau=-2((\eta+\delta)^2+2\delta^2-1)$ may vanish, the space is not a pp-wave. Moreover,
the underlying Lie algebra is  $\mathfrak{r}'_{4,\eta,\delta}$, and  case~(i) is obtained.

%
%

	\medskip
	
	\subsection*{Case 2: $\bm{\eta=\delta=\gamma_1=\gamma_3=0}$,   $\bm{\cARS=3}$ and $\bm{\gamma_2=\ve}$, with $\bm{\ve^2=1}$}
	With these assumptions the system $\{\mfP_{ijk}=0\}$ is satisfied, so that we obtain an algebraic Ricci soliton with    associated left-invariant metric given by
	\[
	[e_1,e_4] = \ve e_3,\quad
	[e_2,e_4] = -e_3,\quad
	[e_3,e_4] = \ve e_1+e_2.
	\] 
	Note that the isomorphic isometry $e_1\mapsto -e_1$ interchanges the sign of $\ve$, so we may assume $\ve=1$. Moreover, since $\operatorname{ad}_{e_4}$ is three-step nilpotent, the underlying Lie algebra is $\mathfrak{n}_4$ and case~(ii) is obtained, thus finishing the proof.
\end{proof}

\subsubsection{\bf The self-adjoint part of the derivation $\bm{D_{sad}}$ has a double root of the minimal polynomial}
In this case, there exists a pseudo-orthonormal basis $\{u_1,u_2,u_3\}$  of $\mathbb{R}^3$, with $\langle u_1,u_2\rangle=\langle u_3,u_3\rangle=1$, so that
\[
D_{sad}=\left(
\begin{array}{ccc}
\eta_1 & 0 & 0
\\
\varepsilon & \eta_1 & 0
\\
0 & 0 & \eta_2
\end{array}
\right),
\quad
D_{asad}=\left(
\begin{array}{ccc}
\gamma_1 & 0  & \gamma_2
\\
0  & -\gamma_1 & \gamma_3
\\
-\gamma_3 & -\gamma_2 & 0
\end{array}
\right),
\]
where $\varepsilon^2=1$. Thus, the corresponding left-invariant metrics are described by 
\begin{equation}\label{R3-eq: g_L.II ip}
\mathfrak{g}_{L.II}
\left\{\!\!\!
\begin{array}{l}
[u_1,u_4]=(\eta_1+\gamma_1) u_1 + \ve  u_2 - \gamma_3 u_3, \,\,\,\,\,
[u_2,u_4]= (\eta_1-\gamma_1) u_2 -  \gamma_2 u_3,
\\
\noalign{\medskip}
[u_3,u_4]=\gamma_2 u_1+\gamma_3 u_2 + \eta_2 u_3 , 
\end{array}
\right.
\end{equation}
where $\{u_1,u_2,u_3,u_4\}$ is a pseudo-orthonormal basis with $\langle u_1,u_2\rangle=\langle u_3,u_3\rangle=\langle u_4,u_4\rangle=1$.

\begin{remark}\label{R3-re: L.II - loc symm}\rm
	A metric~\eqref{R3-eq: g_L.II ip} is Einstein if and only if $\eta_2=\eta_1$, $\gamma_1=-\frac{3}{2}\eta_1$ and $\gamma_2=0$.
	Moreover, 
	a metric~\eqref{R3-eq: g_L.II ip} is locally symmetric if and only if either
	$\eta_1=\eta_2=\gamma_1=\gamma_2=0$ (in which case it is flat), or 
	$\eta_2=\gamma_2=\gamma_3=0$ and $\gamma_1=-\eta_1\neq 0$, in which case is locally conformally flat and locally isometric to $ N_{-1}^3(\kappa)\times\mathbb{R}$.
\end{remark}

\begin{theorem}\label{R3-th: g_L.II}
	A  left-invariant metric $\mathfrak{g}_{L.II}$ on $\sdR$  given by Equation~\eqref{R3-eq: g_L.II ip} 
	is a strict algebraic Ricci soliton
	if and only if it is isomorphically homothetic to one of the following:
	\begin{itemize}
		
		\item[(i)]  
		$[u_1,u_4] = \eta_1 u_1+ u_2$,
		$[u_2,u_4] = \eta_1 u_2$,
		$[u_3,u_4] = \eta_2 u_3$, $\eta_2\neq 0$.

		\smallskip
		
		\item[(ii)]  
		$[u_1,u_4] = -\tfrac{\eta_2}{2}u_1+  u_2$,
		$[u_2,u_4] = \tfrac{4\eta_1+\eta_2}{2} u_2$,
		$[u_3,u_4] = \eta_2 u_3$,
		
		\noindent
		with 
		 $\eta_2(\eta_2-\eta_1)(\eta_2+2\eta_1)\neq 0$.
		
		\smallskip
		
		\item[(iii)]
		$[u_1,u_4] = \eta_1 u_1+  u_2-\gamma_3 u_3$,
		$[u_2,u_4] = \eta_1 u_2$,
		$[u_3,u_4] = \gamma_3 u_2 +\eta_1 u_3$, 
		$\eta_1\gamma_3\neq 0$.

	\end{itemize}
	
	\smallskip
	
	\noindent
	Here $\{u_i\}$ is a pseudo-orthonormal basis with $\langle u_1,u_2\rangle=\langle u_3,u_3\rangle=\langle u_4,u_4\rangle=1$. 
\end{theorem}

	\begin{remark}\rm\label{re:II}
		(i) The Lie algebra is $\mathfrak{r}_{4,\lambda}$ with $\lambda=\frac{\eta_1}{\eta_2}$. The algebraic Ricci solitons are expanding with  $\cARS = -(2\eta_1^2+\eta_2^2)$ and  $\tau=-2(3\eta_1^2+\eta_2^2+2\eta_1\eta_2)<0$. Moreover, these metrics are $\mathcal{F}[t]$-critical with zero energy  for 
		$t=-\frac{2\eta_1^2+\eta_2^2}{2(3\eta_1^2+\eta_2^2+2\eta_1\eta_2)}
		\in [-1,-\frac{1}{4}]$. 

		\medskip
				
		(ii) In this case the underlying Lie algebra  is $\mathfrak{r}_{4,-2,\lambda}$ with $\lambda=	-\tfrac{4\eta_1+\eta_2}{\eta_2}$, which reduces to the product $\mathfrak{r}_{3,-2}\times\mathbb{R}$ if $4\eta_1+\eta_2=0$. However the left-invariant metric is not the product one in the latter case. The algebraic Ricci solitons are expanding with soliton constant $\cARS = -(2\eta_1^2+\eta_2^2)$ and $\tau=-2(3\eta_1^2+\eta_2^2+2\eta_1\eta_2)<0$. Moreover, these metrics are $\mathcal{F}[t]$-critical with zero energy for 
		$t=-\frac{2\eta_1^2+\eta_2^2}{2(3\eta_1^2+\eta_2^2+2\eta_1\eta_2)}
		\in(-1,-\frac{1}{4})$.

		\medskip

		(iii) The underlying Lie algebra is $\mathfrak{r}_4$ and the algebraic Ricci solitons are expanding with  $\cARS = -3\eta_1^2$ and $\tau=-12\eta_1^2$. Moreover they are $\mathcal{F}[-1/4]$-critical with zero energy. 
		Furthermore, it follows from \cite{Kulkarni} that left-invariant metrics in (iii) are homothetic (although not isomorphically homothetic) to the corresponding left-invariant metrics with $\gamma_3=0$, and hence left-invariant metrics in case~(iii) are homothetic to those in case~(i) with $\eta_2=\eta_1$.
	\end{remark}

\begin{proof} 
	The conditions for $\mathfrak{D}=\Ricci-\cARS\Id$ to be a derivation are determined by a system of polynomial equations, $\{\mfP_{ijk}=0\}$, on the soliton constant $\cARS$ and  the structure constants in~\eqref{R3-eq: g_L.II ip}, where

	\medskip

	$
	\begin{array}{l}
	\mfP_{141} =     
	\ve \gamma_2^2 + 
	2  (\eta_1 - \eta_2) \gamma_2 \gamma_3 + (2 \eta_1^2 + 
	\eta_2^2 + \cARS) ( \gamma_1+\eta_1) ,
	\end{array}
	$

	\vspace{\vsep}

	$
	\begin{array}{l}
	\mfP_{142} =  
	-4  \ve \gamma_1^2 + 
	2  (\eta_1 - \eta_2) \gamma_3^2 + 
	2 \ve \gamma_2 \gamma_3 - 
	2  \ve (2 \eta_1 + \eta_2) \gamma_1 + \ve (2 \eta_1^2 + \eta_2^2 + \cARS)  ,
	\end{array}
	$

	\vspace{\vsep}

	$
	\begin{array}{l}
	\mfP_{143} =   
	-3 \ve \gamma_1 \gamma_2 - (\eta_1 - \eta_2) 
	\gamma_1 \gamma_3 - \ve (4 \eta_1 - \eta_2) 
	\gamma_2 - (5 \eta_1^2 + \eta_2^2 - 
	3 \eta_1 \eta_2 + \cARS) \gamma_3 ,
	\end{array}
	$

	\vspace{\vsep}

	$
	\begin{array}{l}
	\mfP_{241} =     
	2  (\eta_1 - \eta_2) \gamma_2^2 ,
	\end{array}
	$

	\vspace{\vsep}

	$
	\begin{array}{l}
	\mfP_{242} =  
	\ve  \gamma_2^2 + 
	2 (\eta_1 - \eta_2) \gamma_2 \gamma_3 - (\gamma_1 - \eta_1) (2 	 \eta_1^2 + \eta_2^2 + \cARS) ,
	\end{array}
	$

	\vspace{\vsep}

	$
	\begin{array}{l}
	\mfP_{243} =  
	\left( (\eta_1 - \eta_2) \gamma_1 - 5 \eta_1^2 - \eta_2^2 + 
	3 \eta_1 \eta_2 - \cARS \right)  \gamma_2 ,
	\end{array}
	$

	\vspace{\vsep}

	$
	\begin{array}{l}
	\mfP_{341} =   
	\left( (\eta_1 - \eta_2) (\gamma_1 + \eta_1) + 
	3 \eta_2^2 + \cARS \right) \gamma_2 ,
	\end{array}
	$

	\vspace{\vsep}

	$
	\begin{array}{l}
	\mfP_{342} =  
	-3 \ve \gamma_1 \gamma_2 - (\eta_1 - \eta_2) 
	\gamma_1 \gamma_3   - 
	3 \ve \eta_2 \gamma_2 + (\eta_1^2 + 
	3 \eta_2^2 - \eta_1 \eta_2 + \cARS) \gamma_3 ,
	\end{array}
	$
	
	\vspace{\vsep}
	
	$
	\begin{array}{l}
	\mfP_{343} =   
	-2 \ve  \gamma_2^2 - 
	4  (\eta_1 - \eta_2) \gamma_2 \gamma_3 + (2 \eta_1^2 + 
	\eta_2^2 + \cARS) \eta_2 .
	\end{array}
	$


	\medskip
	
	First of all note that if $\gamma_2\neq 0$ then $\mfP_{241}=0$ implies $\eta_2=\eta_1$. Now, a direct calculation shows that $\mfP_{242}+\frac{\gamma_1-\eta_1}{\gamma_2}\mfP_{341}=\ve \gamma_2^2$, which does not vanish. Hence, necessarily $\gamma_2=0$.
	
	Assuming $\gamma_2=0$, we compute
	\[
	\begin{array}{l}
	\mfP_{141} = (\eta_1+\gamma_1)(2\eta_1^2+\eta_2^2+\cARS),
	\\[\vsep]
	\mfP_{242} = (\eta_1-\gamma_1)(2\eta_1^2+\eta_2^2+\cARS),
	\\[\vsep]
	\mfP_{343} = \eta_2(2\eta_1^2+\eta_2^2+\cARS).
	\end{array}
	\]
	If $2\eta_1^2+\eta_2^2+\cARS\neq 0$ then $\eta_1=\eta_2=\gamma_1=0$ and the space is Einstein (see Remark~\ref{R3-re: L.II - loc symm}). Hence, $\cARS = -2\eta_1^2-\eta_2^2$ and
	the system $\{\mfP_{ijk}=0\}$ reduces to
	\begin{equation}\label{R3-II eq 1}
	\begin{array}{l}
	\mfP_{142} = 2(\eta_1-\eta_2)\gamma_3^2 - 2\ve (2\eta_1+\eta_2+2\gamma_1) \gamma_1=0,
	\\[\vsep]
	\mfP_{143} = - (\eta_1-\eta_2)(3\eta_1+\gamma_1)\gamma_3 =0,
	\\[\vsep]
	\mfP_{342} =  -(\eta_1-\eta_2)(\eta_1+2\eta_2+\gamma_1)\gamma_3 =0.
	\end{array}
	\end{equation}
	Next we consider the cases $\gamma_3=0$ and $\eta_2=\eta_1$, which we analyze separately. 
	
	\medskip
	
	\subsection*{Case 1: $\bm{\gamma_2=0}$,  $\bm{\cARS = -2\eta_1^2-\eta_2^2$}, $\bm{\gamma_3=0}$} In this case Equation~\eqref{R3-II eq 1} reduces to
	\[
	(2\eta_1+\eta_2+2\gamma_1) \gamma_1=0,
	\]
	so that we analyze the cases $\gamma_1=0$ and $\gamma_1= -\frac{2\eta_1+\eta_2}{2}\neq 0$.
	
	\medskip
	
	If $\gamma_1=0$, we get an algebraic Ricci soliton with soliton constant $\cARS = -2\eta_1^2-\eta_2^2$ and associated left-invariant metric given by
	\[
	[u_1,u_4] = \eta_1 u_1+\ve u_2,\quad
	[u_2,u_4] = \eta_1 u_2,\quad
	[u_3,u_4] = \eta_2 u_3.
	\] 
The isomorphic isometry	$u_4\mapsto -u_4$ interchanges $(\ve,\eta_1,\eta_2)$ and $(-\ve,-\eta_1,-\eta_2)$, so we may assume $\ve=1$. This space  is  Einstein if and only if $\eta_1=\eta_2=0$, and the same holds to   be locally symmetric (see Remark~\ref{R3-re: L.II - loc symm}). 

If $\eta_2=0$ and $\eta_1\neq 0$, then the underlying Lie algebra splits as a product 
$\mathfrak{k}\times\mathbb{R}$, where the subalgebra $\mathfrak{k}$ is spanned by $\{ u_1,u_2,u_4\}$ and $u_3$ determines a spacelike parallel left-invariant vector field. Hence the Lorentzian Lie group splits as a Lorentzian product and this situation is covered by the analysis in Section~\ref{se:3D}.

Finally, if $\eta_2\neq 0$, then the underlying Lie algebra is $\mathfrak{r}_{4,\lambda}$ with $\lambda=\frac{\eta_1}{\eta_2}$, and case~(i) is obtained.

	\bigskip

	If $\gamma_1= -\frac{2\eta_1+\eta_2}{2}\neq 0$  then the left-invariant metric 
	\[
	[u_1,u_4] = -\tfrac{\eta_2}{2}u_1+\ve u_2,\quad
	[u_2,u_4] = \tfrac{4\eta_1+\eta_2}{2} u_2,\quad
	[u_3,u_4] = \eta_2 u_3,
	\] 
	determines an algebraic Ricci soliton with soliton constant $\cARS = -2\eta_1^2-\eta_2^2$.  Exactly  as in the previous case, the isomorphic isometry $u_4\mapsto -u_4$ let us assume $\ve=1$.
The space is Einstein if and only if $\eta_1=\eta_2$, and it is locally symmetric whenever $\eta_2=0$ (see Remark~\ref{R3-re: L.II - loc symm}). 
		Next, we assume  $\eta_2(\eta_2-\eta_1)\neq 0$.
		Since $2\eta_1+\eta_2(=-2\gamma_1)\neq 0$
		considering the new basis
		$$
		\bar u_1=(2\eta_1+\eta_2)u_1-u_2,\quad \bar u_2=u_3,\quad \bar u_3=u_2,\quad \bar u_4=-\tfrac{2}{\eta_2}u_4,
		$$
		one has the brackets
		$$
		[\bar u_1,\bar u_4]=\bar u_1,\quad [\bar u_2,\bar u_4]=-2\bar u_2,\quad [\bar u_3,\bar u_4]=-\tfrac{4\eta_1+\eta_2}{\eta_2}\bar u_3,
		$$
		with $\langle\bar u_3,\bar u_3\rangle=0$.
	This shows that the underlying Lie algebra is $\mathfrak{r}_{4,-2,\lambda}$ with $\lambda=	-\tfrac{4\eta_1+\eta_2}{\eta_2}$. In the special case when $4\eta_1+\eta_2=0$, the Lie algebra splits as a product $\mathfrak{r}_{3,-2}\times\mathbb{R}$ but since $\bar u_3$ is a null vector, the left-invariant metric does not split as a product. This corresponds to case~(ii).

	\medskip
	
	\subsection*{Case 2: $\bm{\gamma_2=0}$,  $\bm{\cARS = -3\eta_1^2$}, $\bm{\gamma_3\neq 0}$, $\bm{\eta_2=\eta_1}$} Equation~\eqref{R3-II eq 1}   reduces to
	\[
	(3\eta_1+2\gamma_1)\gamma_1=0.
	\] 
	Note that if $3\eta_1+2\gamma_1=0$ then the space is Einstein (see Remark~\ref{R3-re: L.II - loc symm}). Now, if $\gamma_1=0$ and $\eta_1\neq 0$, then we obtain an algebraic Ricci soliton with soliton constant $\cARS = -3\eta_1^2$ and associated left-invariant metric given by 
	\[
	[u_1,u_4] = \eta_1 u_1+\ve u_2-\gamma_3 u_3,\quad
	[u_2,u_4] = \eta_1 u_2,\quad
	[u_3,u_4] = \gamma_3 u_2 +\eta_1 u_3.
	\]
As in the previous cases we can take $\ve=1$ using the isomorphic isometry 
		$u_4\mapsto -u_4$. Moreover $\operatorname{ad}_{u_4}$ has eigenvalues $0$ and $-\eta_1$, the latter with multiplicity three. Since $\eta_1\gamma_3\neq 0$, one has that $-\eta_1$ is a triple root of the minimal polynomial, and hence the underlying Lie algebra is $\mathfrak{r}_4$, corresponding  to  case~(iii). This  finishes the proof.
\end{proof}

\subsubsection{\bf The self-adjoint part of the derivation $\bm{D_{sad}}$ has a triple root}
There exists a  pseudo-orthonormal basis   $\{u_1$, $u_2$, $u_3\}$  of $\mathbb{R}^3$, with $\langle u_1,u_2\rangle=\langle u_3,u_3\rangle=1$, so that
\[
D_{sad}=\left(
\begin{array}{ccc}
\eta & 0 & 1
\\
0 & \eta & 0
\\
0 & 1 & \eta
\end{array}
\right),
\quad
D_{asad}=\left(
\begin{array}{ccc}
\gamma_1 & 0  & \gamma_2
\\
0  & -\gamma_1 & \gamma_3
\\
-\gamma_3 & -\gamma_2 & 0
\end{array}
\right).
\]
Therefore the corresponding left-invariant metrics are given by
\begin{equation}\label{R3-eq: g_L.III ip}
\mathfrak{g}_{L.III}
\left\{
\begin{array}{l}
[u_1,u_4]=(\eta+\gamma_1) u_1  - \gamma_3 u_3,  \quad
[u_2,u_4]= (\eta-\gamma_1) u_2 -  (\gamma_2-1) u_3,
\\
\noalign{\medskip}
[u_3,u_4]=(\gamma_2+1) u_1+\gamma_3 u_2 + \eta u_3, 
\end{array}
\right.
\end{equation}
where  $\{u_1,u_2,u_3, u_4\}$  is  a pseudo-orthonormal basis of    $\algsdR$, with $\langle u_1,u_2\rangle=\langle u_3,u_3\rangle=\langle u_4,u_4\rangle=1$.

\begin{remark}\label{R3-re: L.III symm}\rm
	A metric~\eqref{R3-eq: g_L.III ip} is Einstein if and only if  $\gamma_1=3\eta$ and  $\gamma_2=\gamma_3=0$.
	Moreover, it  is locally symmetric if and only if 
	$\eta=\gamma_1=\gamma_3=0$ and $\gamma_2=1$ 
	(in which case is locally isometric to a Cahen-Wallach symmetric space).
\end{remark}

\begin{theorem}\label{R3-th: g_L.III}
	A  left-invariant metric $\mathfrak{g}_{L.III}$ on $\sdR$  given by Equation~\eqref{R3-eq: g_L.III ip} 
	is a strict algebraic Ricci soliton
	if and only if it is isomorphically homothetic to 
	\begin{itemize}
		\item[]
		$[u_1,u_4]=\eta u_1$,
		$[u_2,u_4]= \eta u_2+u_3$,
		$[u_3,u_4]=u_1+\eta u_3$,\quad
		$\eta\neq 0$.

%
%
		
	\end{itemize}
	
	\smallskip
	
	\noindent 
	Here $\{u_i\}$ is a pseudo-orthonormal basis with $\langle u_1,u_2\rangle=\langle u_3,u_3\rangle=\langle u_4,u_4\rangle=1$. 
\end{theorem}

	\begin{remark}\rm\label{re:III}
The underlying Lie algebra in this case  is $\mathfrak{r}_4$ and the corresponding algebraic Ricci solitons are expanding with  $\cARS = -3\eta^2$ and $\tau=-12\eta^2$. Moreover these metrics are  $\mathcal{F}[-1/4]$-critical with zero energy.
	\end{remark}

\begin{proof} 
	A straightforward calculation  shows that the conditions for $\mathfrak{D}=\Ricci-\cARS\Id$ to be a derivation are determined by a system of polynomial equations on the soliton constant $\cARS$ and  the structure constants in~\eqref{R3-eq: g_L.III ip}, given by $\{\mfP_{ijk}=0\}$, where

	\medskip

	$
	\begin{array}{l}
	\mfP_{141} =  
	-\gamma_1 \gamma_3 + (3 \eta^2 + \cARS) (\gamma_1 + 
	\eta) + 3 \eta \gamma_3 ,
	
\end{array}
$

\vspace{\vsep}

$
\begin{array}{l}
\mfP_{143} =   
(3 \gamma_3 - 3 \eta^2 - \cARS) \gamma_3 ,
\end{array}
$

\vspace{\vsep}

$
\begin{array}{l}
\mfP_{241} =     
-6  (\gamma_1 - \eta) \gamma_2 ,
\end{array}
$

\vspace{\vsep}

$
\begin{array}{l}
\mfP_{242} =  
-\gamma_1 \gamma_3 - (3 \eta^2 + \cARS) (\gamma_1 - 
\eta ) + 3 \eta \gamma_3  ,
\end{array}
$

\vspace{\vsep}

$
\begin{array}{l}
\mfP_{243} =  
-\gamma_1^2 + 5 \gamma_2 \gamma_3 + 
3 \eta  \gamma_1 - (3 \eta^2 + \cARS) 
\gamma_2 - 3 \gamma_3 + 3 \eta^2 + \cARS ,
\end{array}
$

\vspace{\vsep}

$
\begin{array}{l}
\mfP_{341} =   
-\gamma_1^2 + 5 \gamma_2 \gamma_3 + 
3  \eta \gamma_1 + (3 \eta^2 + \cARS) 
\gamma_2 + 3 \gamma_3 + 3 \eta^2 + \cARS ,
\end{array}
$

\vspace{\vsep}

$
\begin{array}{l}
\mfP_{342} =  
(3 \gamma_3 + 3 \eta^2 + \cARS) \gamma_3 ,
\end{array}
$

\vspace{\vsep}

$
\begin{array}{l}
\mfP_{343} =   
2 \gamma_1 \gamma_3 - 
6  \eta \gamma_3 + (3 \eta^2 + \cARS) \eta.
\end{array}
$


\bigskip

Note that $\mfP_{143} + \mfP_{342} = 6\gamma_3^2$ and therefore necessarily $\gamma_3=0$. Now, using this condition, we have
\[
\mfP_{141}-\mfP_{242} = 2(3\eta^2+\cARS)\gamma_1,\quad
\mfP_{243}+\mfP_{341} = 2(3\eta^2+\cARS)-2(\gamma_1-3\eta)\gamma_1,
\]
which imply that $\cARS=-3\eta^2$, and the system $\{\mfP_{ijk}=0\}$ reduces to
\[
\mfP_{241}= -6(\gamma_1-\eta)\gamma_2=0,\quad
\mfP_{243}=\mfP_{341}=-(\gamma_1-3\eta)\gamma_1=0 .
\]

As a consequence,  $\gamma_1\gamma_2=0$, so that we must consider the case $\gamma_1=\gamma_2=0$, the case $\gamma_1=0$, $\gamma_2\neq 0$, $\eta=0$ and the case $\gamma_1=3\eta\neq 0$, $\gamma_2=0$. Since the latter case gives an Einstein space (see Remark~\ref{R3-re: L.III symm}), next we analyze the other two cases separately.

\medskip

\subsection*{Case 1: $\bm{\gamma_3=0}$, $\bm{\cARS=-3\eta^2}$, $\bm{\gamma_1=\gamma_2=0}$} In this case, the left-invariant metric, given by
\[
[u_1,u_4]=\eta u_1,\quad
[u_2,u_4]= \eta u_2+u_3,\quad
[u_3,u_4]=u_1+\eta u_3,
\] 
determines an algebraic Ricci soliton with soliton constant $\cARS=-3\eta^2$, 
being Einstein if and only if $\eta=0$ and which is never locally symmetric (see Remark~\ref{R3-re: L.III symm}). The underlying Lie algebra is $\mathfrak{r}_4$ in the non-Einsteinian case, thus   corresponding  to the case in the theorem.

\medskip

\subsection*{Case 2: $\bm{\gamma_3=0}$, $\bm{\cARS=0}$, $\bm{\gamma_1=0}$, $\bm{\gamma_2\neq 0}$, $\bm{\eta=0}$} 
The algebraic Ricci soliton, with soliton constant $\cARS=0$, has  associated left-invariant metric 
\[ 
[u_2,u_4]= -(\gamma_2-1)u_3,\quad
[u_3,u_4]=(\gamma_2+1)u_1 .
\] 
A straightforward calculation shows that $u_1$ determines a left-invariant parallel null vector field. Moreover the curvature tensor satisfies $R(x,y)=0$ for all $x,y\in u_1^\perp$, and the covariant derivatives $\nabla_xR=0$ for all $x\in u_1^\perp$. Hence it follows from \cite{Leistner} that the underlying structure is a plane wave corresponding to case~(b) in Section~\ref{sse:plane-wave-homog}, which finishes the proof.
\end{proof} 

\begin{remark}\rm\label{grupos-ext-R1}
The non-Einstein left-invariant plane wave metrics in Case 2 above are realized on the four-dimensional nilpotent Lie groups. The underlying Lie algebra is $\mathfrak{h}_3\times\mathbb{R}$ if $\gamma_2=\pm1$, and  $\mathfrak{n}_4$ otherwise.
\end{remark}

\subsection{Semi-direct extensions with degenerate Lie group $\pmb{\mathbb{R}}^{\bm{3}}$} \label{R3-se:degenerate}

Let  $\mathfrak{g}=\algsdR$ be a four-dimensional Lie algebra with a Lorentzian inner product $\ip$ which restricts to a degenerate inner product on the subalgebra $\mathbb{R}^3$.    
In this case there exists a pseudo-orthonormal basis $\{u_1,u_2,u_3,u_4\}$  of $\mathfrak{g}=\algsdR$, with $\langle u_1,u_1\rangle=\langle u_2,u_2\rangle=\langle u_3,u_4\rangle=1$, possibly after rotating the vectors $u_1$ and $u_2$, so that 
\begin{equation}\label{R3-eq: g_D ip}
\mathfrak{g}_D
\left\{
\begin{array}{l}
[u_1,u_4]=\gamma_1 u_1-\gamma_2 u_2 + \gamma_3 u_3, \qquad 
[u_2,u_4]=\gamma_2 u_1+\gamma_4 u_2 + \gamma_5 u_3,
\\
\noalign{\medskip}
[u_3,u_4]=\gamma_6 u_1+\gamma_7 u_2 + \gamma_8 u_3, 
\end{array}
\right.
\end{equation}
for certain  $\gamma_i\in\mathbb{R}$.

\begin{remark}\label{R3-re: D symm}\rm
	A metric~\eqref{R3-eq: g_D ip} is Einstein if and only if  $\gamma_6=\gamma_7 =0$ and, moreover,
	$
	\gamma_1^2  +  \gamma_4^2 - 
	(\gamma_1 + \gamma_4) \gamma_8 = 0
	$.
	Besides, it  is locally symmetric if and only if  $\gamma_6=\gamma_7 =0$ and  one of the following holds:
	\begin{itemize}
		\item[(a)] $
		\gamma_2 = \gamma_8 = 0$. Moreover it is Einstein if and only if, in addition,  $\gamma_1=\gamma_4=0$, in which case it is flat. Otherwise it is locally a Cahen-Wallach symmetric space.
		
		\item[(b)] $
		\gamma_2=\gamma_1=(\gamma_4-\gamma_8)\gamma_4=0$, $\gamma_8\neq 0$, in which case it is flat.
		
		\item[(c)] $
		\gamma_2= (\gamma_1-\gamma_4)\gamma_4=0$, $\gamma_8=\gamma_1\neq 0
		$, in which case it is flat.
		
		\item[(d)] $ \gamma_2\neq 0$, $\gamma_1-\gamma_4=(\gamma_1-\gamma_8)\gamma_1\gamma_8=0$. Moreover, it is Einstein if, in addition $(\gamma_1-\gamma_8)\gamma_1=0$, in which case it is flat. Otherwise it is locally a Cahen-Wallach symmetric space.
	\end{itemize}
\end{remark}

\begin{theorem}\label{R3-th: g_D}
No left-invariant metric $\mathfrak{g}_{D}$ on $\sdR$ given by Equation~\eqref{R3-eq: g_D ip} is a strict algebraic Ricci soliton.
\end{theorem}

\begin{proof} 
	The conditions for $\mathfrak{D}=\Ricci-\cARS\Id$ to be a derivation are determined by a system of polynomial equations, $\{\mfP_{ijk}=0\}$, on the soliton constant $\cARS$ and  the structure constants in~\eqref{R3-eq: g_D ip}, where 
	
	\medskip

	$
	\begin{array}{l}
	2\mfP_{141} =     
	-(3 \gamma_1 + \gamma_4) \gamma_6^2
	- \gamma_1 \gamma_7^2
	+3 \gamma_2   \gamma_6 \gamma_7 + 
	2 \gamma_1 \cARS ,
	\end{array}
	$

	\vspace{\vsep}

	$
	\begin{array}{l}
	2\mfP_{142} =  
	3 \gamma_2 \gamma_7^2 - 3 \gamma_1 \gamma_6 \gamma_7 -
	2 \gamma_2 \cARS ,
	\end{array}
	$

	\vspace{\vsep}

	$
	\begin{array}{l}
	2\mfP_{143} =   
	\gamma_3 \gamma_6^2
	+ \gamma_5 \gamma_6 \gamma_7
	-  \gamma_2 (2 \gamma_1 + 
	2 \gamma_4 - \gamma_8)   \gamma_7
	
	\\[\vsep]
	
	\phantom{2\mfP_{143}=}
	+ \left(2 \gamma_1^2 - \gamma_2^2 + \gamma_1 
	(\gamma_4 - 
	2 \gamma_8)  - \gamma_4 \gamma_8\right) 
	\gamma_6 + 2 \gamma_3 \cARS ,
\end{array}
$

\vspace{\vsep}

$
\begin{array}{l}
2\mfP_{241} =     
-3 \gamma_2 \gamma_6^2 - 3 \gamma_4 \gamma_6 \gamma_7 + 
2 \gamma_2 \cARS ,
\end{array}
$

\vspace{\vsep}

$
\begin{array}{l}
2\mfP_{242} =  
-\gamma_4 \gamma_6^2
- (\gamma_1 + 3 \gamma_4) \gamma_7^2
-3\gamma_2 \gamma_6 \gamma_7 + 
2 \gamma_4 \cARS ,
\end{array}
$

\vspace{\vsep}

$
\begin{array}{l}
2\mfP_{243} =  
\gamma_5 \gamma_7^2
+ \gamma_3 \gamma_6 \gamma_7
+ \gamma_2  (2 \gamma_1 + 2 \gamma_4 - \gamma_8) \gamma_6

\\[\vsep]
\phantom{2\mfP_{243} =}

- \left(
\gamma_2^2-(\gamma_1 + 
2 \gamma_4) (\gamma_4 - \gamma_8) \right) 
\gamma_7 + 2 \gamma_5 \cARS ,
\end{array}
$

\vspace{\vsep}

$
\begin{array}{l}
2\mfP_{341} =   
- (3 \gamma_6^2 + 3  \gamma_7^2 - 2  \cARS) 
\gamma_6,
\end{array}
$

\vspace{\vsep}

$
\begin{array}{l}
2\mfP_{342} =  
- (3 \gamma_6^2 + 3 \gamma_7^2 - 2 \cARS) \gamma_7 ,
\end{array}
$

\vspace{\vsep}

$
\begin{array}{l}
2\mfP_{343} =   
(2 \gamma_1 + \gamma_4 - \gamma_8) \gamma_6^2 + (\gamma_1 + 
2 \gamma_4 - \gamma_8) \gamma_7^2
+ 
2 \gamma_8 \cARS .
\end{array}
$


\bigskip

In view of the components above we study the case $\gamma_6=\gamma_7=0$, the case $\gamma_6=0$, $\gamma_7\neq 0$, and the case $\gamma_6\neq 0$ separately.

\medskip

\subsection*{Case 1: $\bm{\gamma_6=\gamma_7=0}$} 
A direct calculation shows that necessarily $\cARS=0$, since otherwise the metric would be flat, and the corresponding algebraic Ricci soliton is given by the left-invariant metric determined by
\[
[u_1,u_4] = \gamma_1 u_1 - \gamma_2 u_2 + \gamma_3 u_3,\quad
[u_2,u_4] = \gamma_2 u_1 + \gamma_4 u_2 + \gamma_5 u_3,\quad
[u_3,u_4] =  \gamma_8 u_3.
\]
The scalar curvature of any metric~\eqref{R3-eq: g_D ip} is $\tau=\frac{1}{2}(\gamma_6^2+\gamma_7^2)$. Hence the scalar curvature vanishes in the setting $\gamma_6=\gamma_7=0$. In the non-Einstein case the Ricci operator is two-step nilpotent and isotropic, since the only non-zero component is given by $\rho(u_4,u_4)=\gamma_1(\gamma_8- \gamma_1)+\gamma_4 (\gamma_8-\gamma_4)$. Moreover $u_3$ determines a null left-invariant vector field which is recurrent since $\nabla u_3=\omega\otimes u_3$, where $\omega(\cdot)=-\gamma_8\langle u_3,\,\cdot\,\rangle$. A straightforward calculation shows that $R(x,y)=0$ and $\nabla_xR=0$ for all $x,y\in u_3^\perp$, from where it follows that the Lorentzian Lie group determined by \eqref{R3-eq: g_D ip} with $\gamma_6=\gamma_7=0$ is a plane wave (cf. \cite{Leistner}), corresponding to case~(c) in Section~\ref{sse:plane-wave-homog} .
The only non-zero component of $\nabla\rho$ is given by
$(\nabla_{u_4}\rho)(u_4,u_4)=-2\gamma_8\rho(u_4,u_4)$. This shows that the Ricci tensor is parallel if and only if $\gamma_8=0$.

\medskip

\subsection*{Case 2: $\bm{\gamma_6=0}$, $\bm{\gamma_7\neq0}$} 
If $\gamma_6=0$, we have
\[
\mfP_{342} = -\tfrac{1}{2}(3\gamma_7^2-2\cARS)\gamma_7,\quad
\mfP_{241} = \gamma_2\cARS,\quad
\mfP_{141} = -\tfrac{1}{2}(\gamma_7^2-2\cARS)\gamma_1.
\]
Since $\gamma_7\neq 0$, it follows that 
$
\cARS = \tfrac{3}{2}\gamma_7^2
$ and $\gamma_1=\gamma_2=0$.
Now, 
\[
\mfP_{343} = (\gamma_4+\gamma_8)\gamma_7^2
\]
leads to $\gamma_8=-\gamma_4$, and the system $\{\mfP_{ijk}=0\}$ reduces to
\[
\mfP_{143} = \tfrac{3}{2}  \gamma_3\gamma_7^2,\quad
\mfP_{243} = 2(\gamma_4^2+\gamma_5\gamma_7)\gamma_7.
\]
Hence, $\gamma_3=0$ and $\gamma_5=-\frac{\gamma_4^2}{\gamma_7}$, and \eqref{R3-eq: g_D ip}  becomes
\[ 
[u_2,u_4]= \gamma_4 u_2 -\tfrac{\gamma_4^2}{\gamma_7} u_3,\quad
[u_3,u_4] = \gamma_7 u_2-\gamma_4 u_3,
\]
so that the Lie algebra splits as a product $\mathfrak{g}=\alg\times\mathbb{R}$ with $\alg=\operatorname{span}\{ u_2,u_3,u_4\}$.
A direct calculation now shows that $\nabla_{u_i} u_1 = 0$, $i=1,\dots,4$. Thus, $u_1$ is a spacelike parallel vector field and the Lorentzian Lie group reduces to a product covered by the results in Section~\ref{se:3D}.

\medskip

\subsection*{Case 3: $\bm{\gamma_6\neq 0}$} 
Since $\mfP_{341}=-\frac{1}{2}(3\gamma_6^2+3\gamma_7^2-2\cARS)\gamma_6$, the soliton constant is given by $\cARS=\frac{3}{2}(\gamma_6^2+\gamma_7^2)$ and, as a consequence, $\mfP_{142}=-\frac{3}{2}(\gamma_2\gamma_6+\gamma_1\gamma_7)\gamma_6$, so that $\gamma_2=-\frac{\gamma_1\gamma_7}{\gamma_6}$.
Now $\mfP_{141} =- \frac{1}{2}(\gamma_4\gamma_6^2+\gamma_1\gamma_7^2)$, 
which implies $\gamma_4=-\frac{\gamma_1\gamma_7^2}{\gamma_6^2}$.
Next, we compute $\mfP_{343}=\frac{\left( \gamma_1(\gamma_6^2-\gamma_7^2)+\gamma_6^2\gamma_8\right)(\gamma_6^2+\gamma_7^2)}{\gamma_6^2}$, 
which leads to
$\gamma_8=-\frac{(\gamma_6^2-\gamma_7^2)\gamma_1}{\gamma_6^2}$.
At this point, we have
\[
2\gamma_6^3 \mfP_{143} = 
(4 \gamma_6^2 + 3  \gamma_7^2) \gamma_6^3 \gamma_3 
+ 2  (\gamma_6^2 - \gamma_7^2) (2 \gamma_6^2 + 
\gamma_7^2)\gamma_1^2
+ \gamma_5 \gamma_6^4 \gamma_7 .
\]
Now using that the coefficient of $\gamma_3$ does not vanish, we can clear up this unknown, and using its expression the system $\{\mfP_{ijk}=0\}$ finally reduces to
\[
\mfP_{243} = 
\frac{6\left(\gamma_5\gamma_6^4-(\gamma_6^2-\gamma_7^2)\gamma_1^2\gamma_7\right)
(\gamma_6^2+\gamma_7^2)^2}
{(4 \gamma_6^2 + 3  \gamma_7^2) \gamma_6^4},
\]
from where $\gamma_5=\frac{(\gamma_6^2-\gamma_7^2)\gamma_1^2\gamma_7}{\gamma_6^4}$. Therefore, \eqref{R3-eq: g_D ip} reduces to
\[
\begin{array}{l}
[u_1,u_4]= \gamma_1 u_1 +\tfrac{\gamma_1\gamma_7}{\gamma_6}  u_2
- \tfrac{(\gamma_6^2-\gamma_7^2)\gamma_1^2}{\gamma_6^3} u_3,
\\[0.2cm]
[u_2,u_4] =  -\tfrac{\gamma_1\gamma_7}{\gamma_6} u_1 
- \tfrac{\gamma_1\gamma_7^2}{\gamma_6^2} u_2
+ \tfrac{(\gamma_6^2-\gamma_7^2)\gamma_1^2\gamma_7}{\gamma_6^4}u_3 ,

\\[0.2cm]
[u_3,u_4] = \gamma_6 u_1+\gamma_7 u_2
- \tfrac{(\gamma_6^2-\gamma_7^2)\gamma_1}{\gamma_6^2} u_3,
\end{array}
\]
which is a shrinking algebraic Ricci soliton with soliton constant $\cARS=\frac{3}{2}(\gamma_6^2+\gamma_7^2)$.
Since $\gamma_6(\gamma_6^2+\gamma_7^2)\neq 0$, considering the basis
$$
\bar u_1=u_3,\quad \bar u_2=\gamma_6 u_1+\gamma_7 u_2,\quad \bar u_3=\gamma_6\gamma_7 u_1-\gamma_6^2 u_2-2\gamma_1\gamma_7 u_3,\quad \bar u_4=2\gamma_1\gamma_7 u_2-\gamma_6^2 u_4,
$$
one has the brackets
$$
[\bar u_1,\bar u_4]=\gamma_1(\gamma_6^2-\gamma_7^2)\bar u_1 -\gamma_6^2\bar u_2,\quad
[\bar u_2,\bar u_4]=\tfrac{\gamma_1^2(\gamma_6^2-\gamma_7^2)^2}{\gamma_6^2}\bar u_1-\gamma_1(\gamma_6^2-\gamma_7^2)\bar u_2,
$$
which shows that the Lie algebra splits as a product $\mathfrak{g}=\alg\times\mathbb{R}$, where $\alg=\operatorname{span}\{ \bar u_1,\bar u_2,\bar u_4\}$. Moreover a straightforward calculation shows that  $\langle\bar u_3,\bar u_3\rangle=\gamma_6^2(\gamma_6^2+\gamma_7^2)$ and $\nabla_{\bar u_i} \bar u_3 = 0$, $i=1,\dots,4$. This shows that the Lorentzian Lie group is a Lorentzian product and thus covered by the analysis in Section~\ref{se:3D}, finishing the proof.
\end{proof}

\begin{remark}\rm\label{grupos-ext-R}
	Any non-Abelian semi-direct extension $\mathbb{R}^3\rtimes\mathbb{R}$ of the Abelian Lie group but  the $\mathfrak{S}$-type Lie group with Lie algebra $\mathfrak{r}_{4,1,1}$ admits non-Einstein left-invariant Lorentz metrics which are plane waves. 
	
	The case of nilpotent extensions was already considered in Remark~\ref{grupos-ext-R1}.
	
	We take $\gamma_6=\gamma_7=0$, and set $\xi=(\gamma_1,\gamma_2,\gamma_3,\gamma_4,\gamma_5,\gamma_8)$. 	
	The product Lie groups corresponding to $\mathfrak{r}_{3}\times\mathbb{R}$, $\mathfrak{r}_{3,\lambda}\times\mathbb{R}$ and $\mathfrak{r}'_{3,\lambda}\times\mathbb{R}$ admit non-Einstein left-invariant plane wave metrics given by $\mathfrak{g}_{D}$ for the special choice of $\xi=(-2,-1,1,0,0,0)$,
	$\xi=(-1,0,0,-\lambda,0,0)$, and 
	$\xi=(-\lambda+1,\frac{2}{\sqrt{2}},0,-\lambda-1,0,0)$, respectively.
	
	The Lie algebra $\mathfrak{r}_4$ admits a non-Einstein plane-wave left-invariant metric $\mathfrak{g}_{D}$ determined by $\xi=(-2,-1,1,0,0,-1)$. The left-invariant metrics corresponding to $\xi=(-\lambda,0,-1,-1,0,-\lambda)$ are realized on $\mathfrak{r}_{4,\lambda}$ for $\lambda\neq1$, and $\xi=(-2,1,0,0,0,-1)$ realizes on $\mathfrak{r}_{4,1}$.

	The Lie algebra $\mathfrak{r}_{4,\mu,\lambda}$ admits left-invariant non-Einstein plane wave metrics $\mathfrak{g}_{D}$ determined by
	$\xi=(-1,0,0,-\mu,0,-\lambda)$ if $\mu^2-\lambda\mu-\lambda+1\neq0$, and 
	$\xi=(-\mu,0,0,-\lambda,0,-1)$ if $\mu^2-\lambda\mu-\lambda+1=0$. This covers all the cases but $\mathfrak{r}_{4,1,1}$.
	The remaining Lie algebra $\mathfrak{r}'_{4,\mu,\lambda}$ admits left-invariant non-Einstein plane wave metrics $\mathfrak{g}_{D}$ determined by
	$\xi=(-\mu-2\lambda,\sqrt{(\lambda+\mu)^2+1},0,\mu,0,-\mu)$ for any value of $\lambda$ and $\mu\neq 0$.
\end{remark}

\section{Semi-direct extensions of the Heisenberg group $\mathcal{H}^3$}

We proceed as in the case of the Abelian Lie group, considering separately the cases when the induced metric on $\mathcal{H}^3$ is positive definite, Lorentzian or degenerate. Hence Theorem~\ref{th:main H3} follows at once from the analysis below.

\subsection{Semi-direct extensions with Riemannian normal subgroup $\bm{\mathcal{H}^3$}} \label{H3-se: H3 Riemann}

In this section we consider   left-invariant  Lorentzian metrics which are obtained as extensions 
of the three-dimensional Riemannian Heisenberg Lie group $\mathcal{H}^3$. 
Following Section~\ref{se:Lorentz-Riemannian} one may describe all such Lorentzian extensions considering an orthonormal basis $\{ e_i\}$ with $e_4$ timelike where the Lie brackets become
\begin{equation}\label{H3-eq: g_R ip}
\mathfrak{g}_R
\left\{
\begin{array}{ll}
[e_1,e_2]=\lambda_3 e_3,  &
[e_1,e_4]=\gamma_1 e_1-\gamma_2 e_2 + \gamma_3 e_3,  
\\
\noalign{\medskip}
[e_2,e_4]=\gamma_2 e_1 + \gamma_4  e_2 + \gamma_5 e_3,  &
[e_3,e_4]=(\gamma_1 +\gamma_4)e_3,
\end{array}
\right.
\end{equation}
where $\lambda_3\neq 0$ and $\gamma_1$, $\dots$, $\gamma_5\in\mathbb{R}$. 

\begin{remark}\label{H3-re: R - loc symm}\rm 	 
	A metric~\eqref{H3-eq: g_R ip}  is never either Einstein or locally symmetric.
\end{remark}

\begin{remark}\label{H3-re: R - paso a R3}\rm 	
	We introduce  a    parameter $\HsignRL$ which will be  used in what follows with  the purpose of facilitating the solution of the case in \S\ref{H3-se:Lorentz-1}
	proceeding exactly as in the present section. In the case at hand $\HsignRL = 1$, while in Section~\ref{H3-se:Lorentz-1} it takes the value $\HsignRL = -1$.
	
Considering the new basis
		$\bar e_1 = e_1$, $\bar e_2 = e_2$, $\bar e_3 = e_3$,
		$\bar e_4 =  \HsignRL\tfrac{ \gamma_5}{\lambda_3} e_1 - \HsignRL\tfrac{ \gamma_3}{\lambda_3} e_2  +e_4$,
		the Lie bracket transforms into
		\[
		[\bar e_1,\bar e_2] =   \HsignRL \lambda_3 \bar   e_3,
		\,\,
		[\bar e_1,\bar e_4] = \gamma_1 \bar e_1 -\gamma_2 \bar e_2,
		\,\,
		[\bar e_2, \bar e_4] =   \gamma_2 \bar e_1 + \gamma_4 \bar e_2 ,
		\,\,
		[\bar e_3, \bar e_4] = (\gamma_1+\gamma_4) \bar e_3,
		\]
		and a   direct  calculation shows that, when evaluating on the basis $\{\bar e_i\}$,
		\[ 
		\ad_{\bar e_4}=\left(
		\begin{array}{cc}
		A& 0 
		\\
		0 & \operatorname{tr}A
		\end{array}
		\right), 
		\quad \text{where}\,\, 
		A=\left(
		\begin{array}{cc}
		-\gamma_1 & -\gamma_2  
		\\
	\gamma_2 & -\gamma_4  
		\end{array}
		\right).
\qquad\qquad
		\]
		Hence,  the Lie algebra corresponds to $ \mathfrak{h}_3\times\mathbb{R} $ or $\mathfrak{n}_4$ if and only if $\gamma_4=-\gamma_1$ and 
		$\gamma_1^2=\gamma_2^2$ (cf. \cite{ABDO}). 
Indeed, setting $\gamma_2=\ve \gamma_1$, with $\ve^2=1$,   and considering the basis 
		\[
		\widetilde e_1 = -\HsignRL\tfrac{\gamma_5}{\lambda_3}   e_1 
		+ \HsignRL\tfrac{\gamma_3}{\lambda_3}   e_2  -  e_4,\quad
		\widetilde e_2 =   \HsignRL \lambda_3 e_3,\quad
		\widetilde e_3 = \ve e_1 - e_2,\quad
		\widetilde e_4 =  e_1,
		\]
		the Lie brackets transform into
	$[\widetilde e_1,\widetilde e_4] = \ve \gamma_1 \widetilde e_3$ and 
	$[\widetilde e_3, \widetilde e_4] =  \widetilde e_2$.
		Therefore,   we conclude that if $\gamma_4=-\gamma_1$ and $\gamma_1^2 = \gamma_2^2$ then the semi-direct extensions are almost Abelian and covered by Theorem~\ref{th:main R3}.

\end{remark}

\begin{theorem}\label{H3-th: g_R}
	Any  left-invariant metric $\mathfrak{g}_R$ on $\sdH$  given by Equation~\eqref{H3-eq: g_R ip} which is an algebraic Ricci soliton is also   realized on $\sdR$.  
\end{theorem}

\begin{proof} 
 The conditions for $\mathfrak{D}=\Ricci-\cARS\Id$ to be a derivation are determined by a system of polynomial equations  $\{\mfP_{ijk}=0\}$ on the soliton constant $\cARS$ and the structure constants in~\eqref{H3-eq: g_R ip}.
		We will make use of the following polynomials:

	\medskip

	\small 
	
	\noindent
	$
	\begin{array}{l}
	2 \mfP_{121} =
	\HsignRL (\gamma_1 \gamma_3 - 2 \gamma_2 \gamma_5 + 
	3 \gamma_3 \gamma_4) \lambda_3 ,
	\end{array}
	$

	\vspace{\vsep}
	
	\noindent
	$
	\begin{array}{l}
	2 \mfP_{122} =  
	\HsignRL (3 \gamma_1 \gamma_5 + 
	2 \gamma_2 \gamma_3 + \gamma_4 \gamma_5 ) \lambda_3,
	\end{array}
	$

	\vspace{\vsep}
	
	\noindent
	$
	\begin{array}{l}
	2 \mfP_{123} =    
	- \HsignRL  (3 \gamma_3^2 + 3 \gamma_5^2 - 3 \HsignRL \lambda_3^2 - 
	2 \cARS) \lambda_3 ,
	\end{array}
	$

	\vspace{\vsep}

	%
	%
	%
	%
	%
	%
	%
	
	\noindent
	$
	\begin{array}{l}
	2 \mfP_{142} =   
	2 \HsignRL  (5 \gamma_1^2 + \gamma_4^2) \gamma_2 + 
	3 \gamma_2 \gamma_3^2 + (4 \gamma_1 + \gamma_4) \gamma_3
	\gamma_5 - 2 \gamma_2 \cARS  ,
	\end{array}
	$

	%

	\vspace{\vsep}

	\noindent
	$
	\begin{array}{l}
	2 \mfP_{144} =
	- \HsignRL (\gamma_1 \gamma_5 + \gamma_2 \gamma_3) \lambda_3 ,
	\end{array}
	$

	%

	\vspace{\vsep}
	
	\noindent
	$
	\begin{array}{l}
	2 	\mfP_{241} =    
	-2 \HsignRL  (\gamma_1^2 + 5 \gamma_4^2) \gamma_2
	- 3 \gamma_2 \gamma_5^2 + (\gamma_1 + 
	4   \gamma_4  ) \gamma_3 \gamma_5 + 2 \gamma_2 \cARS ,
	\end{array}
	$

	\vspace{\vsep}
	
	\noindent
	$
	\begin{array}{l}
	2 \mfP_{242} =    
	-4 \HsignRL \left(\gamma_4^3 + \gamma_1^2 \gamma_4 - (\gamma_1  - 
	\gamma_4) \gamma_2^2  + \gamma_1 \gamma_4^2 \right)
	- \gamma_3^2 \gamma_4 + (3 \gamma_1  + \gamma_4) \gamma_5^2 + 
	3 \gamma_2 \gamma_3 \gamma_5 + 2 \gamma_4 \cARS,
	\end{array}
	$
	
	%
	
	\vspace{\vsep}
	
	\noindent
	$
	\begin{array}{l}
	2	\mfP_{244} =  
	- \HsignRL (\gamma_2 \gamma_5 - \gamma_3 \gamma_4) \lambda_3,
	\end{array}
	$

	\vspace{\vsep}
	
	\noindent
	$
	\begin{array}{l}
	2 \mfP_{341} =    
	-\gamma_2^2 \gamma_3 - 
	3  (\gamma_1 + \gamma_4) \gamma_2 \gamma_5 + (2 \gamma_1 + 
	3 \gamma_4) \gamma_3 \gamma_4 ,
	\end{array}
	$
	
	\vspace{\vsep}
	
	\noindent
	$
	\begin{array}{l}
	2 \mfP_{342} =    
	(3 \gamma_1^2 - \gamma_2^2) \gamma_5 + 
	3 ( \gamma_1 +  \gamma_4) \gamma_2 \gamma_3 + 
	2 \gamma_1 \gamma_4 \gamma_5   ,
	\end{array}
	$
	
	\vspace{\vsep}
	
	\noindent
	$
	\begin{array}{l}
	2 \mfP_{343} =   
	-4 \HsignRL (\gamma_1^2 + \gamma_4^2 + \gamma_1 \gamma_4) (\gamma_1 + 
	\gamma_4)
	- (3 \gamma_1 + 4 \gamma_4) \gamma_3^2 - (4 \gamma_1 + 
	3 \gamma_4 ) \gamma_5^2 + 2 (\gamma_1   + \gamma_4) \cARS .
	\end{array}
	$
	
	
	\normalsize
	
	\bigskip
	
	Since $\lambda_3\neq 0$, we may consider the orthogonal basis $\hat e_i=\frac{1}{\lambda_3} e_i$ and assume $\lambda_3=1$ from now on, working in the homothetic class of the initial metric. From $\mfP_{123}=0$ we have
	$\cARS =   \frac{3}{2} (  \gamma_3^2 + \gamma_5^2-\HsignRL)$, and a straightforward calculation shows that
	\[
	\mfP_{121}-2 \mfP_{244} = \tfrac{\HsignRL}{2}  (\gamma_1 + \gamma_4) \gamma_3,
	\quad
	\mfP_{122}+2 \mfP_{144} = \tfrac{\HsignRL}{2} (\gamma_1 + \gamma_4) \gamma_5  .
	\]
	Note that, if $\gamma_3=\gamma_5=0$, then 
	\[
	\mfP_{343} =  -\tfrac{\HsignRL}{2} \left( 4 (\gamma_1^2 +  \gamma_4^2 
	+    \gamma_1 \gamma_4) + 3 \right) (\gamma_1 + \gamma_4)   ,
	\]
	which vanishes if and only if $\gamma_1+\gamma_4=0$. Hence, in any case,  $\gamma_4=-\gamma_1$, which leads to
	\[
	\mfP_{341} =   \tfrac{1}{2}  (\gamma_1^2 - \gamma_2^2) \gamma_3 ,
	\quad
	\mfP_{342} = \tfrac{1}{2} (\gamma_1^2  -\gamma_2^2) \gamma_5 .
	\]
	Suppose      $\gamma_1^2 -\gamma_2^2\neq 0$, so that  $\gamma_3=\gamma_5=0$. Then, we have 
	\[
	\mfP_{242} =  \tfrac{\HsignRL}{2}  \left(4 \gamma_1^2 +8 \gamma_2^2  
	+ 3\right) \gamma_1 ,
	\]
	which vanishes if and only if $\gamma_1=0$, and this  implies
	\[
	\mfP_{142}-\mfP_{241} =  3 \HsignRL    \gamma_2 ,
	\]
	which leads to a contradiction since we are assuming 
	$\gamma_1^2 - \gamma_2^2 = -\gamma_2^2\neq 0$. Hence, necessarily $\gamma_1^2 -\gamma_2^2=0$.
	
	We have shown that 
	$
	\gamma_4=-\gamma_1$ and 
	$\gamma_1^2  - \gamma_2^2 = 0$,
	so    we conclude that the possible algebraic Ricci solitons can be also realized on   $\sdR$ (see Remark~\ref{H3-re: R - paso a R3}), finishing the proof.
\end{proof}

\subsection{Semi-direct extensions with Lorentzian normal subgroup $\bm{\mathcal{H}^3$}}\label{H3-se:EGR-Lorentzian}
We proceed as indicated in Section~\ref{se:Lorentz-Lorentz}. It was shown by Rahmani \cite{Rahmani} that there are three non-homothetic classes of left-invariant Lorentz metrics on $\mathcal{H}^3$ corresponding to structure operator of types Ia and II. If the structure operator is diagonalizable, one distinguishes the two cases corresponding to $\operatorname{ker}L$ being positive definite or of Lorentzian signature. If the structure operator is of type~II, then it is necessarily nilpotent.
We shall now analyze   these three cases separately.

\subsubsection{\bf The structure operator is diagonalizable of rank one with positive definite kernel}\label{H3-se:Lorentz-1}

In this case there exists an orthonormal basis $\{e_1,e_2,e_3,e_4\}$ of $\mathfrak{g}=\algsdH$, with $e_3$ timelike,  where  $\mathfrak{h}_3=\span\{e_1,e_2,e_3\}$ and $\mathbb{R}=\span\{e_4\}$, so that 
\begin{equation}\label{H3-eq: g_L.Ia+ ip}
\mathfrak{g}_{L.Ia+}
\left\{
\begin{array}{ll}
[e_1,e_2]=-\lambda_3 e_3,  &
[e_1,e_4]=\gamma_1 e_1 - \gamma_2 e_2 + \gamma_3 e_3,  
\\
\noalign{\medskip}
[e_2,e_4]=\gamma_2 e_1 + \gamma_4  e_2 + \gamma_5 e_3,  &
[e_3,e_4]=(\gamma_1 +\gamma_4)e_3,
\end{array}
\right.
\end{equation}
where $\lambda_3\neq 0$ and $\gamma_1$, $\dots$, $\gamma_5\in\mathbb{R}$. 

As in Section~\ref{H3-se: H3 Riemann}, metrics \eqref{H3-eq: g_L.Ia+ ip} are never Einstein nor locally symmetric. Moreover, the following result is obtained proceeding exactly as      in  the proof of Theorem~\ref{H3-th: g_R} with $\HsignRL=-1$, using Remark~\ref{H3-re: R - paso a R3}.

\begin{theorem}\label{H3-th: g_L.Ia+}
Any  left-invariant metric $\mathfrak{g}_{L.Ia+}$ on $\sdH$  given by Equation~\eqref{H3-eq: g_L.Ia+ ip} which is an algebraic Ricci soliton is also   realized on $\sdR$.  
\end{theorem}

\subsubsection{\bf The structure operator is diagonalizable of rank one and with Lorentzian kernel}\label{H3-se:Lorentz-2}

In this setting, it is possible to choose    an orthonormal basis $\{e_1,e_2,e_3,e_4\}$ of $\mathfrak{g}=\algsdH$, with $e_3$ timelike,  where  $\mathfrak{h}_3=\span\{e_1,e_2,e_3\}$ and $\mathbb{R}=\span\{e_4\}$, so that 
the left-invariant metrics are described by
\begin{equation}\label{H3-eq: g_L.Ia- ip}
\mathfrak{g}_{L.Ia-}
\left\{
\begin{array}{ll}
[e_1,e_3]=-\lambda_2 e_2,  &
[e_1,e_4]=\gamma_1 e_1+\gamma_2 e_2 + \gamma_3 e_3,  
\\
\noalign{\medskip}
[e_2,e_4]=\gamma_4 e_2,  &
[e_3,e_4]=\gamma_5 e_1+\gamma_6 e_2-(\gamma_1 -\gamma_4)e_3,
\end{array}
\right.
\end{equation}
where $\lambda_2\neq 0$ and $\gamma_1$, $\dots$, $\gamma_6\in\mathbb{R}$.

\begin{remark}\label{H3-re: L.Ia- - loc symm}\rm 	 
	A metric~\eqref{H3-eq: g_L.Ia- ip}  is never either Einstein or locally symmetric.
\end{remark}

\begin{remark}\label{H3-re: L.Ia- - paso a R3}\rm 
Considering the new basis given by
	$\bar e_1= e_1$,
	$\bar e_2 = e_3$,
	$\bar e_3 = e_2$ and
	$\bar e_4 = -\tfrac{\gamma_6}{\lambda_2} e_1+\tfrac{\gamma_2}{\lambda_2} e_3+e_4$,
	the Lie brackets in Equation~\eqref{H3-eq: g_L.Ia- ip} transform into
	\[
	\begin{array}{ll}
	[\bar e_1,\bar e_2] = -\lambda_2 \bar e_3,  
	&
	[\bar e_1,\bar e_4] = \gamma_1 \bar e_1 +\gamma_3 \bar e_2,
	\\
	\noalign{\medskip}
	[\bar e_2,\bar e_4] = \gamma_5 \bar e_1-(\gamma_1-\gamma_4)\bar e_2,
	&
	[\bar e_3,\bar e_4] = \gamma_4 \bar e_3.
	\end{array}
	\]
	Hence, a   direct calculation shows that, when evaluating on the basis $\{\bar e_i\}$,
		\[ 
	\ad_{\bar e_4}=\left(
	\begin{array}{cc}
	A& 0 
	\\
	0 & \operatorname{tr}A
	\end{array}
	\right), 
	\quad \text{where}\,\, 
	A=\left(
	\begin{array}{cc}
	-\gamma_1 & -\gamma_5  
	\\
	-\gamma_3 & \gamma_1-\gamma_4  
	\end{array}
	\right),
	\]
	which implies that the Lie algebra corresponds to $ \mathfrak{h}_3\times\mathbb{R} $ or $\mathfrak{n}_4$ if and only if $\gamma_4=0$ and $\gamma_1^2+\gamma_3 \gamma_5=0$
	(cf.~\cite{ABDO}).
  In this case,  if  $\gamma_3\neq 0$, the change of basis 
		\[
		\widetilde e_1 = -\tfrac{\gamma_6}{\lambda_2} e_1+\tfrac{\gamma_2}{\lambda_2} e_3+e_4 ,\quad
		\widetilde e_2 = -  \lambda_2  e_2,\quad
		\widetilde e_3 = \gamma_1  e_1 + \gamma_3 e_3,\quad
		\widetilde e_4 = -   e_1,
		\]
		transforms the Lie bracket into
	$[\widetilde e_1,\widetilde e_4] =   \widetilde e_3$ and $[\widetilde e_3, \widetilde e_4] =    \gamma_3 \widetilde e_2$,
		while if $\gamma_3=0$, which implies $\gamma_1=0$,  we use the new basis
		\[
		\widetilde e_1 = -\tfrac{\gamma_6}{\lambda_2} e_1+\tfrac{\gamma_2}{\lambda_2} e_3+e_4,\quad
		\widetilde e_2 =\lambda_2   e_2,\quad
		\widetilde e_3 =   e_1,\quad
		\widetilde e_4 = -   e_3,
		\]
		to get
	$[\widetilde e_1,\widetilde e_4] = \gamma_5 \widetilde e_3$ and 
	$[\widetilde e_3, \widetilde e_4] =    \widetilde e_2$.
		Therefore, we conclude that if $\gamma_4=0$ and $\gamma_1^2+\gamma_3 \gamma_5=0$ the corresponding left-invariant metrics are also realized on $\sdR$.
\end{remark}

\begin{theorem}\label{H3-th: g_L.Ia-}
	A  left-invariant metric $\mathfrak{g}_{L.Ia-}$ on $\sdH$  given by Equation~\eqref{H3-eq: g_L.Ia- ip}
 which is not realized on $\sdR$
	is a strict algebraic Ricci soliton
	if and only if it is isomorphically homothetic to  
	\[
	[e_1,e_3] =-e_2,\,\,
	[e_1,e_4] = \gamma_1 e_1 + \gamma_3 e_3,\,\,
	[e_2,e_4] = \gamma_4 e_2,\,\,
	[e_3,e_4] = -\gamma_3 e_1-(\gamma_1-\gamma_4)e_3,
	\]
	where $\gamma_3$ is the only positive solution of 
	$4\gamma_3^2 = 4 (\gamma_1^2  + \gamma_4^2 - \gamma_1 \gamma_4) + 3$ and $\{e_i\}$ is an orthonormal basis of the Lie algebra with $e_3$ timelike.
\end{theorem}

	\begin{remark}\rm\label{re:Heisenberg-(i)}
The Lie algebra underlying metrics in  Theorem~\ref{H3-th: g_L.Ia-} is $\mathfrak{d}'_{4,\lambda}$ with $\lambda=-\frac{\gamma_4}{\sqrt{3(\gamma_4^2+1)}}\in(-\frac{1}{\sqrt{3}},\frac{1}{\sqrt{3}})$. They are shrinking solitons with $\cARS=\frac{3}{2}$ and the scalar curvature is given by $\tau=-2(2\gamma_4^2-1)$, while $\|\rho\|^2=-3(2\gamma_4^2-1)$. Hence the scalar curvature vanishes for $\gamma_4=\pm\frac{1}{\sqrt{2}}$, in which case the corresponding metrics are $\mathcal{S}$-critical, but not $\mathcal{F}[t]$-critical for any $t\in\mathbb{R}$. 
In any other case the metrics in Theorem~\ref{H3-th: g_L.Ia-} are $\mathcal{F}[t]$-critical with zero energy for $ t=-\|\rho\|^2\tau^{-2}\in (-\infty,-\frac{3}{4}]\cup (0,+\infty)$.
\end{remark}

\begin{proof}
	The conditions for $\mathfrak{D}=\Ricci-\cARS\Id$ to be a derivation are determined by a system of polynomial equations, $\{\mfP_{ijk}=0\}$, on the soliton constant $\cARS$ and  the structure constants in~\eqref{H3-eq: g_L.Ia- ip}, where 
	\medskip

	$
	\begin{array}{l}
	2 \mfP_{122} = 
	(\gamma_1 \gamma_6 + \gamma_2 \gamma_3 + 
	3 \gamma_4 \gamma_6 ) \lambda_2 ,
	\end{array}
	$

	\vspace{\vsep}

	$
	\begin{array}{l}
	2 \mfP_{131} =   
	-(2 \gamma_1 \gamma_2 - 3 \gamma_2 \gamma_4 - 
	2 \gamma_5 \gamma_6) \lambda_2 ,
	\end{array}
	$

	\vspace{\vsep}

	$
	\begin{array}{l}
	2 \mfP_{132} = 
	- \left(3 ( \gamma_2^2 - \gamma_6^2 - \lambda_2^2) + 
	2 \cARS\right) \lambda_2 ,
	\end{array}
	$

	\vspace{\vsep}

	$
	\begin{array}{l}
	2 \mfP_{133} =  
	-(2 \gamma_1 \gamma_6 + 
	2 \gamma_2 \gamma_3 + \gamma_4 \gamma_6 ) \lambda_2 ,
	\end{array}
	$

	\vspace{\vsep}

	$
	\begin{array}{l}
	2 \mfP_{141} =    
	( 4 \gamma_1^2 + 2 \gamma_2^2 - 3 \gamma_3^2 + 4 \gamma_4^2 - 
	3 \gamma_5^2 - \gamma_6^2 - 4 \gamma_1 \gamma_4 - 
	2 \gamma_3 \gamma_5 + 2 \cARS  ) \gamma_1
	
	\\[\vsep]
	\phantom{2 \mfP_{141} =}
	
	- 3 \gamma_2^2 \gamma_4 + 
	3 \gamma_3^2 \gamma_4 - \gamma_4 \gamma_5^2 - \gamma_2 
	(\gamma_3 + 2  \gamma_5) \gamma_6 + 2 \gamma_3 \gamma_4 \gamma_5 ,
\end{array}
$

\vspace{\vsep}

$
\begin{array}{l}
2 \mfP_{142} =  
( 5 \gamma_1^2 + 3 \gamma_2^2 - 3 \gamma_3^2 + 3 \gamma_4^2 - 
3 \gamma_6^2 - 4 \gamma_1   \gamma_4 + 2 \gamma_3 \gamma_5 - 
3 \lambda_2^2 + 2 \cARS  ) \gamma_2

\\[\vsep]
\phantom{2 \mfP_{142} =}

- 3  \gamma_1 (\gamma_3 + \gamma_5) \gamma_6 + \gamma_3 
\gamma_4 \gamma_6 ,
\end{array}
$

\vspace{\vsep}

$
\begin{array}{l}
2  \mfP_{143} =    
( 8 \gamma_1^2 + 3 \gamma_2^2 - 3 \gamma_3^2 + 
3 \gamma_4^2 + \gamma_5^2 - 4 \gamma_1 \gamma_4 + 
2 \gamma_3 \gamma_5 + 2 \cARS  ) \gamma_3

\\[\vsep]
\phantom{2 \mfP_{143} =}

+ (4 \gamma_1^2 - \gamma_4^2) \gamma_5 + (3 \gamma_1  + 
\gamma_4) \gamma_2 \gamma_6 ,
\end{array}
$

\vspace{\vsep}

$
\begin{array}{l}
2 \mfP_{144} = 
(\gamma_1 \gamma_6 + \gamma_2 \gamma_3) \lambda_2,
\end{array}
$

\vspace{\vsep}

$
\begin{array}{l}
2 \mfP_{232} =   
(\gamma_1 \gamma_2 - 
4 \gamma_2 \gamma_4 - \gamma_5 \gamma_6) \lambda_2 ,
\end{array}
$

\vspace{\vsep}

$
\begin{array}{l}
2 	\mfP_{241} =   
-(\gamma_1^2  + 3  \gamma_4^2) \gamma_2 + 
4 \gamma_1 \gamma_2 \gamma_4 - \gamma_2 \gamma_3 \gamma_5 - 
3 \gamma_4 \gamma_5 \gamma_6 ,
\end{array}
$

\vspace{\vsep}

$
\begin{array}{l}
2 \mfP_{242} =   
( 4 \gamma_1^2 + 4 \gamma_2^2 - \gamma_3^2 + 
4 \gamma_4^2 - \gamma_5^2 - 3 \gamma_6^2 - 
4 \gamma_1 \gamma_4 + 2 \gamma_3 \gamma_5 + 
2 \cARS  ) \gamma_4

\\[\vsep]
\phantom{2 \mfP_{242} =}

- \gamma_1 (\gamma_2^2 + \gamma_6^2) - \gamma_2 (\gamma_3 - 
\gamma_5) \gamma_6 ,
\end{array}
$

\vspace{\vsep}

$
\begin{array}{l}
2 \mfP_{243} =    
\gamma_1^2 \gamma_6 + 2 \gamma_1 \gamma_4 \gamma_6 + 
3 \gamma_2 \gamma_3 \gamma_4 + \gamma_3 \gamma_5 \gamma_6 ,
\end{array}
$

\vspace{\vsep}

$
\begin{array}{l}
2	\mfP_{341} =   
( 8 \gamma_1^2 + \gamma_3^2 + 7 \gamma_4^2 - 3 \gamma_5^2 - 
3 \gamma_6^2 - 12 \gamma_1 \gamma_4 + 2 \gamma_3 \gamma_5 + 
2 \cARS ) \gamma_5

\\[\vsep]
\phantom{2 \mfP_{341} =}

+ (4 \gamma_1^2 + 3   \gamma_4^2) \gamma_3 + 
3 \gamma_1 \gamma_2 \gamma_6 - 8 \gamma_1 \gamma_3 \gamma_4 - 
4 \gamma_2 \gamma_4 \gamma_6 ,
\end{array}
$

\vspace{\vsep}

$
\begin{array}{l}
2 \mfP_{342} =     
( 5 \gamma_1^2 + 3 \gamma_2^2 + 4 \gamma_4^2 - 3 \gamma_5^2 - 
3 \gamma_6^2 - 6 \gamma_1 \gamma_4 + 2 \gamma_3 \gamma_5 - 
3 \lambda_2^2 + 2 \cARS  ) \gamma_6

\\[\vsep]
\phantom{2 \mfP_{342} =}

+ 3 \gamma_1 \gamma_2 (\gamma_3 + \gamma_5) - \gamma_2  (3 
\gamma_3 + 2 \gamma_5) \gamma_4 ,
\end{array}
$

\vspace{\vsep}

$
\begin{array}{l}
2 \mfP_{343} =  
- ( 4 \gamma_1^2 + \gamma_2^2 - 3 \gamma_3^2 + 8 \gamma_4^2 - 
3 \gamma_5^2 - 2 \gamma_6^2 - 8 \gamma_1 \gamma_4 - 
2 \gamma_3 \gamma_5 + 2 \cARS  ) \gamma_1

\\[\vsep]
\phantom{2 \mfP_{343} =}

+ (\gamma_2^2  - 4 \gamma_3^2 + 4 \gamma_4^2 + \gamma_6^2 + 
2 \cARS) \gamma_4
+ \gamma_2 (2 \gamma_3 + \gamma_5) \gamma_6 ,
\end{array}
$

\vspace{\vsep}

$
\begin{array}{l}
2 \mfP_{344} =    	
-(\gamma_1 \gamma_2 - \gamma_2 \gamma_4 - \gamma_5 \gamma_6) 
\lambda_2 .
\end{array}
$


\normalsize

\bigskip

Since $\lambda_2\neq 0$, we may assume $\lambda_2=1$ from now on working in the homothetic class of the initial metric, just considering  the orthogonal basis $\hat e_i=\frac{1}{\lambda_2} e_i$.  We start taking the ideal $\mathcal{I}_1= \langle 
\mfP_{ijk} \cup \{ \lambda_2-1   \}\rangle$ in the polynomial ring 
$\mathbb{R}[ \gamma_1,\dots,\gamma_6,\lambda_2,\cARS]$ and computing a Gröbner basis for this ideal with respect to the lexicographical order. As a consequence, we obtain a set of $26$ polynomials, among which we find
\[
\mathbf{g}_{11} = (2\cARS-3)\gamma_4,
\quad
\mathbf{g}_{12} = \gamma_2\gamma_4,
\quad
\mathbf{g}_{13} = \gamma_6 \gamma_4,
\quad
\mathbf{g}_{14} = (\gamma_3+\gamma_5)(\gamma_4^2+1)\gamma_4.
\]
Now, if $\gamma_4=0$, note that the metric is also realized on $\sdR$ whenever $\gamma_1^2+\gamma_3\gamma_5=0$ (see Remark~\ref{H3-re: L.Ia- - paso a R3}). Hence,  
we introduce and auxiliary variable $\varphi$ to indicate that $\gamma_1^2+\gamma_3\gamma_5\neq0$ by means of the polynomial $(\gamma_1^2+\gamma_3\gamma_5)\varphi-1$. Let $\mathcal{I}_2$ be the ideal generated by 
$\mathcal{I}_1\cup \{ \gamma_4,(\gamma_1^2+\gamma_3\gamma_5)\varphi-1 \}$ in the polynomial ring $\mathbb{R}[ \varphi, \gamma_1,\dots,\gamma_6,\lambda_2,\cARS]$, where we consider again the lexicographic order. Computing a Gröbner basis for $\mathcal{I}_2$ we get $8$ polynomials, including
\[
\mathbf{g}_{21} = 2\cARS-3,
\quad
\mathbf{g}_{22} = \gamma_2,
\quad
\mathbf{g}_{23} = \gamma_6 ,
\quad
\mathbf{g}_{24} = \gamma_3+\gamma_5 .
\]
Hence, it follows that if a left-invariant metric~\eqref{H3-eq: g_L.Ia- ip} is an algebraic Ricci soliton which is not realized in $\sdR$ then necessarily 
\[
\cARS=\tfrac{3}{2},\quad
\gamma_2=\gamma_6=0,\quad
\gamma_5=-\gamma_3,
\]
and a straightforward calculation shows that, in that case, the system of polynomial equations $\{\mfP_{ijk}=0\}$ is determined by
\[
\mfP_{141} = \tfrac{\gamma_1}{2} \zeta=0,\quad
\mfP_{143} = -\mfP_{341} = \tfrac{\gamma_3}{2} \zeta=0,\quad
\mfP_{242} = \tfrac{\gamma_4}{2} \zeta=0 ,
\]
where $\zeta=-4 \gamma_3^2 +  4 (\gamma_1^2  + \gamma_4^2 - \gamma_1 \gamma_4) + 3$.
Using again Remark~\ref{H3-re: L.Ia- - paso a R3}, if $\zeta\neq 0$ then the left-invariant metric is also realized on $\sdR$. Hence, we take $\zeta=0$ to obtain an algebraic Ricci soliton with soliton constant $\cARS=\frac{3}{2}$ and associated left-invariant metric given by
\[
[e_1,e_3] =-e_2,\,\,\,
[e_1,e_4] = \gamma_1 e_1 + \gamma_3 e_3,\,\,\,
[e_2,e_4] = \gamma_4 e_2,\,\,\,
[e_3,e_4] = -\gamma_3 e_1-(\gamma_1-\gamma_4)e_3,
\]
where  $\gamma_3=\frac{\ve}{2}\sqrt{4 (\gamma_1^2  + \gamma_4^2 - \gamma_1 \gamma_4) + 3}$. The isomorphic isometry $(e_1,e_2,e_3,e_4)\mapsto (e_1,-e_2,-e_3,e_4)$  interchanges the sign of $\ve$, so that we may take $\ve=1$. Note that the space   is not realized on $\sdR$, since if $\gamma_4=0$ then   
\[
\gamma_1^2+\gamma_3 \gamma_5 = \gamma_1^2 -\gamma_3^2=\tfrac{\zeta-3}{4}=-\tfrac{3}{4}\neq 0
\]
(see Remark~\ref{H3-re: L.Ia- - paso a R3}). Finally, a straightforward calculation shows that
the Ricci operator is diagonalizable with non-zero eigenvalues
\[
\tfrac{1}{2}
\Big\{
-4\gamma_4^2-1,\,
-2\gamma_4^2+1\pm  2 \gamma_4 \sqrt{3(\gamma_4^2+1)}\, \sqrt{-1}   ,\, 
3
\Big\} .
\]
Hence, although the scalar curvature $\tau=-2(2\gamma_4^2-1)$ may vanish,  the space is not a pp-wave.
Moreover, the  underlying Lie algebra is $\mathfrak{d}'_{4,\lambda}$ with $\lambda=-\frac{\gamma_4}{\sqrt{3(\gamma_4^2+1)}}$, 
  corresponding  to  the case given in the theorem. This  finishes the proof.
\end{proof}

\subsubsection{\bf  The structure operator is 2-step nilpotent}\label{H3-se:Lorentz-3} 

In this   case  there exists a pseudo-orthonormal basis $\{u_1,u_2,u_3,u_4\}$ of $\mathfrak{g}=\algsdH$, with $\langle u_1,u_2\rangle=\langle u_3,u_3\rangle=\langle u_4,u_4\rangle=1$,  where  $\mathfrak{h}_3=\span\{u_1,u_2,u_3\}$ and $\mathbb{R}=\span\{u_4\}$, so that 
\begin{equation}\label{H3-eq: g_L.II ip}
\mathfrak{g}_{L.II}
\left\{
\begin{array}{ll}
[u_1,u_3]=-\varepsilon u_2,&
[u_1,u_4]=\gamma_1 u_1+\gamma_2 u_2+\gamma_3 u_3,   

\\
\noalign{\medskip}

[u_2,u_4]=\gamma_4 u_2, &
[u_3,u_4]=\gamma_5 u_1+\gamma_6 u_2 - (\gamma_1-\gamma_4)u_3 ,
\end{array}
\right.
\end{equation}
with $\varepsilon^2=1$ and $\gamma_1$, $\dots$, $\gamma_6\in\mathbb{R}$.

\begin{remark}\label{H3-re: L.II - loc symm}\rm 	 
A metric~\eqref{H3-eq: g_L.II ip}  is Einstein if and only if
either $\gamma_1=\gamma_4=\gamma_5=0$, $\gamma_6=\pm\gamma_3$ (in which case it is Ricci-flat), or $\gamma_4=3\gamma_1\neq 0$, $\gamma_2=\gamma_5=0$, $\gamma_6=-\gamma_3$ (in which case  it is of constant sectional curvature).
Moreover, a metric~\eqref{H3-eq: g_L.II ip} is locally symmetric if and only if either
$\gamma_1=\gamma_4=\gamma_5=\gamma_3(\gamma_3+\gamma_6)=0$ (in which case it is locally conformally flat), or
$\gamma_4=3\gamma_1\neq 0$, $\gamma_2=\gamma_5=0$, $\gamma_6=-\gamma_3$  (in which case it is Einstein).
\end{remark}

\begin{remark}\label{H3-re: L.II - paso a R3}\rm 
The brackets in \eqref{H3-eq: g_L.II ip} are the same as those in \eqref{H3-eq: g_L.Ia- ip} just interchanging the basis $\{ e_i\}$ by $\{ u_i\}$ and replacing the parameter $\lambda_2$ by $\varepsilon$.
Hence, proceeding exactly as in Remark~\ref{H3-re: L.Ia- - paso a R3} (setting $\lambda_2=\ve)$, we get that 	  the underlying Lie algebra is almost Abelian, thus corresponding to $ \mathfrak{h}_3\times\mathbb{R} $ or $\mathfrak{n}_4$,  if and only if $\gamma_4=0$ and $\gamma_1^2+\gamma_3 \gamma_5=0$.
\end{remark}

\begin{theorem}\label{H3-th: g_L.II}
A left-invariant metric $\mathfrak{g}_{L.II}$ on $\sdH$  given by Equation~\eqref{H3-eq: g_L.II ip}
 which is not realized on $\sdR$
is a strict algebraic Ricci soliton
if and only if it is isomorphically homothetic to  
\[
\begin{array}{ll}
[u_1,u_3] = - u_2,
&
[u_1,u_4] = \gamma_1 u_1 + \gamma_2 u_2 + 3 \gamma_6 u_3,

\\
\noalign{\medskip}

[u_2,u_4] = -\gamma_1 u_2,
&
[u_3,u_4] = \gamma_6 u_2 - 2\gamma_1 u_3 ,
\end{array}
\]  
where $\gamma_1\neq 0$ and $\{u_i\}$ is a pseudo-orthonormal basis with $\langle u_1,u_2\rangle=\langle u_3,u_3\rangle=\langle u_4,u_4\rangle=1$.
\end{theorem}

\begin{remark}\rm\label{re:Heisenberg-(ii)}
Since $\gamma_1\neq 0$, considering the basis
$$
\bar u_1=u_1+\tfrac{\gamma_6}{\gamma_1}u_3,\quad
\bar u_2=-u_3,\quad
\bar u_3=u_2,\quad
\bar u_4=-\gamma_6 u_1+\gamma_2 u_3+u_4 ,
$$
one has
$[\bar u_1,\bar u_2]=\bar u_3$, 
$[\bar u_1,\bar u_4]=\gamma_1 \bar u_1$,
$[\bar u_2,\bar u_4]=-2\gamma_1\bar u_2$ and
$[\bar u_3,\bar u_4]=-\gamma_1\bar u_3$,
which shows that the underlying Lie algebra is $\mathfrak{d}_{4,2}$.

The scalar curvature of metrics in Theorem~\ref{H3-th: g_L.II} is $\tau=-8\gamma_1^2$ and they are expanding algebraic Ricci solitons with
$\cARS = -4\gamma_1^2$. Moreover, these metrics are $\mathcal{F}[t]$-critical with zero energy for 
$t=-\frac{1}{2}$.
\end{remark}

\begin{proof}
A straightforward calculation  shows that the conditions for $\mathfrak{D}=\Ricci-\cARS\Id$ to be a derivation are determined by a system of polynomial equations on the soliton constant $\cARS$ and the structure constants in~\eqref{H3-eq: g_L.II ip}, given by $\{\mfP_{ijk}=0\}$, where

\medskip

\small

\noindent
$
\begin{array}{l}
2 \mfP_{122} = 
-\ve (\gamma_1 + 3 \gamma_4) \gamma_5 ,
\end{array}
$

\vspace{\vsep}

\noindent
$
\begin{array}{l}
2 \mfP_{131} =   
-2  \ve \gamma_5^2  ,
\end{array}
$

\vspace{\vsep}

\noindent
$
\begin{array}{l}
2 \mfP_{132} =  
-\ve (4 \gamma_4^2 - 4 \gamma_1 \gamma_4 + 
3 \gamma_5 \gamma_6 + 2 \cARS) ,
\end{array}
$

\vspace{\vsep}

\noindent
$
\begin{array}{l}
2 \mfP_{133} =   
\ve  (2 \gamma_1 + \gamma_4) \gamma_5  ,
\end{array}
$

\vspace{\vsep}

\noindent
$
\begin{array}{l}
2 \mfP_{141} =    
( 3 \gamma_1^2 + 3 \gamma_4^2 - 2 \gamma_1 \gamma_4 + 
4 \gamma_5 \gamma_6 + 2 \cARS  ) \gamma_1
+ (2 \gamma_2 \gamma_5 + \gamma_4 \gamma_6) \gamma_5 ,
\end{array}
$

\vspace{\vsep}

\noindent
$
\begin{array}{l}
2 \mfP_{142} =   
(\gamma_1^2 + 5 \gamma_4^2 - 
2 \gamma_1 \gamma_4 + \gamma_3 \gamma_5 + 
3 \gamma_5 \gamma_6 + 2 \cARS ) \gamma_2 
- \gamma_3^2 \gamma_4 - 
3 \gamma_1  (\gamma_3 - \gamma_6) \gamma_6 + \gamma_3 \gamma_4 
\gamma_6 ,
\end{array}
$

\vspace{\vsep}

\noindent
$
\begin{array}{l}
2  \mfP_{143} =    
( 5 \gamma_1^2 + 3 \gamma_4^2 - \gamma_1  \gamma_4 +  
2 \gamma_3 \gamma_5 - \gamma_5 \gamma_6 + 
2 \cARS) \gamma_3 
- (4 \gamma_1^2  - \gamma_4^2) \gamma_6 - (3 \gamma_1    + 
\gamma_4) \gamma_2  \gamma_5 ,
\end{array}
$

\vspace{\vsep}

\noindent
$
\begin{array}{l}
2 \mfP_{144} =  
- \ve \gamma_1 \gamma_5  ,
\end{array}
$

\vspace{\vsep}

\noindent
$
\begin{array}{l}
2 \mfP_{232} =    
\ve \gamma_5^2  ,
\end{array}
$

\vspace{\vsep}

\noindent
$
\begin{array}{l}
2 	\mfP_{241} =    	
3 \gamma_4 \gamma_5^2 ,
\end{array}
$

\vspace{\vsep}

\noindent
$
\begin{array}{l}
2 \mfP_{242} =    
(3 \gamma_1^2  + 3 \gamma_4^2 - 2 \gamma_1  \gamma_4 + 
2 \gamma_3 \gamma_5 + 4 \gamma_5 \gamma_6 + 
2 \cARS) \gamma_4
+ (\gamma_1  \gamma_6 - \gamma_2 \gamma_5 ) \gamma_5 ,
\end{array}
$

\vspace{\vsep}

\noindent
$
\begin{array}{l}
2 \mfP_{243} =   
- (\gamma_1^2 + 
2 \gamma_1 \gamma_4 + \gamma_3 \gamma_5) \gamma_5 ,
\end{array}
$

\vspace{\vsep}

\noindent
$
\begin{array}{l}
2	\mfP_{341} =    
(5 \gamma_1^2 + 3 \gamma_4^2 - 5 \gamma_1 \gamma_4 + 
2 \gamma_3 \gamma_5 + 6 \gamma_5 \gamma_6 + 
2 \cARS) \gamma_5 ,
\end{array}
$

\vspace{\vsep}

\noindent
$
\begin{array}{l}
2 \mfP_{342} =    
(5 \gamma_1^2  + 5 \gamma_4^2 - 7 \gamma_1 \gamma_4  + 
2 \gamma_3 \gamma_5  + 6 \gamma_5 \gamma_6 + 
2 \cARS) \gamma_6 
- (\gamma_1^2  - 
2 \gamma_1 \gamma_4 + \gamma_3 \gamma_5)  \gamma_3 ,
\end{array}
$

\vspace{\vsep}

\noindent
$
\begin{array}{l}
2 \mfP_{343} = 
- ( 3 \gamma_1^2 + 5 \gamma_4^2 - 5 \gamma_1 \gamma_4 + 
5 \gamma_5 \gamma_6 + 2 \cARS) \gamma_1 
+ (3 \gamma_4^2 + 2 \gamma_3 \gamma_5 - \gamma_5 \gamma_6 + 
2 \cARS) \gamma_4
- \gamma_2 \gamma_5^2 ,
\end{array}
$

\vspace{\vsep}

\noindent
$
\begin{array}{l}
2 \mfP_{344} =
- \ve \gamma_5^2 .
\end{array}
$


\normalsize

\bigskip

First of all note that the isomorphic isometry $u_3\mapsto -u_3$ interchanges the sign of $\ve$, $\gamma_3$, $\gamma_5$ and $\gamma_6$. Hence, without loss of generality, we may assume $\ve=1$ from now on. Moreover, from $\mfP_{131}=0$ and $\mfP_{132}=0$ we get
\[
\gamma_5=0, \quad
\cARS = 2(\gamma_1-\gamma_4)\gamma_4 ,
\]
which lead to
\[
\mfP_{141} = \tfrac{1}{2} (\gamma_1+\gamma_4)(3\gamma_1-\gamma_4)\gamma_1,
\quad
\mfP_{242} = \tfrac{1}{2} (\gamma_1+\gamma_4)(3\gamma_1-\gamma_4)\gamma_4 .
\] 
Therefore, it is enough to consider the cases $\gamma_4=-\gamma_1$ and $\gamma_4=3\gamma_1\neq 0$. Next we analyze them separately.

\medskip

\subsection*{Case 1: $\bm{\gamma_4=-\gamma_1}$} 
In this case,  a direct calculation shows that the system $\{ \mfP_{ijk}=0\}$ reduces to
\[
\mfP_{142} = (\gamma_3-\gamma_6)(\gamma_3-3\gamma_6)\gamma_1 = 0,
\quad
\mfP_{143} =-\tfrac{1}{3} \mfP_{342} = (\gamma_3-3\gamma_6)\gamma_1^2 = 0 .
\]

\smallskip

If $\gamma_1=0$ then the left-invariant metric is also realized on $\sdR$ (see Remark~\ref{H3-re: L.II - paso a R3}). Otherwise, if $\gamma_1\neq 0$, we take $\gamma_3=3\gamma_6$ to obtain an algebraic Ricci soliton with soliton constant $\cARS=-4\gamma_1^2$ and associated left-invariant metric given by
\[
\begin{array}{ll}
[u_1,u_3] = - u_2,
&
[u_1,u_4] = \gamma_1 u_1 + \gamma_2 u_2 + 3 \gamma_6 u_3,

\\
\noalign{\medskip}

[u_2,u_4] = -\gamma_1 u_2,
&
[u_3,u_4] = \gamma_6 u_2 - 2\gamma_1 u_3 .
\end{array}
\]
Note that the space is not realized on $\sdR$, since $\gamma_4=-\gamma_1\neq 0$ (see Remark~\ref{H3-re: L.II - paso a R3}). This  corresponds   to  the case given in the theorem.

\medskip

\subsection*{Case 2: $\bm{\gamma_4=3\gamma_1\neq 0}$} 
The system $\{\mfP_{ijk}=0\}$ reduces to
\[
\mfP_{142} = -(3\gamma_3^2-3\gamma_6^2-16\gamma_1\gamma_2)\gamma_1 = 0,
\quad
\mfP_{143} =  \mfP_{342} = 5(\gamma_3+\gamma_6)\gamma_1^2 = 0 ,
\]
from where we obtain $\gamma_6=-\gamma_3$ and $\gamma_2=0$.
Hence, the space is Einstein (see Remark~\ref{H3-re: L.II - loc symm}), which finishes the proof.
\end{proof}

\subsection{Semi-direct extensions with degenerate normal subgroup $\bm{\mathcal{H}^3}$}\label{H3-se:2}

In this section we analyze left-invariant Lorentzian metrics which are extensions of the three-dimensional unimodular Lie group $\mathcal{H}^3$ equipped with a degenerate metric. 
Hence, let $\mathfrak{g}=\algsdH$ be a four-dimensional Lie algebra with a Lorentzian inner product $\ip$  which restricts to a degenerate inner product on   $\mathfrak{h}_3$. We proceed as in case~(ii) of Section~\ref{se:Lorentz-degenerate} and consider the two-distinct situations when the restriction of the metric to the derived subalgebra $\mathfrak{h}'_3$ is degenerate or spacelike.

\subsubsection{\bf $\bm{{\mathfrak{h}_3'=\operatorname{\bf span}\{v\}}}$ is a null subspace}\label{H3-se:deg-null subspace}

In this case, setting $u_3=v$ we can take a pseudo-orthonormal basis $\{u_1,u_2,u_3,u_4\}$ of $\mathfrak{g}=\algsdH$, with $\langle u_1,u_1\rangle=\langle u_2,u_2\rangle=\langle u_3,u_4\rangle=1$,  where  $\mathfrak{h}_3=\span\{u_1,u_2,u_3\}$ and $\mathbb{R}=\span\{u_4\}$, so that  
\begin{equation}\label{H3-eq: g_D0 ip}
\mathfrak{g}_{D0}
\left\{
\begin{array}{ll}
[u_1,u_2]= \lambda_1 u_3 , &
[u_1,u_4]= \gamma_1 u_1 -  \gamma_2 u_2 + \gamma_3 u_3 ,
\\
\noalign{\medskip}
[u_2,u_4]=  \gamma_2 u_1 + \gamma_4 u_2 + \gamma_5 u_3 , &
[u_3,u_4]=(\gamma_1+\gamma_4) u_3 ,
\end{array}
\right.
\end{equation}
where  $\lambda_1\neq 0$ and    $\gamma_1$, $\dots$, $\gamma_5\in\mathbb{R}$.


\begin{remark}\label{H3-re: D0 - loc symm}\rm 
A metric~\eqref{H3-eq: g_D0 ip}  is Einstein if and only if $4\gamma_1\gamma_4+\lambda_1^2=0$, and the only non-zero component of the covariant derivative of the Ricci tensor is given by $(\nabla_{u_4}\rho)(u_4,u_4)=-(\gamma_1+\gamma_4)(4\gamma_1\gamma_4+\lambda_1^2)$, which shows that the  Ricci tensor is parallel if and only if $\gamma_1=-\gamma_4$ or the metric is Einstein.
Furthermore a metric $\mathfrak{g}_{D0}$ is locally symmetric if and only if either
$\gamma_2=\frac{\lambda_1}{2}$ and
the Ricci tensor is parallel
or, otherwise, 
$\gamma_2\neq \frac{\lambda_1}{2}$ and $\gamma_1=\gamma_4=0$.
In both cases the left-invariant metric  is also  locally conformally flat. 

Moreover, any metric $\mathfrak{g}_{D0}$ is homothetic (although not necessarily isomorphically homothetic) to a metric given by \eqref{H3-eq: g_D0 ip} with $\gamma_3=\gamma_5=0$ (cf. \cite{Kulkarni}).
\end{remark}

\begin{remark}\label{H3-re: D0 - paso a R3}\rm 	 
We can proceed exactly as in Remark~\ref{H3-re: R - paso a R3} (after interchanging $\lambda_3$ by $\lambda_1$)  to get that  the underlying Lie algebra of any metric in \eqref{H3-eq: g_D0 ip} is either $ \mathfrak{h}_3\times\mathbb{R} $ or $\mathfrak{n}_4$ if and only if $\gamma_4=-\gamma_1$ and  $\gamma_1^2 = \gamma_2^2$, in which case the analysis is covered by that of Section~\ref{se:R3}.
\end{remark}

\begin{theorem}\label{H3-th: g_D0}
Left-invariant metrics $\mathfrak{g}_{D0}$ on $\sdH$  given by Equation~\eqref{H3-eq: g_D0 ip}
		are never strict algebraic Ricci solitons.
\end{theorem}

\begin{proof}
A straightforward calculation  shows that the conditions for $\mathfrak{D}=\Ricci-\cARS\Id$ to be a derivation are determined by a system of polynomial equations on the soliton constant $\cARS$ and the structure constants in~\eqref{H3-eq: g_D0 ip}, given by $\{\mfP_{ijk}=0\}$, where

\medskip

\[
\begin{array}{llll}
\mfP_{123} = \lambda_1 \cARS ,  
&
\mfP_{141} =     \gamma_1 \cARS ,  
&
\mfP_{142} = -\mfP_{241} = -\gamma_2 \cARS ,     
&
\mfP_{143} = \gamma_3 \cARS ,    

\\[\vsep]

\mfP_{242} =  \gamma_4 \cARS ,      
&
\mfP_{243} =  \gamma_5 \cARS , 
&
\mfP_{343} =\mfP_{141} + \mfP_{242} .     
\end{array}
\]


\normalsize

\bigskip

Hence, any left-invariant metric $\mathfrak{g}_{D0}$  is an steady algebraic Ricci soliton.
A straightforward calculation shows that $u_3$ determines a null recurrent left-invariant vector field, since $\nabla u_3=\omega\otimes u_3$	where $\omega(\cdot)=-(\gamma_1+\gamma_4)\langle u_3,\cdot\rangle$. It follows by a direct calculation that the only non-zero component of the Ricci tensor is given by $\rho(u_4,u_4)=\frac{1}{2}(4\gamma_1\gamma_4+\lambda_1^2)$, which proves that the Ricci operator is isotropic. Furthermore one has that $R(x,y)=0$  and $\nabla_xR=0$ for all $x,y\in u_3^\perp$. Hence left-invariant metrics determined by  \eqref{H3-eq: g_D0 ip} are plane waves (see \cite{Leistner}).

Moreover, since $\lambda_1\neq 0$, we may take $\lambda_1=1$ working in the homothetic class of the initial metric, which corresponds to plane waves in class (a) in Section~\ref{sse:plane-wave-homog}. 
This finishes the proof.  
\end{proof}

\begin{remark}\rm\label{grupos-ext-H}
Any semi-direct extension $\mathcal{H}^3\rtimes\mathbb{R}$ of the Heisenberg group admits non-Einstein left-invariant Lorentz metrics which are plane waves.

Let $\xi=(\gamma_1,\gamma_2,\gamma_3,\gamma_4,\gamma_5)$.
Four-dimensional nilpotent Lie groups admit non-Einstein plane wave metrics which are special cases of the metrics $\mathfrak{g}_{D0}$ with $\lambda_1=1$, just choosing the parameters in \eqref{H3-eq: g_D0 ip} as $\xi=(1,-1,0,-1,0)$ (resp., $\xi=(0,0,0,0,0)$), so that they are realized on the Lie group associated to $\mathfrak{n}_4$ (resp., $\mathfrak{h}_3\times\mathbb{R}$).

Non-Einstein left-invariant plane waves on solvable extensions of the Heisenberg group are obtained also from metrics $\mathfrak{g}_{D0}$.
Setting $\lambda_1=1$ and varying the parameters   in \eqref{H3-eq: g_D0 ip}, one has that the underlying Lie algebras are as follows. For $\xi=(-1,0,0,1,0)$ the corresponding non-Einstein plane wave realizes on $\mathfrak{d}_4$, while metrics corresponding to $\xi=(-2,1,0,0,0)$ are non-Einstein plane waves on $\mathfrak{h}_4$. 
The one-parameter family $\mathfrak{d}_{4,\lambda}$ is obtained considering non-Einstein left-invariant plane waves determined by the parameters
$\xi=(-\lambda,0,0,\lambda-1,0)$ with $\lambda\neq \frac{1}{2}(1\pm\sqrt{2})$, and 
$\xi=(-\frac{1}{2}(1+\sqrt{6}),1,0,-\frac{1}{2}(1-\sqrt{6}),0)$ otherwise,
Finally, the one-parameter family $\mathfrak{d}'_{4,\lambda}$ is obtained considering the parameters
$\xi=(-\lambda,-1,0,-\lambda,0)$.
\end{remark}

\subsubsection{\bf $\bm{{\mathfrak{h}_3'=\span\{v\}}}$ is a spacelike subspace}\label{H3-se:deg-spacelike subspace}

We set  $u_1=\frac{v}{\|v\|}$ and consider a pseudo-orthonormal basis $\{u_1,u_2,u_3,u_4\}$ of $\mathfrak{g}=\algsdH$, with $\langle u_1,u_1\rangle=\langle u_2,u_2\rangle=\langle u_3,u_4\rangle=1$,  where  $\mathfrak{h}_3=\span\{u_1,u_2,u_3\}$ and $\mathbb{R}=\span\{u_4\}$, so that  
\begin{equation}\label{H3-eq: g_D+ ip}
\mathfrak{g}_{D+}
\left\{
\begin{array}{l}
[u_1,u_4]= \gamma_1 u_1, \quad
[u_2,u_3]= \lambda_3 u_1, \quad
[u_2,u_4]= \gamma_2 u_1 +\gamma_3 u_2 + \gamma_4 u_3,
\\
\noalign{\medskip}
[u_3,u_4]=  \gamma_5 u_1+\gamma_6 u_2 + (\gamma_1-\gamma_3) u_3 ,
\end{array}
\right.
\end{equation}  
where  $\lambda_3\neq 0$ and $\gamma_1$, $\dots$, $\gamma_6\in\mathbb{R}$.

\begin{remark}\label{H3-re: D+ - loc symm}\rm 	 
A metric~\eqref{H3-eq: g_D+ ip}  is never either Einstein or locally symmetric.
\end{remark}

\begin{remark}\label{H3-re: D+ - paso a R3}\rm 
Considering the change of basis 
\[
\bar u_1= u_3,\quad
\bar u_2 = u_2,\quad
\bar u_3 = u_1,\quad
\bar u_4 = \tfrac{\gamma_5}{\lambda_3} 	
u_2-\tfrac{\gamma_2}{\lambda_3} u_3+u_4,
\]
the Lie bracket in Equation~\eqref{H3-eq: g_D+ ip} transforms into
\[
\begin{array}{ll}
[\bar u_1,\bar u_2] = -\lambda_3 \bar u_3,  
&
[\bar u_1,\bar u_4] = (\gamma_1-\gamma_3) \bar u_1
+\gamma_6 \bar u_2,
\\
\noalign{\medskip}
[\bar u_2,\bar u_4] = \gamma_4 \bar u_1
+ \gamma_3\bar u_2,
&
[\bar u_3,\bar u_4] = \gamma_1 \bar u_3,
\end{array}
\]
so that, when evaluating on the basis $\{\bar u_i\}$,
\[ 
\ad_{\bar u_4}=\left(
\begin{array}{cc}
A& 0 
\\
0 & \operatorname{tr}A
\end{array}
\right), 
\quad \text{where}\,\, 
A=\left(
\begin{array}{cc}
-\gamma_1+\gamma_3 & -\gamma_4  
\\
-\gamma_6 & -\gamma_3  
\end{array}
\right).
\]
Hence the Lie algebra corresponds to $ \mathfrak{h}_3\times\mathbb{R} $ or $\mathfrak{n}_4$ if and only if $\gamma_1=0$ and $\gamma_3^2+\gamma_4 \gamma_6=0$ (cf.~\cite{ABDO}).
 Indeed,  if  $\gamma_6\neq 0$, the change of basis 
$\widetilde u_1 = -\tfrac{\gamma_5}{\lambda_3} 	
u_2 + \tfrac{\gamma_2}{\lambda_3} u_3 - u_4$, 
$\widetilde u_2 =   \lambda_3  u_1$, 
$\widetilde u_3 = \gamma_6  u_2 - \gamma_3 u_3$, 
$\widetilde u_4 =    u_3$,
transforms the Lie bracket into
$[\widetilde u_1,\widetilde u_4] =   \widetilde u_3$ and 
$[\widetilde u_3, \widetilde u_4] =    \gamma_6 \widetilde u_2$,
while if $\gamma_6=0$ (which implies $\gamma_3=0$),  we use the basis
$\widetilde u_1 =   \tfrac{\gamma_5}{\lambda_3} 	
u_2 - \tfrac{\gamma_2}{\lambda_3} u_3 + u_4$, 
$\widetilde u_2 =\lambda_3   u_1$, 
$\widetilde u_3 =   u_3$, 
$\widetilde u_4 = -   u_2$,
to get
$[\widetilde u_1,\widetilde u_4] = \gamma_4 \widetilde u_3$,
and
$[\widetilde u_3, \widetilde u_4] =    \widetilde u_2$.
Thus, we conclude that if $\gamma_1=0$ and $\gamma_3^2+\gamma_4 \gamma_6=0$ the corresponding left-invariant metrics are also realized on $\sdR$.

\end{remark}

\begin{theorem}\label{H3-th: g_D+}
Any  left-invariant metric $\mathfrak{g}_{D+}$ on $\sdH$  given by Equation~\eqref{H3-eq: g_D+ ip} which is an algebraic Ricci soliton is also   realized on $\sdR$.  
\end{theorem}

\begin{proof}
The conditions for $\mathfrak{D}=\Ricci-\cARS\Id$ to be a derivation are determined by a system of polynomial equations on the soliton constant $\cARS$ and the structure constants in~\eqref{H3-eq: g_D+ ip},   $\{\mfP_{ijk}=0\}$.
In particular, we will make use of the following polynomials: 
\medskip


%
%
%
%

$
\begin{array}{l}
2 \mfP_{131} = 
\left(\gamma_5 \gamma_6 + (4 \gamma_1 - \gamma_3) \lambda_3 \right) 
\lambda_3 ,  

\end{array}
$

%
%
%
%
%

\vspace{\vsep}

$
\begin{array}{l}
2 \mfP_{142} =  
-(3 \gamma_1^2 + \gamma_3^2 - 
4 \gamma_1 \gamma_3 + \gamma_4 \gamma_6) \lambda_3 - 
3 \gamma_1 \gamma_5 \gamma_6 ,
\end{array}
$

%
%
%
%
%
%
%

\vspace{\vsep}

$
\begin{array}{l}
2  \mfP_{232} =    
- \left(2 \gamma_5 \gamma_6 + (3 \gamma_1  - 
2 \gamma_3) \lambda_3\right) \lambda_3 .
\end{array}
$

\normalsize

\bigskip

Since $2\mfP_{131}+\mfP_{232}=\frac{5}{2}\gamma_1\lambda_3^2$, we get $\gamma_1=0$. Now, using this condition, we have
$\mfP_{142} = -\frac{1}{2}(\gamma_3^2+\gamma_4 \gamma_6)\lambda_3$, which implies $\gamma_3^2+\gamma_4 \gamma_6=0$. Hence, by Remark~\ref{H3-re: D+ - paso a R3}, any algebraic Ricci soliton is also realized on $\sdR$, finishing the proof. 
\end{proof}

\section{Semi-direct extensions of the Euclidean and Poincaré Lie groups}\label{se:EE}

We show that a semi-direct extension of the Euclidean or Poincaré Lie groups admitting a non-Einstein algebraic Ricci soliton is necessarily unimodular. Hence it follows from the results in  \cite{ABDO} that they reduce to a direct product $E(1,1)\times \mathbb{R}$ or $\widetilde{E}(2)\times\mathbb{R}$, and thus they are isomorphic to a semi-direct extension $\mathbb{R}^3\rtimes\mathbb{R}$ of the Abelian Lie group covered by the analysis in Section~\ref{se:R3}.

\begin{theorem}\label{E-ATS-thm}
	Let $G=G_3\rtimes\mathbb{R}$ with $G_3=E(1,1)$ or $G_3=\widetilde{E}(2)$. If $(G,\ip)$ is a  non-Einstein algebraic Ricci soliton, then the Lie group is unimodular and isomorphic to a semi-direct extension $\mathbb{R}^3\rtimes\mathbb{R}$.
\end{theorem}

The proof of Theorem~\ref{E-ATS-thm} follows directly from the discussion below, where we consider all left-invariant Lorentz metrics on $G=G_3\rtimes\mathbb{R}$ with $G_3=E(1,1)$ or $G_3=\widetilde{E}(2)$ following Section~\ref{se:Lorentz Lie groups}. 

\subsection{Extensions of the Riemannian Euclidean and Poincaré Lie groups}
Let $G=G_3\rtimes\mathbb{R}$ with $G_3=E(1,1)$ or $G_3=\widetilde{E}(2)$ so that the restriction of the metric to $G_3$ is positive definite. Proceeding as in Section~\ref{se:Lorentz-Riemannian},  there exists an orthonormal basis $\{ e_i\}$ of the Lie algebra $\mathfrak{g}=\mathfrak{g}_3\rtimes\mathbb{R}$ with $e_4$ timelike so that 
$$
\begin{array}{lll}
[e_1,e_3]=-\lambda_2 e_2, & [e_2,e_3]=\lambda_1 e_1,& 
\\ \noalign{\medskip}
{[e_1,e_4]}=\gamma_1 e_1+\gamma_2\lambda_2 e_2, &{[e_2,e_4]}=-\gamma_2\lambda_1 e_1+\gamma_1 e_2,&[e_3,e_4]=\gamma_3 e_1+\gamma_4 e_2,
\end{array}
$$
where $\lambda_1 \lambda_2\neq 0$ and $\gamma_1, \gamma_2, \gamma_3, \gamma_4\in\mathbb{R}$.
The Lie algebra is unimodular if and only if $\gamma_1=0$.
Moreover, a left-invariant metric as above is Einstein if and only if $\gamma_1=\gamma_3=\gamma_4=0$ and $\lambda_1=\lambda_2$.

We analyze the polynomial equations $\{\mfP_{ijk}=0\}$ equivalent to $\mathfrak{D}= \operatorname{Ric} -\cARS\Id$ being a derivation. Considering the equations
$\mfP_{134}=\frac{1}{2}\gamma_3\lambda_1\lambda_2$ and $\mfP_{234}=\frac{1}{2}\gamma_4\lambda_1\lambda_2$, one has $\gamma_3=\gamma_4=0$ and hence a straightforward calculation shows that 
\[
\begin{array}{l}
-2 \mfP_{231} =   \left( (\gamma_2^2 - 
1)   (\lambda_1 - \lambda_2) (3 \lambda_1 + \lambda_2) - 
2 \cARS \right) \lambda_1 ,

\\ \noalign{\medskip}

-2 \mfP_{132} =   \left( (\gamma_2^2 - 1) (\lambda_1 - \lambda_2)  (\lambda_1 + 
3 \lambda_2) + 2  \cARS \right) \lambda_2 ,

\\ \noalign{\medskip}

-2 \mfP_{141} =  \left(  \gamma_2^2 (\lambda_1 - \lambda_2) (3 \lambda_1 + 
\lambda_2) + 4 \gamma_1^2 - 2 \cARS \right) \gamma_1 ,

\\ \noalign{\medskip}

\phantom{-} 2 \mfP_{242} =  \left(   \gamma_2^2 (\lambda_1 - \lambda_2) (\lambda_1 + 
3 \lambda_2) - 4 \gamma_1^2 + 2 \cARS \right) \gamma_1.
\end{array}
\]
Now  a direct calculation leads to
\[
\gamma_1 \lambda_2 \mfP_{231}-\gamma_1 \lambda_1 \mfP_{132}
-  \lambda_1\lambda_2 
( \mfP_{141}+\mfP_{242} ) 
=\gamma_1 (4\gamma_1^2+(\lambda_1-\lambda_2)^2)\lambda_1\lambda_2,
\]
which implies    $\gamma_1=0$,  and therefore the underlying Lie algebra is unimodular.

\subsection{Extensions of the Lorentzian Euclidean and Poincaré Lie groups}

We consider the different possibilities for the Jordan normal form of the structure operator $L$.

\subsubsection{\bf Diagonalizable structure operator with spacelike $\bm{\ker L}$}\label{E-Ia-spacelike}

In this case we proceed as in Section~\ref{se:Lorentz-Lorentz} and assume without loss of generality that $\lambda_1=0$, so that the metric is described in an orthonormal basis $\{ e_i\}$ with $e_3$ timelike by
$$
\begin{array}{lll}
[e_1,e_2]=-\lambda_3 e_3, & [e_1,e_3]=-\lambda_2 e_2, &
\\ \noalign{\medskip}
{[e_1,e_4]}=\gamma_1 e_2+\gamma_2 e_3, & [e_2,e_4]=\gamma_3 e_2+\gamma_4\lambda_3 e_3,&[e_3,e_4]=\gamma_4 \lambda_2 e_2+\gamma_3 e_3,
\end{array}
$$
where $\lambda_2\lambda_3\neq 0$ and $\gamma_1,\gamma_2,\gamma_3,\gamma_4\in\mathbb{R}$.
The Lie algebra is unimodular if and only if $\gamma_3=0$. Moreover, a left-invariant metric above is Einstein if and only if $\gamma_1=\gamma_2=\gamma_3=0$ and $\lambda_2=\lambda_3$.

A straightforward calculation shows that the algebraic Ricci soliton equations $\{\mfP_{ijk}=0\}$ are such that $\mfP_{124}=-\frac{1}{2}\gamma_1\lambda_2\lambda_3$ and $\mfP_{134}=\frac{1}{2}\gamma_2\lambda_2\lambda_3$. Therefore,  $\gamma_1=\gamma_2=0$ and
we calculate  
\[
\begin{array}{l}
\phantom{-} 
2 \mfP_{132} =   \left( (\gamma_4^2 + 
1) (\lambda_2 - \lambda_3) (3 \lambda_2 + \lambda_3) - 
2 \cARS \right) \lambda_2 ,

\\ \noalign{\medskip}

-2 \mfP_{123} =   \left((\gamma_4^2 + 1) (\lambda_2 - \lambda_3)   (\lambda_2 + 
3 \lambda_3) + 2  \cARS \right) \lambda_3 ,

\\ \noalign{\medskip}

-2 \mfP_{242} =   \left(  \gamma_4^2 (\lambda_2 - \lambda_3) (3 \lambda_2 + \lambda_3) 
- 4 \gamma_3^2 - 2   \cARS \right) \gamma_3 ,

\\ \noalign{\medskip}

\phantom{-} 2 \mfP_{343} =  \left(    \gamma_4^2 (\lambda_2 - \lambda_3) (\lambda_2 + 
3 \lambda_3) + 4 \gamma_3^2 + 2  \cARS \right) \gamma_3 .
\end{array}
\]
By a direct calculation  we obtain
\[
\gamma_3 \lambda_2 \mfP_{123} + \gamma_3 \lambda_3 \mfP_{132}
+  \lambda_2\lambda_3 
( \mfP_{242}+\mfP_{343} ) 
=\gamma_3 (4\gamma_3^2+(\lambda_2-\lambda_3)^2)\lambda_2\lambda_3,
\]
which shows that $\gamma_3=0$. Hence,  the underlying Lie algebra is unimodular.

\subsubsection{\bf Diagonalizable structure operator with timelike $\bm{\ker L}$}\label{E-Ia-timelike}

We proceed as in Section~\ref{se:Lorentz-Lorentz} assuming that $\lambda_3=0$, so that the metric is described in an orthonormal basis $\{ e_i\}$ with $e_3$ timelike by
$$
\begin{array}{lll}
[e_1,e_3]=-\lambda_2 e_2, & [e_2,e_3]=\lambda_1 e_1, &
\\ \noalign{\medskip}
{[e_1,e_4]}=\gamma_1 e_1+\gamma_2\lambda_2 e_2, & [e_2,e_4]=-\gamma_2\lambda_1 e_1+\gamma_1 e_2, &[e_3,e_4]=\gamma_3 e_1+\gamma_4 e_2,
\end{array}
$$
where $\lambda_1\lambda_2\neq 0$ and   $\gamma_1, \gamma_2, \gamma_3, \gamma_4\in\mathbb{R}$.
The Lie algebra is unimodular if and only if $\gamma_1=0$.
Moreover, a left-invariant metric as above is Einstein if and only if $\gamma_1=\gamma_3=\gamma_4=0$ and $\lambda_1=\lambda_2$.

We analyze the polynomial equations $\{\mfP_{ijk}=0\}$ equivalent to $\mathfrak{D}= \operatorname{Ric} -\cARS \Id$ being a derivation as follows. Considering the polynomials
$\mfP_{134}=\frac{1}{2}\gamma_3\lambda_1\lambda_2$ and $\mfP_{234}=\frac{1}{2}\gamma_4\lambda_1\lambda_2$, one has that $\gamma_3=\gamma_4=0$ and hence
$$
\begin{array}{l}
\phantom{-}
2\mfP_{231} =
\left(( \gamma_2^2 - 
1) ( \lambda_1 -  \lambda_2) (3 \lambda_1 +  \lambda_2) + 
2 \cARS  \right) \lambda_1 ,

\\ \noalign{\medskip}

\phantom{-}
2\mfP_{132} = 
\left(( \gamma_2^2 - 1) ( \lambda_1 -  \lambda_2) ( \lambda_1 + 
3 \lambda_2) - 2 \cARS  \right) \lambda_2 ,

\\ \noalign{\medskip}

\phantom{-}
2\mfP_{141} = 
\left(   \gamma_2^2 ( \lambda_1 -  \lambda_2) (3 \lambda_1 + 
\lambda_2) + 4 \gamma_1^2 + 2 \cARS  \right) \gamma_1 ,

\\ \noalign{\medskip}

- 2\mfP_{242} = 
\left(    \gamma_2^2 ( \lambda_1 -  \lambda_2) ( \lambda_1 + 
3 \lambda_2) - 4 \gamma_1^2 - 2 \cARS  \right) \gamma_1 .
\end{array}
$$
Now one easily checks that
\[
\gamma_1 \lambda_1 \mfP_{132} - \gamma_1 \lambda_2 \mfP_{231}+
\lambda_1\lambda_2 (\mfP_{141}+\mfP_{242}) =
\gamma_1  \left(4 \gamma_1^2 + ( \lambda_1 - \lambda_2)^2 \right) \lambda_1 
\lambda_2 ,
\]
which leads to $\gamma_1=0$. Hence, we conclude that the underlying Lie algebra is unimodular.

\subsubsection{\bf Structure operator $\bm{L}$ with a complex eigenvalue}\label{E-complex}

We proceed as in the previous cases to show that the existence of non-Einstein algebraic Ricci solitons leads to unimodularity of the underlying Lie group. In this case, the metric is determined by 
$$
\begin{array}{l}
[e_1,e_2]=-\beta e_2-\alpha e_3, 
\qquad 
[e_1,e_3]=-\alpha e_2+\beta e_3,\qquad [e_1,e_4]=\gamma_1 e_2+\gamma_2 e_3,
\\ \noalign{\medskip}
[e_2,e_4]=2\gamma_3\beta e_2+(\gamma_3-\gamma_4)\alpha e_3,
\quad
[e_3,e_4]=(\gamma_3-\gamma_4)\alpha e_2+2\gamma_4\beta e_3 ,
\end{array}
$$
where $\beta\neq 0$, $\alpha,\gamma_1, \gamma_2, \gamma_3, \gamma_4\in\mathbb{R}$ and
$\{ e_i\}$ is an orthonormal basis with $e_3$ timelike. The Lie algebra is unimodular if and only if $\gamma_3+\gamma_4=0$, and the left-invariant metric is Einstein if and only if $\alpha=\gamma_1=\gamma_2=0$ and $\gamma_3=\gamma_4=\pm\frac{1}{2}$.

Considering the polynomial equations $\{\mfP_{ijk}=0\}$ equivalent to $\mathfrak{D}= \operatorname{Ric} -\cARS\Id$ being a derivation, one has that $\mfP_{144}=\frac{1}{2}(\gamma_1^2+\gamma_2^2)\beta$ so that $\gamma_1=\gamma_2=0$. In this situation one further has
\small
$$
\begin{array}{l}
\phantom{-}
\mfP_{122} = 	
\left\{ 4  \left((\gamma_3 - \gamma_4)^2 + 1\right) \alpha^2 - 
2 \left(2  (\gamma_3 - \gamma_4) \gamma_3 + 
1 \right) \beta^2 - \cARS \right\} \beta ,

\\ \noalign{\medskip}

-\mfP_{133} =
\left\{ 4  \left((\gamma_3 - \gamma_4)^2 + 1\right) \alpha^2 + 
2  \left(2  (\gamma_3 - \gamma_4) \gamma_4 - 
1 \right) \beta^2 - \cARS \right\} \beta , 

\\ \noalign{\medskip}

-\mfP_{123} =
\left\{ 2    \left(  (5 \gamma_3 - \gamma_4) (\gamma_3 - \gamma_4)  + 
3 \right) \beta^2 + \cARS \right\} \alpha ,

\\ \noalign{\medskip}

-\mfP_{242} =
2 \left\{
2  \left((\gamma_3 - \gamma_4)^2 + 1\right) (\gamma_3 - \gamma_4) \alpha^2 
-  \left(4  (\gamma_3^2 + \gamma_4^2) \gamma_3 + \gamma_3 - \gamma_4 \right) 
\beta^2 -  \gamma_3 \cARS
\right\} \beta .
\end{array}
$$
\normalsize
Note that $\mfP_{122}+\mfP_{133} = -4(\gamma_3^2-\gamma_4^2)\beta^3$, which implies $\gamma_4=\pm\gamma_3$. If $\gamma_4=-\gamma_3$   the underlying Lie algebra is unimodular. Now, if $\gamma_4=\gamma_3\neq 0$,  a direct calculation shows that 
\[
\alpha \mfP_{122} - \beta \mfP_{123} = 4(\alpha^2+\beta^2)\alpha \beta,
\quad
2\gamma_3\mfP_{122} + \mfP_{242} = 
4 \left( 2\alpha^2 +(4\gamma_3^2-1)\beta^2 \right) \gamma_3 \beta .
\]
Hence,   $\alpha=0$, $4\gamma_3^2=1$ and the corresponding left-invariant metric is Einstein.

\subsubsection{\bf Structure operator of type II with degenerate kernel}\label{E-II-degenerate}
In this case the metric is described, in terms of a pseudo-orthonormal basis $\{ u_i\}$ of the Lie algebra with $\langle u_1,u_2\rangle=\langle u_3,u_3\rangle=\langle u_4,u_4\rangle=1$, by the Lie brackets
$$
\begin{array}{lll}
[u_1,u_2]=\lambda_2 u_3,& [u_1,u_3]=-\varepsilon u_2, &
\\ \noalign{\medskip}
[u_1,u_4]=\gamma_1 u_2+\gamma_2 u_3, & 
[u_2,u_4]=\gamma_3 u_2+\gamma_4\lambda_2  u_3, & 
[u_3,u_4]=-\varepsilon\gamma_4 u_2+\gamma_3 u_3,
\end{array}
$$
where $\lambda_2\neq 0$, $\ve=\pm 1$ and  $\gamma_1, \gamma_2, \gamma_3, \gamma_4\in\mathbb{R}$. 
The Lie algebra is unimodular if and only if $\gamma_3=0$. Moreover, these metrics are never Einstein.

Considering the component $\mfP_{131}=\frac{\ve}{2}\gamma_4^2\lambda_2^2$ one has $\gamma_4=0$, and then we obtain
\[
\mfP_{123} =-\tfrac{1}{2} (3\lambda_2^2-2\cARS)\lambda_2,
\quad
\mfP_{132} = -\tfrac{\ve}{2} (\lambda_2^2+4\gamma_3^2+2\cARS). 
\]
Now, one easily checks that
\[
\ve \mfP_{123} +\lambda_2 \mfP_{132} = -2\ve (\lambda_2^2+\gamma_3^2)\lambda_2,	
\]
which is non-zero since $\lambda_2\neq 0$,  thus showing the non-existence of algebraic Ricci solitons in this setting.

\subsubsection{\bf Structure operator of type II with spacelike kernel}\label{E-II-spacelike}
The metric is described by
$$
\begin{array}{lll}
[u_1,u_3]= -\lambda_1 u_1 -\varepsilon u_2 , &
[u_2,u_3]=\lambda_1 u_2 ,
\\ \noalign{\medskip}
[u_1,u_4]=\gamma_1 u_1+\gamma_2 u_2, &
[u_2,u_4]=(\gamma_1-2\varepsilon\gamma_2\lambda_1)u_2, & 
[u_3,u_4]=\gamma_3 u_1+\gamma_4 u_2,
\end{array}
$$
where $\lambda_1\neq 0$, $\ve=\pm 1$, $\gamma_1, \gamma_2, \gamma_3, \gamma_4\in\mathbb{R}$ and 
$\{ u_i\}$ is a pseudo-orthonormal basis of the Lie algebra with $\langle u_1,u_2\rangle=\langle u_3,u_3\rangle=\langle u_4,u_4\rangle=1$. The Lie algebra is unimodular if and only if $\varepsilon\gamma_2\lambda_1-\gamma_1=0$. Moreover, the left-invariant metrics in this family are never Einstein.

Considering the polynomials
$$
\mfP_{134}=\tfrac{1}{2}\gamma_4\lambda_1^2,\quad
\mfP_{234}=\tfrac{1}{2}\gamma_3\lambda_1^2,\quad
\mfP_{232}=\tfrac{1}{2}(3\gamma_4\lambda_1+\ve\gamma_3)\gamma_3+\cARS\lambda_1,
$$
one has that $\gamma_3=\gamma_4=\cARS=0$ and we get 
\[
\mfP_{132} = 4( \ve \lambda_1+\gamma_1\gamma_2)\lambda_1,
\quad
\mfP_{141} =2\left(   (\gamma_2 \lambda_1 -2\ve \gamma_1 )\gamma_2\lambda_1+\gamma_1^2\right) \gamma_1.
\]
Hence, one easily checks that
\[
\gamma_1 \mfP_{132} +\ve \mfP_{141} =
2\ve \left((\gamma_2^2+2)\lambda_1^2 +\gamma_1^2\right) \gamma_1 ,
\]
which leads to $\gamma_1=0$. Hence, $\mfP_{132} = 4\ve \lambda_1^2\neq 0$ and we conclude that left-invariant metrics in this case do not support any algebraic Ricci soliton.

\subsubsection{\bf Structure operator of type III}\label{E-III}
The metric is described by
$$
\begin{array}{lll}
[u_1,u_2]=u_1,  &
[u_2,u_3]=u_3,
\\ \noalign{\medskip}
[u_1,u_4]=\gamma_1 u_1, & 
[u_2,u_4]=\gamma_2 u_1+\gamma_3 u_3, & 
[u_3,u_4]=\gamma_4 u_3,
\end{array}
$$
where $\gamma_1, \gamma_2, \gamma_3, \gamma_4\in\mathbb{R}$ and  $\{ u_i\}$ is a pseudo-orthonormal basis of the Lie algebra with $\langle u_1,u_2\rangle=\langle u_3,u_3\rangle=\langle u_4,u_4\rangle=1$. The Lie algebra is unimodular if and only if $\gamma_1+\gamma_4=0$.
A left-invariant metric determined by the brackets above is Einstein if and only if $\gamma_1=2\gamma_4$, $\gamma_3=0$ and $\gamma_2\gamma_4=-2$.

We proceed as in the previous cases. Considering the polynomials
\[
\mfP_{121} = \tfrac{1}{2}  ( 2\gamma_1^2-\gamma_1\gamma_4+2\cARS  ),
\quad
\mfP_{233} = \tfrac{1}{2}  ( \gamma_1^2+2\gamma_4^2+2\cARS),
\quad
\mfP_{231} = - 2 \gamma_3\gamma_4,
\]
it follows that 
\[
\mfP_{121} - \mfP_{233} =\tfrac{1}{2} (\gamma_1-2\gamma_4)(\gamma_1+\gamma_4). 
\]
If $\gamma_1+\gamma_4=0$ the underlying Lie algebra is unimodular. Otherwise, 
if $\gamma_1=2\gamma_4\neq 0$, then $\gamma_3=0$ and we calculate
\[
\mfP_{121} = 3\gamma_4^2+\cARS,
\quad
\mfP_{241} = 5\gamma_2 \gamma_4^2+4\gamma_4+\gamma_2\cARS,
\]
which implies $\mfP_{241} - \gamma_2 \mfP_{121} = 2(\gamma_2\gamma_4+2)\gamma_4$. Hence,  $\gamma_2\gamma_4=-2$ and  the corresponding left-invariant metric is Einstein.

\red

\subsection{Extensions of degenerate Euclidean and Poincaré Lie groups}
{\black We follow the discussion in Section~\ref{se:Lorentz-degenerate} and consider the derived algebra $\mathfrak{g}_3'=[\mathfrak{g}_3,\mathfrak{g}_3]$. Since the restriction of the metric to $\mathfrak{g}_3$ is degenerate of signature $(++0)$, we consider separately the cases when the induced metric on $\mathfrak{g}'_3$ is Riemannian in Section~\ref{EE-se:riemannian}, while the case when the restriction of the metric to  $\mathfrak{g}'_3$ is degenerate in considered in Section~\ref{EE-se:degenerate}.

\subsubsection{\bf The induced metric on the derived algebra $\bm{\mathfrak{g}'_3}$ is positive definite}\label{EE-se:riemannian}

There exists a pseudo-orthonormal basis $\{u_i\}$ of the Lie algebra $\mathfrak{g}=\mathfrak{g}_3\rtimes\mathbb{R}$ with $\langle u_1,u_1\rangle=\langle u_2,u_2\rangle=\langle u_3,u_4\rangle=1$ so that $\mathfrak{g}_3=\operatorname{span}\{ u_1,u_2,u_3\}$ and the derived algebra $\mathfrak{g}'_3=\operatorname{span}\{ u_1,u_2\}$.
In this situation, using that $\mathfrak{g}_3$ is unimodular and the derived algebra $\mathfrak{g}'_3$ is Abelian, the Jacobi identity leads to two-distinct situations which we consider as follows.

\subsubsection*{\underline{Case 1}}
There exists a pseudo-orthonormal basis $\{u_i\}$ of the Lie algebra with $\langle u_1,u_1\rangle=\langle u_2,u_2\rangle=\langle u_3,u_4\rangle=1$ so that
\begin{equation}\label{deg-riem-1}
\begin{array}{lll}
[u_1,u_3]=\lambda_1 u_2,  &\!
[u_2,u_3]=-\lambda_1 u_1,
\\ \noalign{\medskip}
[u_1,u_4]=\gamma_1 u_1 + \gamma_2 u_2, & \!
[u_2,u_4]=-\gamma_2 u_1+\gamma_1 u_2, & \!
[u_3,u_4]=\gamma_3 u_1 + \gamma_4 u_2,
\end{array}
\end{equation}
where $\lambda_1\neq 0$ and $\gamma_1, \gamma_2, \gamma_3, \gamma_4\in\mathbb{R}$. The Lie algebra $\mathfrak{g}=\mathfrak{g}_3\rtimes\mathbb{R}$ is unimodular if and only if $\gamma_1 =0$. Moreover, 
a left-invariant metric determined by the brackets above is Einstein if and only if $\gamma_1=\gamma_3=\gamma_4=0$. The three-dimensional Lie algebra $\mathfrak{g}_3=\mathfrak{e}(2)$ in this case.

We analyze the polynomial equations $\{\mfP_{ijk}=0\}$ equivalent to $\mathfrak{D}=\Ricci-\cARS \Id$ being a derivation. Considering the polynomials
\small
\[
2 \mfP_{142} =        4 \gamma_1^2 { \lambda_1} - 3 \gamma_2 \gamma_4^2 - 
3 \gamma_1 \gamma_3 \gamma_4 + 2 \gamma_2 \cARS ,
\,\,\,
-2 \mfP_{131} =3 \gamma_3\gamma_4{ \lambda_1},
\,\,\,
-2 \mfP_{132} =   (3 \gamma_4^2 - 2 \cARS) { \lambda_1},
\]
\normalsize
a direct calculation shows that
\[
{ \lambda_1} \mfP_{142} - \gamma_1 \mfP_{131} - \gamma_2 \mfP_{132} 
= 2\gamma_1^2 { \lambda_1}^2.
\] 
Hence, $\gamma_1=0$ and the underlying Lie algebra is unimodular, thus covered by the analysis in Section~\ref{se:R3}.

\begin{remark}\rm\label{re:plane-wave no ARS}
Let $G$ be the affine group corresponding to $\mathfrak{aff}(\mathbb{C})$ with the left-invariant metric determined by 
$$
	[u_1,u_3]\!=\!\lambda_1 u_2,  \,\,\,\,
	[u_2,u_3]\!=\!-\lambda_1 u_1,\,\,\,\,
	[u_1,u_4]\!=\!\gamma_1 u_1 + \gamma_2 u_2, \,\,\,\,
	[u_2,u_4]\!=\!\gamma_1 u_2-\gamma_2 u_1, 
$$
corresponding to \eqref{deg-riem-1} with $\gamma_3=\gamma_4=0$, where $\gamma_1\neq 0$. A straightforward calculation shows that $u_3$ determines a parallel null vector field. Moreover one has that $R(x,y)=0$  and $\nabla_xR=0$ for all $x,y\in u_3^\perp$. Hence left-invariant metrics determined by  \eqref{deg-riem-1} with $\gamma_3=\gamma_4=0$ are plane waves (see \cite{Leistner}). Note that we may take $\lambda_1=1$ working in the homothetic class of the initial metric, corresponding to case~(d) in Section~\ref{sse:plane-wave-homog}.
Furthermore $(G,\ip)$ is a locally symmetric and locally conformally flat plane wave. Moreover, it follows from the above that $(G,\ip)$ is not an algebraic Ricci soliton.
This reflects the fact that algebraic Ricci solitons depend both on the Lorentzian and Lie group structure, since the left-invariant metric on $\mathcal{H}^3\times\mathbb{R}$ in Section~\ref{re:producto}
corresponding to $(\alpha,\beta)=(1,1)$, being locally conformally flat Cahen-Wallach symmetric spaces, are algebraic Ricci solitons.
\end{remark}

\subsubsection*{\underline{Case 2}}
There exists a pseudo-orthonormal basis $\{u_i\}$ of the Lie algebra with $\langle u_1,u_1\rangle=\langle u_2,u_2\rangle=\langle u_3,u_4\rangle=1$ so that
\begin{equation}\label{deg-riem-2}
\begin{array}{l}
[u_1,u_3]\!=\!\lambda_1 u_1+\lambda_2 u_2,  \quad
[u_2,u_3]\!=\! -\lambda_2 u_1-\lambda_1 u_2, \quad
[u_3,u_4]\!=\!\gamma_3 u_1 + \gamma_4 u_2,

\\ \noalign{\medskip}
[u_1,u_4]\!=\! \gamma_1 u_1 + \frac{(\gamma_1-\gamma_2)\lambda_2}{2\lambda_1} u_2, \quad
[u_2,u_4]\!=\! - \frac{(\gamma_1-\gamma_2)\lambda_2}{2\lambda_1} u_1+\gamma_2 u_2,  
\end{array}
\end{equation}
where $\lambda_1\neq 0$, $\lambda_1^2-\lambda_2^2\neq 0$ and  $\gamma_1, \gamma_2, \gamma_3, \gamma_4\in\mathbb{R}$. The Lie algebra $\mathfrak{g}=\mathfrak{g}_3\rtimes\mathbb{R}$ is unimodular if and only if $\gamma_1+\gamma_2=0$. Moreover, the left-invariant metrics in this family are never Einstein. The three-dimensional Lie algebra is $\mathfrak{g}_3=\mathfrak{e}(1,1)$ or $\mathfrak{g}_3=\mathfrak{e}(2)$, depending on whether $\lambda_1^2-\lambda_2^2$ is positive or negative, respectively.

Considering the polynomials
$\mfP_{134}=-\tfrac{1}{2}(\lambda_1^2-\lambda_2^2)\gamma_3$ and
$\mfP_{234}=-\tfrac{1}{2}(\lambda_1^2-\lambda_2^2)\gamma_4$, one has that
$\gamma_3=\gamma_4=0$ and hence
\[
\mfP_{131} = 
(3 \gamma_1 - \gamma_2 ) \lambda_1^2 + 
4 (\gamma_1 - \gamma_2) \lambda_2^2 + \lambda_1 \cARS ,
\,\,\,
\mfP_{232} = 
-(\gamma_1 - 3 \gamma_2) \lambda_1^2 - 
4 ( \gamma_1 - \gamma_2) \lambda_2^2 - \lambda_1 \cARS .
\] 
Now one easily checks that 
\[
\mfP_{131}+\mfP_{232} = 2(\gamma_1+\gamma_2)\lambda_1^2,
\]
which leads to $\gamma_1+\gamma_2=0$. Hence, we conclude that the underlying Lie algebra is unimodular. Therefore the analysis is covered by that in Section~\ref{se:R3}.

\subsubsection{\bf The induced metric on the derived algebra $\bm{\mathfrak{g}'_3}$ is degenerate}\label{EE-se:degenerate}

There exists a pseudo-orthonormal basis $\{u_i\}$ of the Lie algebra $\mathfrak{g}=\mathfrak{g}_3\rtimes\mathbb{R}$ with $\langle u_1,u_1\rangle=\langle u_2,u_2\rangle=\langle u_3,u_4\rangle=1$ so that $\mathfrak{g}_3=\operatorname{span}\{ u_1,u_2,u_3\}$ and the derived algebra $\mathfrak{g}'_3=\operatorname{span}\{ u_1,u_3\}$.
In this situation, using that $\mathfrak{g}_3$ is unimodular and $\mathfrak{g}'_3$ is Abelian, the Jacobi identity leads to two-different situations which we consider as follows.

\subsubsection*{\underline{Case 1}}
There exists a pseudo-orthonormal basis $\{u_i\}$ of the Lie algebra with $\langle u_1,u_1\rangle=\langle u_2,u_2\rangle=\langle u_3,u_4\rangle=1$ so that
the metric is determined by
\begin{equation}\label{deg-deg-1}
\begin{array}{lll}
[u_1,u_2]\!=\!\lambda_1 u_3, \! &\!
[u_2,u_3]\!=\!\lambda_2 u_1,
\\ \noalign{\medskip}
[u_1,u_4]\!=\!\gamma_1 u_1 + \gamma_2 u_3, \!& \!
[u_2,u_4]\!=\!\gamma_3 u_1+\gamma_4 u_3, \!& \!
[u_3,u_4]\!=\!-\frac{\gamma_2\lambda_2}{\lambda_1} u_1 + \gamma_1 u_3,
\end{array}
\end{equation}
where $\lambda_1\lambda_2\neq 0$,  $\gamma_1, \gamma_2, \gamma_3, \gamma_4\in\mathbb{R}$.
The Lie algebra $\mathfrak{g}=\mathfrak{g}_3\rtimes\mathbb{R}$ is unimodular if and only if $\gamma_1 =0$. Moreover, 
a left-invariant metric determined by the brackets above is never Einstein. The three-dimensional Lie algebra is $\mathfrak{g}_3=\mathfrak{e}(1,1)$ or $\mathfrak{g}_3=\mathfrak{e}(2)$, depending on whether $\lambda_1\lambda_2$ is negative or positive, respectively.

We only need to calculate the polynomial $\mfP_{124}=-\frac{1}{2}\lambda_1\lambda_2^2$ to  conclude that left-invariant metrics in this case do not support any algebraic Ricci soliton.

\subsubsection*{\underline{Case 2}}
There exists a pseudo-orthonormal basis $\{u_i\}$ of the Lie algebra with $\langle u_1,u_1\rangle=\langle u_2,u_2\rangle=\langle u_3,u_4\rangle=1$ so that
the metric is determined by
\begin{equation}\label{def-deg-2}
\begin{array}{l}
[u_1,u_2]\!=\! \lambda_1 u_1+\lambda_2 u_3,  \quad
[u_2,u_3]\!=\! \lambda_3 u_1+\lambda_1 u_3, \quad
[u_2,u_4]\!=\! \gamma_2 u_1 + \gamma_3 u_3,

\\ \noalign{\medskip}

[u_1,u_4]\!=\! \gamma_1 u_1 + \frac{(\gamma_1-\gamma_4)\lambda_2}{2\lambda_1} u_3, \quad
[u_3,u_4]\!=\! - \frac{(\gamma_1-\gamma_4)\lambda_3}{2\lambda_1} u_1+\gamma_4 u_3,  
\end{array}
\end{equation}
where $\lambda_1\neq 0$, $\lambda_1^2-\lambda_2\lambda_3\neq 0$ and  $\gamma_1, \gamma_2, \gamma_3, \gamma_4\in\mathbb{R}$. The Lie algebra $\mathfrak{g}=\mathfrak{g}_3\rtimes\mathbb{R}$ is unimodular if and only if $\gamma_1+\gamma_4=0$. Moreover, these left-invariant metrics  are never Einstein. The three-dimensional Lie algebra is $\mathfrak{g}_3=\mathfrak{e}(1,1)$ or $\mathfrak{g}_3=\mathfrak{e}(2)$, depending on whether $\lambda_1^2-\lambda_2\lambda_3$ is positive or negative, respectively.

We proceed as in the previous cases. Since $\mfP_{124}=\tfrac{1}{2}(\lambda_1^2-\lambda_2\lambda_3)\lambda_3$ one has $\lambda_3=0$, and then we get
\[
\mfP_{121} = \tfrac{1}{2} (3\lambda_1^2+2\cARS)\lambda_1,
\quad
\mfP_{343} = -\tfrac{1}{2}(3\gamma_1\lambda_1^2- 2 \gamma_4\cARS).
\]
A direct calculation shows that
\[
\gamma_4\mfP_{121}-\lambda_1\mfP_{343} = \tfrac{3}{2}(\gamma_1+\gamma_4)\lambda_1^3,
\]
which leads to $\gamma_1+\gamma_4=0$. Hence, the underlying Lie algebra is unimodular and this case is also covered by the analysis in Section~\ref{se:R3}.

\section{Direct extensions of the non-solvable Lie groups}\label{se:SLSU}

We proceed as in the previous sections considering all the possible left-invariant Lorentzian metrics on the non-solvable four-dimensional Lie groups $\widetilde{SL}(2,\mathbb{R})\times\mathbb{R}$ and $SU(2)\times\mathbb{R}$. Following the discussion in Section~\ref{se:Lorentz Lie groups}, the different cases corresponding to a Riemannian, Lorentzian or degenerate metrics on $\widetilde{SL}(2,\mathbb{R})$ or $SU(2)$ are considered separately in Sections~\ref{se:91}, \ref{se:92} and  \ref{se:93}, respectively. 
Algebraic Ricci solitons in this setting are necessarily rigid as follows.

\begin{theorem}\label{Nonsolvable-ARS-thm}
	Let $G=G_3\times\mathbb{R}$ with $G_3=\widetilde{SL}(2,\mathbb{R})$ or $G_3=SU(2)$. If $(G,\ip)$ is a  non-Einstein algebraic Ricci soliton, then the left-invariant metric is the product one and the algebraic Ricci soliton is locally isometric to $\mathbb{S}^3\times\mathbb{R}$ or $\mathbb{S}_1^3\times\mathbb{R}$.
\end{theorem}

The proof of Theorem~\ref{Nonsolvable-ARS-thm} follows directly from the discussion below.

\subsection{Direct extensions with Riemannian Lie groups $\bm{\widetilde{SL}(2,} \pmb{\mathbb{R}}\bm{)}$ or $\bm{SU(2)}$} \label{se:91}
Let $G=G_3\times \mathbb{R}$ with $G_3=\widetilde{SL}(2, \mathbb{R})$ or $G_3=SU(2)$ so that the restriction of the metric to $G_3$ is positive definite. Hence, there exists  an orthonormal basis $\{e_i\}$ of the Lie algebra, with $e_4$ timelike,  so that    
\[
\begin{array}{l}
[e_1,e_2]=\lambda_3 e_3,\qquad
[e_1,e_3]=-\lambda_2 e_2, \qquad
[e_2,e_3]=\lambda_1 e_1,

\\
\noalign{\medskip}

[e_1,e_4]=\gamma_1\lambda_2 e_2+\gamma_2\lambda_3 e_3, \qquad 
[e_2,e_4]=-\gamma_1\lambda_1 e_1+\gamma_3\lambda_3 e_3, 

\\
\noalign{\medskip}

[e_3,e_4]=-\gamma_2\lambda_1 e_1-\gamma_3\lambda_2 e_2,
\end{array}
\]
where $\lambda_1\lambda_2\lambda_3\neq 0$ and  $\gamma_1,\gamma_2,\gamma_3\in\mathbb{R}$. The Lie algebra corresponds to  $\mathfrak{su}(2)$ if $\lambda_1$, $\lambda_2$, $\lambda_3$   have the same sign  and to $\mathfrak{sl}(2,\mathbb{R})$ otherwise. Moreover, a left-invariant metric as above is never Einstein.

We analyze the derivation condition for $\mathfrak{D}=\Ricci-\cARS \Id$ by studying the polynomial equations $\{\mfP_{ijk}=0\}$. First, we consider  the polynomials 
\[
\mfP_{124}= -\tfrac{1}{2} (\lambda_1-\lambda_2)^2 \gamma_1 \lambda_3,
\,\,\,
\mfP_{134}= -\tfrac{1}{2} (\lambda_1-\lambda_3)^2 \gamma_2 \lambda_2,
\,\,\,
\mfP_{234}= -\tfrac{1}{2} (\lambda_2-\lambda_3)^2 \gamma_3 \lambda_1.
\]
If $\lambda_1$, $\lambda_2$ and $\lambda_3$ are all different then clearly $\gamma_1=\gamma_2=\gamma_3=0$.
Next we suppose that  at least two of $\lambda_1$, $\lambda_2$ and $\lambda_3$ are  equal.  
Note that the isomorphic isometry
$(e_1,e_2,e_3,e_4)\mapsto (e_3,e_2,-e_1,e_4)$ interchanges $(\lambda_1,\lambda_2,\lambda_3,\gamma_1,\gamma_2,\gamma_3)$ with 
$(\lambda_3,\lambda_2,\lambda_1,-\gamma_3,\gamma_2,\gamma_1)$, while 
$(e_1,e_2,e_3,e_4)\mapsto (e_1,e_3,-e_2,e_4)$ 
interchanges $(\lambda_1,\lambda_2,\lambda_3,\gamma_1,\gamma_2,\gamma_3)$ with 
$(\lambda_1,\lambda_3,\lambda_2,\gamma_2,-\gamma_1,\gamma_3)$. Hence,  without loss of generality,   we can assume $\lambda_1=\lambda_2$  and   we calculate
\small
\[
\begin{array}{l}
\phantom{-}
2 \mfP_{132} =  \left(
\lambda_3^2
+\gamma_3^2(3\lambda_1+\lambda_3)(\lambda_1-\lambda_3)
-2\cARS \right) \lambda_1,

\\ \noalign{\medskip}

-2\mfP_{142} =\left(
3 (\gamma_3^2 \lambda_1 -\gamma_2^2 \lambda_3) (\lambda_1-\lambda_3)-2 \cARS
\right) \gamma_1 \lambda_1,

\\ \noalign{\medskip}

-2\mfP_{34i} \,=
\left( (\gamma_2^2 + \gamma_3^2-1) (\lambda_1-\lambda_3)^2
- (\gamma_1^2+4\gamma_2^2+4\gamma_3^2-2)(\lambda_1-\lambda_3)
\lambda_1 + 2\cARS
\right) \gamma_{i+1} \lambda_1  ,
\end{array}
\] 
\normalsize
for $i=1,2$. 
Now, if $\lambda_1=\lambda_2=\lambda_3$, a direct calculation shows that
\[
(\gamma_1^2+\gamma_2^2+\gamma_3^2) \mfP_{132}
+\gamma_1\mfP_{142} -\gamma_2\mfP_{341}-\gamma_3 \mfP_{342}
= \tfrac{1}{2}(\gamma_1^2+\gamma_2^2+\gamma_3^2)\lambda_1^3  ,
\]
while    if $\lambda_1=\lambda_2\neq\lambda_3$ then  $\mfP_{134}=\mfP_{234}=0$ imply  $\gamma_2=\gamma_3=0$ and 
\[
\gamma_1 \mfP_{132}+\mfP_{142} 
= \tfrac{1}{2}\gamma_1\lambda_1\lambda_3^2  .
\] 
Hence, in any case, we conclude that   $\gamma_1=\gamma_2=\gamma_3=0$ and  the left-invariant metric is the product one.

Finally, for a product metric ($\gamma_1=\gamma_2=\gamma_3=0$) the system $\{\mfP_{ijk}=0\}$ reduces to 
$$
\begin{array}{l}
\phantom{-}
2\mfP_{123}=(3\lambda_3^2-(\lambda_1-\lambda_2)^2-2(\lambda_1+\lambda_2)\lambda_3+2\cARS)\lambda_3 =0,

\\ \noalign{\medskip}

-2\mfP_{132}=(3\lambda_2^2-(\lambda_1-\lambda_3)^2- 2(\lambda_1+\lambda_3)\lambda_2+2\cARS)\lambda_2 =0,

\\ \noalign{\medskip}

\phantom{-}
2\mfP_{231}=(3\lambda_1^2-(\lambda_2-\lambda_3)^2-2(\lambda_2+\lambda_3)\lambda_1+2\cARS)\lambda_1 =0,
\end{array}
$$ 
from where a straightforward calculation shows that $\lambda_1=\lambda_2=\lambda_3$ and hence the metric on $SU(2)$ is that of the round sphere.

\subsection{Direct extensions with Lorentzian Lie groups $\bm{\widetilde{SL}(2,} \pmb{\mathbb{R}}\bm{)}$ or $\bm{SU(2)}$} \label{se:92}
We consider the different possibilities for the Jordan normal form of the structure operator as discussed in Section~\ref{se:Lorentz-Lorentz}.

\subsubsection{\bf Diagonalizable structure operator}\label{sl-Ia}
There exists an orthonormal basis $\{e_i\}$ of the Lie algebra, with $e_3$ timelike,  so that  
\[
\begin{array}{l}
[e_1,e_2]=-\lambda_3 e_3,\qquad
[e_1,e_3]=-\lambda_2 e_2,  \qquad
[e_2,e_3]=\lambda_1 e_1,

\\
\noalign{\medskip}
[e_1,e_4]=\gamma_1\lambda_2 e_2+\gamma_2\lambda_3 e_3,\qquad
[e_2,e_4]=-\gamma_1\lambda_1 e_1+\gamma_3\lambda_3 e_3, 

\\
\noalign{\medskip}

[e_3,e_4]=\gamma_2\lambda_1 e_1+\gamma_3\lambda_2 e_2,
\end{array}
\]
where $\lambda_1\lambda_2\lambda_3\neq 0$  and  $\gamma_1,\gamma_2,\gamma_3\in\mathbb{R}$. The associated Lie algebra corresponds to  $\mathfrak{su}(2)$ if $\varepsilon_1\lambda_1$,  $\varepsilon_2\lambda_2$ and $\varepsilon_3\lambda_3$ have the same sign, or  to $\mathfrak{sl}(2,\mathbb{R})$ otherwise, where $\varepsilon_k=\langle e_k,e_k\rangle$.
Moreover, a left-invariant metric determined by the brackets above is never Einstein.

We proceed as in the previous case.
Considering the polynomials 
\[
\mfP_{124}= -\tfrac{1}{2} (\lambda_1-\lambda_2)^2 \gamma_1 \lambda_3,
\,\,\,
\mfP_{134}= \tfrac{1}{2} (\lambda_1-\lambda_3)^2 \gamma_2 \lambda_2,
\,\,\,
\mfP_{234}= \tfrac{1}{2} (\lambda_2-\lambda_3)^2 \gamma_3 \lambda_1,
\]
it follows that if $\lambda_1$, $\lambda_2$ and $\lambda_3$ are all different then  $\gamma_1=\gamma_2=\gamma_3=0$.
Next we suppose that  at least two of $\lambda_1$, $\lambda_2$ and $\lambda_3$ are  equal.    Although we do  not have the isometries we used in the previous case and we  must analyze additional possibilities, the process is identical to that carried out in the previous section so we omit the details.

First, if $\lambda_1=\lambda_2=\lambda_3$, we get
\[
(\gamma_1^2+\gamma_2^2+\gamma_3^2) \mfP_{132}
+\gamma_1\mfP_{142} +\gamma_2\mfP_{341}+\gamma_3 \mfP_{342}
= -\tfrac{1}{2}(\gamma_1^2+\gamma_2^2+\gamma_3^2)\lambda_1^3  .
\]
Now, if $\lambda_1=\lambda_2\neq \lambda_3$, then $\gamma_2=\gamma_3=0$ and
\[
\gamma_1 \mfP_{132} + \mfP_{142} = -\tfrac{1}{2} \gamma_1\lambda_1\lambda_3^2 .
\]
Next, for $\lambda_1=\lambda_3\neq \lambda_2$ one has $\gamma_1=\gamma_3=0$ and, moreover, 
\[
\gamma_2 \mfP_{123} + \mfP_{143} = -\tfrac{1}{2} \gamma_2\lambda_1\lambda_2^2 .
\]
Finally, if $\lambda_2=\lambda_3\neq\lambda_1$ then $\gamma_1=\gamma_2=0$ and
\[
\gamma_3 \mfP_{123} + \mfP_{243} = -\tfrac{1}{2} \gamma_3\lambda_1^2\lambda_2 .
\]
Therefore, in any case, $\gamma_1=\gamma_2=\gamma_3=0$.

Finally, for a product metric ($\gamma_1=\gamma_2=\gamma_3=0$) one has that $\{\mfP_{ijk}=0\}$ reduces to
	
	$$
	\begin{array}{l}
		\phantom{-}
	2\mfP_{123}=(3\lambda_3^2-(\lambda_1-\lambda_2)^2-2(\lambda_1+\lambda_2)\lambda_3-2\cARS)\lambda_3 =0,
	
	\\
	\noalign{\medskip}
		\phantom{-}2
		\mfP_{132}=(3\lambda_2^2-(\lambda_1-\lambda_3)^2- 2(\lambda_1+\lambda_3)\lambda_2-2\cARS)\lambda_2 =0,
	
	\\
	\noalign{\medskip}
	
	-2\mfP_{231}=(3\lambda_1^2-(\lambda_2-\lambda_3)^2-2(\lambda_2+\lambda_3)\lambda_1-2\cARS)\lambda_1 =0,
	\end{array}
	$$ 
	from where a straightforward calculation shows that $\lambda_1=\lambda_2=\lambda_3$. Hence the metric on $\widetilde{SL}(2,\mathbb{R})$ is that of the round pseudo-sphere $\mathbb{S}^3_1$.

\subsubsection{\bf Structure operator $\bm{L}$ with a complex eigenvalue}\label{sl-complex}
If the structure operator $L$ is of type~Ib then  there exists an orthonormal basis $\{e_i\}$ of the Lie algebra, with $e_3$ timelike,  such that the Lie algebra structure is given by
\[
\begin{array}{l}
[e_1,e_2]=-\beta e_2-\alpha e_3,\quad
[e_1,e_3]=-\alpha e_2 + \beta e_3,\quad
[e_2,e_3]=\lambda e_1, 
\\
\noalign{\medskip}
[e_1,e_4]=(\alpha^2+\beta^2)(\gamma_1 e_2+\gamma_2 e_3), 
\\
\noalign{\medskip}
[e_2,e_4]=-(\gamma_1\alpha-\gamma_2\beta)\lambda e_1 +  \gamma_3 \beta e_2 +  \gamma_3 \alpha e_3, 
\\
\noalign{\medskip}
[e_3,e_4]=(\gamma_2\alpha+\gamma_1\beta)\lambda e_1 + \gamma_3 \alpha e_2 -   \gamma_3 \beta e_3 ,
\end{array}
\]
where    $\beta\lambda\neq 0$ and $\alpha, \gamma_1,\gamma_2,\gamma_3\in\mathbb{R}$.
In this case the three-dimensional unimodular Lie algebra corresponds to $\mathfrak{sl}(2,\mathbb{R})$ and, moreover, these metrics are never Einstein.

Considering the polynomials 
\[ 
\begin{array}{l} 

-2 \mfP_{124} = 
\gamma_1 (\alpha^2 + \beta^2 - \alpha 
\lambda)^2 - (\alpha^2 + \beta^2 - 
\alpha \lambda) (\gamma_1 \lambda - 
2 \gamma_2 \beta) \lambda + \gamma_1 (
\alpha - \lambda) \alpha  \lambda^2  ,

\\ \noalign{\medskip}

\phantom{-}
2 \mfP_{134} = 
\gamma_2 (\alpha^2 + \beta^2 - \alpha 
\lambda)^2 - (\alpha^2 + \beta^2 - 
\alpha \lambda)  (\gamma_2 \lambda + 
2 \gamma_1 \beta) \lambda + \gamma_2   (
\alpha - \lambda) \alpha  \lambda^2 ,

\\ \noalign{\medskip}

\phantom{-}
2 \mfP_{144} = (\gamma_1^2+\gamma_2^2) (\alpha^2+\beta^2) (\alpha^2+\beta^2-\lambda^2)\beta, 
\end{array}
\]
a direct calculation shows that
\[
\gamma_2 (\alpha^2+\beta^2) \mfP_{124}
+\gamma_1 (\alpha^2+\beta^2) \mfP_{134}
+ 2\lambda \mfP_{144} 
=    (\gamma_1^2+\gamma_2^2)(\alpha^2+\beta^2) (\alpha-\lambda) \beta\lambda^2 .
\]
Moreover, note that if $\lambda=\alpha$ then $\mfP_{144} = \tfrac{1}{2} (\gamma_1^2+\gamma_2^2)(\alpha^2+\beta^2)\beta^3$. On the other hand, without any assumptions, one has 
$\mfP_{234}=-2 \gamma_3 \beta^2\lambda$.
Hence, we conclude that necessarily  $\gamma_1=\gamma_2=\gamma_3=0$.
Finally, for a product metric ($\gamma_1=\gamma_2=\gamma_3=0$) one has 
	$$
	\begin{array}{l}
		\phantom{-}
	2\mfP_{122}=(8\alpha^2-4\beta^2-\lambda^2-4\alpha\lambda-2\cARS)\beta,
	
	\\ \noalign{\medskip}
	
	-2\mfP_{123}=12\alpha\beta^2-4\beta^2\lambda+\alpha\lambda^2+2\alpha\cARS,
	
	\\ \noalign{\medskip}
	
	-2\mfP_{231}=(4\beta^2+3\lambda^2-4\alpha\lambda-2\cARS)\lambda,
	\end{array}
	$$ 
	and a straightforward calculation shows that there are no algebraic Ricci soliton metrics in this setting.

\subsubsection{\bf Structure operator of type II}\label{sl-II}
In this case the metric is described by
\[
\begin{array}{l}
[u_1,u_2]= \lambda_2 u_3 ,\qquad
[u_1,u_3]=-\lambda_1 u_1-\ve u_2,\qquad
[u_2,u_3]=\lambda_1 u_2,

\\
\noalign{\medskip}
[u_1,u_4]=\gamma_1\lambda_1 u_1 + \ve\gamma_1 u_2 + \gamma_2\lambda_2 u_3, 
\quad
[u_2,u_4]=-\gamma_1\lambda_1 u_2 + \gamma_3 \lambda_2 u_3, 
\\
\noalign{\medskip}
[u_3,u_4]=-\gamma_3\lambda_1 u_1 - (\gamma_2\lambda_1+\ve\gamma_3)u_2 ,
\end{array}
\]
where   $\varepsilon^2=1$, $\lambda_1\lambda_2\neq 0$,  $\gamma_1,\gamma_2,\gamma_3\in\mathbb{R}$ and
$\{u_i\}$  is a pseudo-orthonormal basis of the Lie algebra with $\langle u_1,u_2\rangle=\langle u_3,u_3\rangle=\langle u_4,u_4\rangle=1$.
The underlying unimodular Lie algebra corresponds to $\mathfrak{sl}(2,\mathbb{R})$.
Moreover, the left-invariant metrics in this family are never Einstein.

Considering the polynomial equations $\{\mfP_{ijk}=0\}$ equivalent to $\mathfrak{D}= \operatorname{Ric} -\cARS\Id$ being a derivation, one has that
\[
\begin{array}{l} 
-2\mfP_{123} =   \big(
( \lambda_1 - 3 \lambda_2 ) (\lambda_1 - \lambda_2)
+ 4  \ve \lambda_1  \gamma_3^2
\\
\noalign{\medskip}

\phantom{-2\mfP_{123} =   \big(}
+ 2  (\gamma_2 (\lambda_1 + 
3 \lambda_2) - \ve \gamma_3 ) \gamma_3 (
\lambda_1 - \lambda_2)
- \lambda_1^2 - 2 \cARS
\big) \lambda_2 ,

\\
\noalign{\medskip}

\phantom{-}
2\mfP_{131} =
(\lambda_1 + \lambda_2)   (\lambda_1 - \lambda_2)  \lambda_1
- 2 \ve   \lambda_1^2 \gamma_3^2
\\
\noalign{\medskip}

\phantom{-2\mfP_{131} =}
- (3 \lambda_1 + \lambda_2) (\gamma_2 \lambda_1 +  \ve \gamma_3) \gamma_3 (\lambda_1 - \lambda_2)
- \lambda_1^3 - 2 \lambda_1 \cARS,
\end{array}
\]
and a direct calculation shows that if $\gamma_3=0$ then
\[
\lambda_1 \mfP_{123} + \lambda_2 \mfP_{131} = 2(\lambda_1-\lambda_2)\lambda_1\lambda_2^2,
\]
while if $\lambda_1-\lambda_2=0$ we get
\[
\mfP_{123} + \mfP_{131} = -3 \ve \gamma_3^2\lambda_1^2.
\]
Moreover, since $\mfP_{234} = -\tfrac{1}{2}\gamma_3 (\lambda_1-\lambda_2)^2 \lambda_1$, it follows that necessarily $\gamma_3=0$ and $\lambda_1=\lambda_2$. Hence, $\mfP_{123} = \frac{1}{2} (\lambda_1^2+2\cARS)\lambda_1$, which implies $\cARS=-\frac{1}{2}\lambda_1^2$.
In this situation,   we calculate
\[
\mfP_{132} =  2 \ve (2\gamma_1^2+1) \lambda_1^2\neq 0,
\]
which shows the non-existence of algebraic Ricci solitons in this setting.

\subsubsection{\bf Structure operator of type III}\label{sl-III}
There exists a pseudo-orthonormal basis $\{u_i\}$ of the Lie algebra, with $\langle u_1,u_2\rangle=\langle u_3,u_3\rangle=\langle u_4,u_4\rangle=1$,  such that the corresponding Lie brackets are determined by 
\[
\begin{array}{l}
[u_1,u_2]= u_1 + \lambda u_3 ,\qquad
[u_1,u_3]=-\lambda u_1 ,\qquad
[u_2,u_3]=\lambda u_2+u_3,

\\
\noalign{\medskip}
[u_1,u_4]=\gamma_1\lambda u_1 + \gamma_2 \lambda^2 u_3,

\\
\noalign{\medskip}

[u_2,u_4]=\gamma_3 u_1 -(\gamma_1-\gamma_2)\lambda u_2 -(\gamma_1-\gamma_2-\gamma_3\lambda)u_3 ,

\\
\noalign{\medskip}

[u_3,u_4]=-\gamma_3\lambda u_1-\gamma_2\lambda^2 u_2-\gamma_2\lambda u_3 ,
\end{array}
\]
where $\lambda\neq 0$ and   $\gamma_1,\gamma_2,\gamma_3\in\mathbb{R}$.  The underlying unimodular Lie algebra corresponds to $\mathfrak{sl}(2,\mathbb{R})$ and, moreover, these metrics are never Einstein.

We proceed  as in the previous cases. Considering the polynomial 
$\mfP_{122} = 3\gamma_2^2 \lambda^4$ one has $\gamma_2=0$ and hence a straightforward calculation shows that
\[
\mfP_{121} = -\tfrac{1}{2}(2\gamma_1^2+3)\lambda^2+\cARS,
\quad
\mfP_{123} = \tfrac{1}{2} (\lambda^2+2\cARS) \lambda.
\]
Thus we have
\[
\lambda \mfP_{121} -\mfP_{123} = -(\gamma_1^2+2)\lambda^3\neq 0
\]
and we conclude that left-invariant metrics in this case do not support any algebraic Ricci soliton.

\subsection{Direct extensions with degenerate Lie groups $\bm{\widetilde{SL}(2,} \pmb{\mathbb{R}}\bm{)}$ or $\bm{SU(2)}$} \label{se:93}

Next we assume that the restriction of the left-invariant metric to $\widetilde{SL}(2,\mathbb{R})$ or $SU(2)$ is degenerate and proceed as in case~(iv) of Section~\ref{se:Lorentz-degenerate}.

Let $u$ be such that $\operatorname{span}\{ u\}$ is degenerate. We consider separately the cases of $\operatorname{ad}_{u}$ having real and complex eigenvalues. While the cases of complex eigenvalues and non-zero real eigenvalues were already considered in \cite{CC}, the possibility of zero eigenvalues (omitted in that reference) gives rise to some left-invariant metrics on $\widetilde{SL}(2,\mathbb{R})\times\mathbb{R}$.

\subsubsection{\bf  $\bm{{\ad_{u}}}$ has complex eigenvalues}\label{se:complex}
In this case there exists a basis  $\{ v_1,v_2,v_3,v_4\}$ of the Lie algebra, with $\langle v_1,v_1\rangle=\langle v_2,v_2\rangle=\langle v_3,v_4\rangle=1$, $\langle v_1,v_2\rangle=\kappa$ with $\kappa^2<1$,   such that the Lie brackets are  given by (up to an isomorphic homothety)
\begin{equation}\label{S-degenerate-complex}
\begin{array}{lll}
\![v_1,v_2]=v_3, &\
\!\![v_1,v_3]=\beta\lambda^2 v_2, &\
\!\![v_1,v_4]= \gamma_1\lambda^2 v_2+\gamma_2 v_3,	
\\
\noalign{\medskip}
\![v_2,v_3]=-\beta v_1, &
\!\![v_2,v_4]= -\gamma_1 v_1 + \gamma_3 v_3, &\
\!\![v_3,v_4]= \gamma_2\beta v_1+\gamma_3\beta\lambda^2 v_2 ,
\end{array}
\end{equation}
where $\beta\lambda\neq 0$ and  $\gamma_1,\gamma_2,\gamma_3\in\mathbb{R}$.
These metrics are realized on $SU(2)\times\mathbb{R}$ if $\beta<0$, while for values of $\beta>0$ we have  left-invariant metrics on   $\widetilde{SL}(2,\mathbb{R})\times\mathbb{R}$. Moreover, these metrics are never Einstein.

We analyze the derivation condition for $\mathfrak{D}=\Ricci-\cARS \Id$ by studying the polynomial equations $\{\mfP_{ijk}=0\}$.
Considering the component $\mfP_{124}=\tfrac{\beta^2\left( (\lambda^2-1)^2+4\kappa^2\lambda^2 \right)}{2(\kappa^2-1)}$ one has 
$\kappa=0$ and $\lambda^2=1$, and hence we obtain
\[
\begin{array}{ll}
\mfP_{121} = -\gamma_3\beta^2, \quad
&
\mfP_{123} = \phantom{-}
\tfrac{1}{2} \left(
(\gamma_2^2+\gamma_3^2)\beta^2 + 4\beta +2\cARS
\right),

\\ \noalign{\medskip}

\mfP_{122} = \phantom{-}
\gamma_2 \beta^2,
&

\mfP_{132} = -\tfrac{1}{2} ( 3\gamma_3^2 \beta^2-2\cARS)\beta .
\end{array}
\]
Now one easily checks that
\[
4\gamma_3 \beta \mfP_{121} - \gamma_2 \beta \mfP_{122} + 2\beta \mfP_{123} - 2\mfP_{132} =
4\beta^2\neq 0,
\]
which shows  that left-invariant metrics in this case do not support any algebraic Ricci soliton.

\subsubsection{\bf $\bm{{\ad_{u}}}$ has non-zero real eigenvalues}\label{se:real}

There exists a basis $\{ v_i\}$ of the Lie algebra, with $\langle v_1,v_1\rangle=\langle v_2,v_2\rangle=\langle v_3,v_4\rangle=1$, $\langle v_1,v_2\rangle=\kappa$ with $\kappa^2<1$, so that   
\[
\begin{array}{lll}

\![v_1,v_2]=v_3, &
\![v_1,v_3]=\lambda v_1, &
\![v_1,v_4]= \gamma_1 v_1+\gamma_2 v_3,	

\\
\noalign{\medskip}

\![v_2,v_3]=-\lambda v_2, &
\![v_2,v_4]= -\gamma_1 v_2 + \gamma_3 v_3, &
\![v_3,v_4]= \gamma_3\lambda v_1+\gamma_2\lambda v_2,
\end{array}
\]
where $\lambda\neq 0$ and  $\gamma_1,\gamma_2,\gamma_3\in\mathbb{R}$. The underlying unimodular Lie algebra corresponds to  
$\widetilde{SL}(2, \mathbb{R})\times\mathbb{R}$. Moreover, a left-invariant metric as above is never Einstein.

In this case, we just need to calculate the polynomial 
$\mfP_{124} = \frac{2\lambda^2}{\kappa^2-1}\neq 0$ to show the non-existence of algebraic Ricci solitons in this case.

	\subsubsection{\bf $\bm{{\ad_{u}}}$ is three-step nilpotent}\label{se:nilpotent}
	As in the previous cases we consider a Jordan basis $\{ u_1,u_2,u_3=u\}$ for $\ad_{u_3}$ so that
	$\ad_{u_3}(u_1)=u_2$, $\ad_{u_3}(u_2)=u_3$. After normalizing $\{ u_1,u_2\}$ and rescaling $u_3$, one has a basis $\{ v_1,v_2,v_3\}$ of $\alg=\mathfrak{sl}(2,\mathbb{R})$ so that
	$[v_1,v_2]=\alpha v_1+\beta v_3$, $[v_1,v_3]=-v_2$, $[v_2,v_3]=\alpha v_3$ for some parameters $\alpha,\beta$ with $\alpha\neq0$. Moreover, the inner product is determined by the non-zero components
	$\langle v_1,v_1\rangle=\langle v_2,v_2\rangle=1$, $\langle v_1,v_2\rangle=\kappa$ with $\kappa^2<1$. Hence the left-invariant metrics on $\widetilde{SL}(2,\mathbb{R})\times\mathbb{R}$ in this case are determined, with respect to a basis  $\{v_i\}$ of the Lie algebra  with $\langle v_1,v_1\rangle=\langle v_2,v_2\rangle=\langle v_3,v_4\rangle=1$ and $\langle v_1,v_2\rangle=\kappa$, by the Lie brackets
	\[
	\begin{array}{lll}
	\![v_1,v_2]=\alpha v_1+\beta v_3, &
	\!\![v_1,v_3]=-v_2, &
	\!\! [v_1,v_4]= \gamma_1 v_1+\gamma_2 v_2+\frac{\beta\gamma_1}{\alpha}v_3,	
	\\
	\noalign{\medskip}
	\![v_2,v_3]=\alpha v_3, &
	\!\! [v_3,v_4]= \gamma_3 v_2-\gamma_1 v_3 , &
	\!\! [v_2,v_4]= -\alpha\gamma_3 v_1 - (\alpha\gamma_2+\beta\gamma_3) v_3,
	\end{array}
	\]
where $\alpha\neq 0$, $\kappa^2\neq 1$ and $\gamma_1,\gamma_2,\gamma_3,\beta\in\mathbb{R}$.
	
A straightforward calculation shows that these metrics are never Einstein nor locally symmetric.
Furthermore, they are never algebraic Ricci solitons since, considering the polynomial $\mfP_{324}$ corresponding to the condition for $\mathfrak{D}=\Ricci-\cARS \Id$ being a derivation, one has that 
$\mfP_{324}=\frac{\alpha}{2(1-\kappa^2)}$, which never vanishes.

\section{Final Remarks}\label{se:final}

\subsection{Left-invariant Lorentzian Ricci solitons}\label{lirs}
A Lorentzian Lie group $(G,\ip)$ is a \emph{left-invariant Ricci soliton} if there is a left-invariant vector field $X$ satisfying the Ricci soliton equation $\mathcal{L}_X\ip +\rho=\cARS \ip$.
	Four-dimensional left-invariant Lorentzian Ricci solitons have been determined in \cite{Maria}. A direct examination  shows that they are plane waves or correspond, in the non-Einstein and not locally symmetric cases, to one of the following cases in \cite[Theorem 1.2]{Maria},	
	with the exception of case $(v)$ below that was originally omitted in \cite{Maria}, and corresponds to metrics in Case 2 of Section~\ref{EE-se:degenerate}.
\begin{enumerate}
	\item[(i)] $G=\mathbb{R}^3\rtimes\mathbb{R}$ with left-invariant metrics given by
	$$
	[e_1,e_4]=\alpha e_1,
	\,\,\,
	[e_2,e_4]= \lambda e_2-e_3,
	\,\,\,
	[e_3,e_4]=e_2 +  \lambda e_3,
	$$
	where $\{e_i\}$ is an orthonormal basis with $e_3$ timelike, $\lambda=\varepsilon(1-\frac{\alpha^2}{2})^{1/2}$ and the parameter
	$0\leq\alpha\leq\sqrt{2}$. If $\alpha=0$ then $\varepsilon=1$, while if    $0<\alpha<\sqrt{2}$   then $\varepsilon^2=1$; in this latter  case,  $\alpha\neq \frac{2}{\sqrt{3}}$ whenever $\varepsilon=-1$. 
	The underlying Lie algebra is $\mathfrak{r}'_{3,-1}\times\mathbb{R}$ if $\alpha=0$, and  $\mathfrak{r}'_{4,-\alpha,-\lambda}$ otherwise.
	Metrics above are steady Ricci solitons and also steady algebraic Ricci solitons. They are $\mathcal{F}[0]$-critical metrics with zero energy. This does not mean that the soliton vector field determined by the derivation $\mathfrak{D}=\operatorname{Ric}$ is left-invariant, but it differs from the left-invariant soliton vector field by a Killing one. 
	
	\item[(ii)]$G=E(1,1)\times\mathbb{R}$ with left-invariant metrics determined by  
	$$
	[u_1,u_4]= \alpha u_1,\quad
	[u_2,u_4]= -\alpha u_2+u_3,\quad
	[u_3,u_4]= u_1, \qquad \alpha>0,
	$$
	where $\{u_i\}$ is a pseudo-orthonormal basis with $\langle u_1,u_2\rangle = \langle u_3,u_3\rangle$ $=$ $\langle u_4,u_4\rangle$ $=$ $1$.
	Metrics in this case are left-invariant steady Ricci solitons which are not algebraic Ricci solitons. Moreover these metrics are $\mathcal{S}$-critical and have $\|\rho\|=\tau=0$. They are Brinkmann waves which are not $pp$-waves since $\operatorname{Ric}^3=0$, but $\operatorname{Ric}^2\neq 0$. 
	
	\item[(iii)] $G=\operatorname{Aff}(\mathbb{R})\times\operatorname{Aff}(\mathbb{R})$  with left-invariant metric determined by   
	$$
	[e_2,e_4]=-[e_1,e_2]=e_2,\quad
	[e_1,e_3]=[e_3,e_4]=\tfrac{1}{2}[e_1,e_4]=e_3,
	$$
	where $\{e_i\}$ is an orthonormal basis with $e_3$ timelike.
	Metrics in this family are left-invariant steady Ricci solitons which are not algebraic ones. Moreover they are not critical for any quadratic curvature functional.
	
	\item[(iv)] $G=\operatorname{Aff}(\mathbb{R})\times\operatorname{Aff}(\mathbb{R})$ 
	with left-invariant metrics given by
	$$
	\begin{array}{l}
	[u_1,u_2]= u_1, \quad
	\hspace*{.215cm}[u_1,u_4]=-2\alpha(\alpha\beta+1) u_1,\quad
	[u_2,u_3]=u_3,
	\\
	\noalign{\medskip}
	[u_2,u_4] = \beta u_1,\quad
	[u_3,u_4]= \alpha u_3,
	\quad \alpha>0,\,\, \text{and}\,\, \alpha\beta\notin\left\{-2,-1,-\tfrac{1}{2}\right\},
	\end{array}
	$$
	where $\{u_i\}$ is a pseudo-orthonormal basis with $\langle u_1,u_2\rangle = \langle u_3,u_3\rangle = \langle u_4,u_4\rangle$ $=$ $1$. 
	Metrics in this family are left-invariant expanding Ricci solitons which are not algebraic ones. Moreover they are $\mathcal{F}[t]$-critical with zero energy for $t=-\frac{2\alpha^2\beta^2+4\alpha\beta+3}{2(3\alpha^2\beta^2+4\alpha\beta+2)}\in(-1,-\frac{1}{4})$.
	
	\item[(v)]	$G=\operatorname{Aff}(\mathbb{R})\times\operatorname{Aff}(\mathbb{R})$ 
	with left-invariant metrics given by 
	$$
	\begin{array}{l}
	[u_1,u_2]= \lambda_1 u_1+\lambda_2u_3, 
	\quad 
	{[u_2,u_3]}= \lambda_1 u_3,
	\quad
	{[u_3,u_4]}=\gamma_4u_3,
	\\
	\noalign{\medskip}
	{[u_1,u_4]}=\gamma_1 u_1+\frac{(\gamma_1-\gamma_4)\lambda_2}{2\lambda_1}u_3,
	\quad
	[u_2,u_4]=-\frac{1}{2}\lambda_2u_1-\frac{3\lambda_2^2+4\gamma_1^2+8\gamma_1\gamma_4}{8\lambda_1}u_3,
	\end{array}
	$$
	where $\lambda_1\neq0$, $\gamma_1+\gamma_4\neq0$ and $\{ u_i\}$ is a pseudo-orthonormal basis with
	$\langle u_1,u_1\rangle = \langle u_2,u_2\rangle = \langle u_3,u_4\rangle$ $=$ $1$.
	They are expanding Ricci solitons but not algebraic ones, although they are $\mathcal{F}[-1]$-critical with zero energy.
\end{enumerate}

\subsubsection{Compact four-dimensional Lorentzian Ricci solitons}	

We use left-invariant Ricci solitons on four-dimensional Lorentzian Lie groups to exhibit new examples of compact Lorentzian steady Ricci solitons on some compact nilmanifolds and solvmanifolds, in addition to the examples provided in \cite{JN}.

	The four-dimensional nilpotent Lie algebras admit cocompact subgroups as well as the solvable Lie group with underlying Lie algebra  $\mathfrak{e}(1,1)\times\mathbb{R}$ (see \cite{Bock}), and so left-invariant metrics pass to the corresponding compact quotients. 
	As a consequence, the left-invariant Ricci solitons in the following remarks induce compact steady Ricci solitons in the corresponding nil and solvmanifolds.

	\begin{remark}\rm\label{re:pwlirs}
	Left-invariant plane waves which are non-Einstein left-invariant steady Ricci solitons  given by
	$[u_1,u_3]=u_2$, $[u_1,u_4]=\gamma_3 u_3$, where $\gamma_3\neq 0$ and  $\{u_i\}$ is a pseudo-orthonormal basis with $\langle u_1,u_2\rangle = \langle u_3,u_3\rangle = \langle u_4,u_4\rangle$ $=$ $1$. (cf. \cite[Theorem 5.3-(i)]{Maria}),  are realized on a nilpotent Lie group with underlying Lie algebra $\mathfrak{n}_4$. Moreover the left-invariant soliton vector field is generated by $\xi=\kappa u_2+\frac{1}{4}\gamma_3^2 u_3$. 
	
	The plane wave left-invariant metrics in \cite[Theorem 5.3-(ii)]{Maria} given by 
	$[u_1,u_2]=u_1$, $[u_2,u_3]=u_3$, 
	$[u_2,u_4]=\gamma_3 u_3$, with $\langle u_1,u_2\rangle = \langle u_3,u_3\rangle = \langle u_4,u_4\rangle$ $=$ $1$,
	are left-invariant steady Ricci solitons with soliton vector field generated by 
	$\xi=-(\frac{1}{4}\gamma_3^2+1)u_1-\gamma_3\kappa u_3+\kappa u_4$. They are realized on the product Lie group $E(1,1)\times\mathbb{R}$.
	
	Moreover observe that while plane waves in \cite[Theorem 5.3-(i)]{Maria} have parallel Ricci tensor, those in \cite[Theorem 5.3-(ii)]{Maria} have not. Therefore they correspond to the different classes (i) and (ii) of four-dimensional homogeneous plane waves discussed in Section~\ref{sse:plane-wave-homog}, respectively.
\end{remark}

\begin{remark}\label{re:pp}
	\rm
	It follows from the analysis in sections~\ref{se:R3}--\ref{se:SLSU} that any pp-wave Lie group which is an algebraic Ricci soliton is a plane wave Lie group. 
	In sharp contrast there exist pp-wave Lie groups 
	which are left-invariant steady Ricci solitons but not plane waves as in   \cite[Theorem 5.1]{Maria}. This is the case of the metric determined on the product Lie algebra $\mathfrak{e}(1,1)\times\mathbb{R}$  by 
	$[u_1,u_4]=\gamma_1 u_1+\varepsilon u_2$, $[u_2,u_4]=-\gamma_1 u_2$, where $\gamma_1\neq 0$ and  $\{ u_i\}$ is a pseudo-orthonormal basis with  $\langle u_1,u_2\rangle=1=\langle u_3,u_3\rangle=\langle u_4,u_4\rangle$. Moreover, these metrics are not $\mathcal{F}[t]$-critical for any quadratic curvature functional, and the soliton vector field, determined by $\xi=\kappa u_3+\gamma_1 u_4$, is locally a gradient.
	
	In addition to the above, there are Brinkmann waves which are left-invariant Ricci solitons but not pp-waves (cf. metrics in Family (ii) in Section~\ref{lirs}) where the underlying Lie algebra $\mathfrak{e}(1,1)\times\mathbb{R}$.
\end{remark}


\begin{remark}\rm
	Three-dimensional left-invariant Lorentzian Ricci solitons on $E(1,1)$ are determined by the Lie algebra structures (see \cite{Israel})
	\begin{enumerate}
		\item $[u_1,u_3]=-\lambda u_1-\varepsilon u_2$,\quad  $[u_2,u_3]=\lambda u_2$,\quad  $\varepsilon=\pm 1$, $\lambda\neq 0$,
		\item $[u_1,u_2]=u_1$,\quad $[u_2, u_3]=u_3$,
	\end{enumerate}
	where $\{ u_i\}$ is a pseudo-orthonormal basis with $\langle u_1,u_2\rangle=\langle u_3,u_3\rangle=1$.
	Moreover, the existence of cocompact subgroups on $E(1,1)$ (see \cite{Bock}) yields three-dimensional examples of compact Lorentzian steady Ricci solitons.
\end{remark}
	
\subsection{Four-dimensional plane wave left-invariant metrics}
Trivially any left-invariant Lorentz metric on the Abelian Lie group is flat. Moreover, left-invariant Lorentz metrics on the Lie group corresponding to the Lie algebra $\mathfrak{r}_{4,1,1}$ are of constant sectional curvature, since they are of type~$\mathfrak{S}$ (see \cite{Milnor, Nomizu}). In opposition to those cases, any other semi-direct extension of the three-dimensional Abelian Lie group admits non-Einstein plane wave metrics (cf. Remark~\ref{grupos-ext-R1} and  Remark~\ref{grupos-ext-R}). Moreover, if was shown in Remark~\ref{grupos-ext-H} that any semi-direct extension of the Heisenberg group admits left-invariant plane wave metrics which are not Einstein. Furthermore, all left-invariant plane waves are steady algebraic Ricci solitons but the plane wave metrics on $\operatorname{Aff}(\mathbb{C})$  corresponding to the family (d) in Section~\ref{sse:plane-wave-homog} (cf. Remark~\ref{re:plane-wave no ARS})). Hence one has
\begin{quote}
	\emph{Any four-dimensional solvable Lie algebra admits Lorentzian inner products so that the corresponding connected and simply connected Lie group is a non-Einstein algebraic Ricci soliton, with the exception of the Abelian Lie algebra, the $\mathfrak{S}$-type Lie algebra $\mathfrak{r}_{4,1,1}$ and the affine algebras $\mathfrak{aff}(\mathbb{R})\times\mathfrak{aff}(\mathbb{R})$ and $\mathfrak{aff}(\mathbb{C})$.}
\end{quote}

Moreover, the affine group $\operatorname{Aff}(\mathbb{R})\times\operatorname{Aff}(\mathbb{R})$ admits left-invariant Ricci soliton metrics as shown in Section~\ref{lirs} (see \cite{Maria}), and the complex affine group $\operatorname{Aff}(\mathbb{C})$ admits steady gradient Ricci soliton metrics since it is a locally conformally flat Cahen-Wallach symmetric space (cf \cite{wafa, Sandra}). On the other hand, the non-solvable products $SU(2)\times\mathbb{R}$ and $\widetilde{SL}(2,\mathbb{R})\times\mathbb{R}$ are algebraic Ricci solitons (indeed rigid Ricci solitons) as shown in Theorem~\ref{Nonsolvable-ARS-thm}. Hence one has that
\begin{quote}
	\emph{Any connected and simply connected four-dimensional Lie group admits left-invariant Lorentz metrics resulting in a Ricci soliton.}
\end{quote}
This shows that the Lorentzian situation is much richer that the Riemannian one (compare with the results in \cite{AL, Lauret2, PW, Wears}).

\end{document}